\documentclass{amsart}
\usepackage{amssymb,amsmath,amsthm,stmaryrd}
\usepackage[spacing]{microtype}
\usepackage[all]{xy}
\usepackage{hyperref}
\usepackage{cancel}
\usepackage{tikz}

\newtheorem{theorem}{Theorem}[section]
\newtheorem{lemma}[theorem]{Lemma}
\newtheorem{proposition}[theorem]{Proposition}
\newtheorem{corollary}[theorem]{Corollary}
\newtheorem{definition}[theorem]{Definition}
\theoremstyle{definition}
\newtheorem{example}[theorem]{Example}
\newtheorem{remark}[theorem]{Remark}

\newcommand{\lie}[1]{\mathfrak{#1}}
\newcommand{\tg}{\mathsf t}
\newcommand{\s}{\mathsf s}
\newcommand{\rr}{\rightrightarrows}

\newcommand{\TT}{{\mathbb{T}}}

\newcommand{\m}{{\mathsf{m}}}
\newcommand{\R}{\mathbb{R}}
\newcommand{\inv}{^{-1}}

\newcommand{\mx}{\mathfrak{X}}
\newcommand{\dr}{\mathbf{d}}
\newcommand{\ldr}[1]{{{\pounds}}_{#1}}
\newcommand{\ip}[1]{{\mathbf{i}}_{#1}}
\newcommand{\an}[1]{\arrowvert_{#1}}
\newcommand{\lb}{\llbracket} 
\newcommand{\rb}{\rrbracket}

\DeclareMathOperator{\Hom}{Hom}

\DeclareMathOperator{\pr}{pr}

\DeclareMathOperator{\id}{id}

\DeclareMathOperator{\graf}{graph}
\DeclareMathOperator{\Exp}{Exp}

\DeclareMathOperator{\Skew}{Skew}

\begin{document}
\title{Dirac groupoids and Dirac bialgebroids}

\author{M. Jotz Lean}
\address{School of Mathematics and Statistics, The University of Sheffield}
\email{M.Jotz-Lean@sheffield.ac.uk}
\thanks{This work was partially supported by Swiss NSF grant 200021-121512 and by a Swiss NSF fellowship for prospective
researchers (PBELP2\_137534) for research conducted
at UC Berkeley, the hospitality of which the author is thankful for. }


\begin{abstract}
  We describe infinitesimally Dirac groupoids via geometric objects
  that we call Dirac bialgebroids.  In the two well-understood special
  cases of Poisson and presymplectic groupoids, the Dirac bialgebroids
  are equivalent to the Lie bialgebroids and IM-$2$-forms,
  respectively. In the case of multiplicative involutive distributions
  on Lie groupoids, we find new properties of infinitesimal ideal
  systems.
\end{abstract}

\maketitle

\tableofcontents

\section{Introduction}
Courant and Weinstein \cite{CoWe88}, and independently Dorfman
\cite{Dorfman87}, introduced the notion of Dirac structure in the
late 1980's. Courant and Weinstein were studying work of Littlejohn on
the ``guiding centre of motion'' of a particle in a magnetic field
\cite{Littlejohn80}. They introduced Dirac structures as a unified
approach to presymplectic and Poisson manifolds.  Dorfman showed Dirac
structures to be useful in the theory of infinite dimensional
Hamiltonian systems. Because the theory of Hamiltonian systems with
constraints can be well formulated using these objects, they were
named after Dirac and his theory of constraints \cite{Dirac67}.
Poisson and presymplectic manifolds and involutive distributions are
the corner cases of Dirac manifolds.  

To be more explicit, let $M$ be
a smooth manifold. Then the direct sum $TM\oplus T^*M$ of vector
bundles over $M$ is endowed with a bracket $\lb\cdot\,,\cdot\rb$ and a
pairing $\langle\cdot\,,\cdot\rangle$ given by
\[\lb(X_1,\theta_1), (X_2,\theta_2)\rb=([X_1,X_2],\ldr{X_1}\theta_2-\ip{X_2}\dr\theta_1)
\]
and
\[\langle (X_1,\theta_1), (X_2,\theta_2)\rangle=\theta_1(X_2)+\theta_2(X_1)
\]
for all vector fields $X_1,X_2\in \mx(M)$ and one-forms
$\theta_1,\theta_2\in\Omega^1(M)$.  A Dirac structure on $M$ is a subbundle
$\mathsf D\subseteq TM\oplus T^*M$ that is maximally isotropic
relative to the pairing, and the sections of which are closed under
the bracket.  The graphs of the vector bundle morphisms
$\pi^\sharp\colon T^*M\to TM$ and $\omega^\flat\colon TM\to T^*M$ associated to a
Poisson bivector $\pi$ and, respectively, a closed $2$-form $\omega$
on $M$ are two examples of Dirac structures.

\medskip

Here we are interested in Dirac groupoids, i.e.~Lie groupoids
with Dirac structures that are compatible with the multiplication.
As in the Lie group
case, a geometric structure on a Lie groupoid, that is compatible with
the multiplication, has a counterpart on the Lie algebroid of the
groupoid.  This counterpart can simply be understood as the
differential of the structure, which is a structure of the same type
but on the Lie algebroid. For instance, the counterpart of a
multiplicative Poisson structure on a Lie groupoid is a Poisson
structure on the Lie algebroid of the groupoid, together with some
compatibility conditions between the Poisson and the Lie algebroid
structures \cite{MaXu00}. More generally, Ortiz has proved in his
thesis that the Lie algebroid of a Dirac groupoid is a Dirac
algebroid, i.e.~a Lie algebroid with a Dirac structure that is
compatible with the Lie algebroid structure (an LA-Dirac structure).
This defines a bijection between multiplicative Dirac structures on a
source simply connected Lie groupoid and LA-Dirac structures on its
Lie algebroid.

On the other hand, we observe in the literature that the infinitesimal
counterpart of a multiplicative geometric object is often described as
a geometric object of completely different, to some extent more
simple, nature.  A multiplicative distribution on a Lie group
corresponds for instance to an ideal in its Lie algebra
\cite{Ortiz08,Jotz11}. Knowing this, it would be rather cumbersome to
infinitesimally describe a multiplicative distribution on a Lie group
as a distribution on the Lie algebra, together with some compatibility
condition in terms of the tangent Lie algebra.

To illustrate this, let us recall how Poisson and presymplectic
groupoids, and involutive multiplicative distributions on Lie
groupoids can be described infinitesimally.  Weinstein introduced
Poisson groupoids as a simultaneous generalisation of Poisson Lie
groups and symplectic groupoids \cite{Weinstein88b}. He proved that a
multiplicative Poisson structure on a Lie groupoid induces a Lie
algebroid structure on the dual of the Lie algebroid of the Lie
groupoid.  Mackenzie and Xu then introduced the notion of \emph{Lie
  bialgebroid} to encode the compatibility condition of this pair of
Lie algebroids in duality \cite{MaXu94}. Then they showed that there
is a one-to-one correspondence between source simply connected Poisson
groupoids and integrable Lie bialgebroids\footnote{A Lie bialgebroid
  $(A,A^*)$ is integrable if $A$ is an integrable Lie algebroid.}
\cite{MaXu00}.  Lie bialgebroids describe hence infinitesimally
Poisson groupoids.  Liu, Weinstein and Xu introduced at the same time
the notion of \emph{Courant algebroid}. With this new notion, a Lie
bialgebroid could be defined as a pair of Lie algebroids which are two
transverse Dirac structures in a Courant algebroid \cite{LiWeXu97}.

Presymplectic groupoids were originally introduced by Xu \cite{Xu04}
under the name ``quasi-symplectic groupoids'' and then used by
Bursztyn, Crainic, Weinstein and Zhu in a work motivated by the
integration of Dirac structures twisted by closed 3-forms
\cite{BuCrWeZh04}. There a presymplectic groupoid is a Lie groupoid
with a multiplicative closed 2-form satisfying some additional
regularity conditions.  Afterwards the terminology slightly
changed and now one often calls presymplectic a Lie groupoid with any
multiplicative closed 2-form.  The infinitesimal description of a
presymplectic groupoid is an IM-2-form, i.e.~a very particular vector
bundle morphism from the Lie algebroid of the groupoid to the
cotangent bundle of the manifold of units \cite{BuCrWeZh04,BuCaOr09}.

Infinitesimal ideal systems, the ``ideals'' in Lie algebroids
\cite{JoOr14}, are the infinitesimal counterpart of multiplicative
distributions \cite{Hawkins08, JoOr14}. Multiplicative distributions
have appeared to be useful as the polarisations in Hawkins' groupoid
approach to quantisation \cite{Hawkins08}, and wide multiplicative
distributions on Lie groupoids were also studied in a modern approach
to Cartan's work on pseudogroups \cite{CrSast12}.

\medskip Lie bialgebroids, IM-2-forms and infinitesimal ideal systems
seem by nature very different, although Poisson groupoids,
presymplectic groupoids and multiplicative distributions on Lie
groupoids are three basic classes of Dirac groupoids.

Our goal is to understand how Lie bialgebroids, IM-$2$-forms and
infinitesimal ideal systems are in fact the corner cases of an
infinitesimal description of Dirac groupoids.  In \cite{Jotz13b} we
constructed a Courant algebroid $\mathsf B$ over $M$ associated to a
Dirac groupoid $(G\rr M, \mathsf D)$, such that the space of units $U$
of the Dirac structure seen as a groupoid is naturally embedded as a
Dirac structure in $\mathsf B$. We classified\footnote{This result
  extends the classification of Poisson homogeneous spaces of Poisson
  groupoids $(G\rr M, \pi)$ via a certain class of Dirac structures in
  the Courant algebroid $A\oplus A^*$ defined by the Lie bialgebroid
  $(A,A^*)$ of the Poisson groupoid \cite{LiWeXu98}.}  the Dirac
homogeneous spaces of $(G\rr M, \mathsf D)$ via Dirac structures in
$\mathsf B$, and we conjectured consequently that the whole
information about $\mathsf D$ was contained in the Manin pair
$(\mathsf B, U)$.  In \cite{Jotz13a} we studied linear splittings of
the double vector bundle $TA\oplus T^*A$ over a Lie algebroid $A$, and
we proved that these correspond to a certain class of \emph{Dorfman
  connections}, which that paper introduced. We studied linear
splittings adapted to LA-Dirac structures (the linearisations of
multiplicative Dirac structures on Lie groupoids), and completely
understood LA-Dirac structures in terms of the corresponding Dorfman
connections, showing how to reconstruct an LA-Dirac structure from
some simple data.

Here we build on the results in those two papers to prove how the
Manin pair $(\mathsf B, U)$ associated to a Dirac groupoid is in fact
equivalent to this Dirac groupoid, and which Manin pairs appear in
this manner as the infinitesimal version of Dirac groupoids.  The
examples that we discuss show that our theorem provides
a unified proof for the results in
 \cite{MaXu00, BuCaOr09}.

\bigskip

Let us summarise our main result in a more detailed manner.  Let $G$
be a Lie groupoid over $M$, with Lie algebroid $A$.  Then the bundle
$TG\oplus T^*G$ has the structure of a Lie groupoid over $TM\oplus
A^*$.  A Dirac structure $\mathsf D$ on $G$ is multiplicative if it is
a subgroupoid of $TG\oplus T^*G$.  The set of units of $\mathsf D$
seen as a groupoid is a subbundle $\iota\colon U \hookrightarrow
TM\oplus A^*$. We proved in \cite{Jotz13b} that $U$ inherits a natural
Lie algebroid structure with anchor the restriction to $U$ of the
projection onto $TM$.  The Lie algebroid $U$ is compatible with the
Lie algebroid $A$ in a sense that this paper identifies. The triple
$(A,U,\iota)$ is then a Dirac bialgebroid, a new geometrical notion
that we now define.

Let $(A, \rho, [\cdot\,,\cdot])$ and $(U,\rho_U,\lb\cdot\,,\cdot\rb)$
be two Lie algebroids over $M$ with an injective morphism $\iota\colon
U\to TM\oplus A^*$ of vector bundles such that
$\rho_U=\pr_{TM}\circ\iota$.  The triple $(A,U,\iota)$ is called a
\emph{Dirac bialgebroid} (over $A$) if there exists a Courant
algebroid $C$ over $M$ such that
\begin{enumerate}
\item $(C,U)$ is a Manin pair \cite{BuIgSe09}, i.e.~$U$ is a Dirac structure in $C$,
\item there exists a (degenerate) Courant morphism\footnote{If $A$ is a Lie
    algebroid over $M$, then the vector bundle $A\oplus T^*M$ inherits
    the structure of a \emph{degenerate Courant algebroid}, see Example~\ref{deg_cou_algebroid}.}
  $\Phi\colon A\oplus T^*M\to C$ such that
\[ \Phi(A\oplus T^*M)+U=C
\]
and $\langle u,\Phi(\tau)\rangle_C=\langle \iota(u),\tau\rangle$
for all $(u,\tau)\in U\times_M(A\oplus T^*M)$.
\end{enumerate}
Two Dirac bialgebroids $(A,U,\iota)$ and $(A,U',\iota')$ over $A$ are
\emph{equivalent} they define the same Lie algebroid
$\iota(U)=\iota'(U')\subseteq TM\oplus A^*$ with the restriction of
$\pr_{TM}$ as anchor.  Our main theorem is the following:
\begin{theorem}
  Let $(G\rr M, \mathsf D)$ be a source simply connected Dirac
  groupoid. Let $A$ be the Lie algebroid of $G\rr M$ and
  $\iota\colon U\hookrightarrow TM\oplus A^*$ the sets of units of $\mathsf
  D$. The triple $(A,U,\iota)$ is a Dirac bialgebroid and the map
  \[ (G, \mathsf D) \mapsto (A,U,\iota)
\] 
defines a bijection between multiplicative Dirac structures on $G\rr
M$ and equivalence classes of Dirac bialgebroids over $A$.
\end{theorem}

\medskip 

Recall for instance that a Lie bialgebroid $(A,A^*)$ is a pair of Lie
algebroids in duality that are two transverse Dirac structures in a
Courant algebroid; which we can hence write $A\oplus A^*$. Set
$\iota:=(\rho_\star,\id_{A^*})\colon A^*\to TM\oplus A^*$, where
$\rho_\star$ is the anchor of the Lie algebroid
$A^*$. Section~\ref{ex_Poisson} shows that $(A,A^*,\iota)$ is a Dirac
bialgebroid, with ambient Courant algebroid $C=A\oplus A^*$.

Take now an IM-2-form $\sigma\colon A\to T^*M$, i.e.~with $\langle
\rho(a),\sigma(b)\rangle=-\langle\rho(b),\sigma(a)\rangle$ and
$\sigma[a,b]=\ldr{\rho(a)}\sigma(b)-\ip{\rho(b)}\dr\sigma(a)$ for all
$a,b\in\Gamma(A)$. Define $\iota=(\id_{TM},\sigma^t)\colon TM\to
TM\oplus A^*$.  Then $(A, TM,\iota)$ is a Dirac bialgebroid with
ambient Courant algebroid $C$ the standard Courant-Dorfman structure
on $TM\oplus T^*M$ (see Section~\ref{ex_presymplectic}).

Finally consider an infinitesimal ideal system $(A,F_M,J,\nabla)$,
i.e.~$F_M\subseteq TM$ is an involutive subbundle, $J\subseteq A$ is a
subbundle such that $\rho(J)\subseteq F_M$ and $\nabla$ is a flat
$F_M$-connection on $A/J$ with the following properties:
\begin{enumerate}
\item If $\bar a\in\Gamma(A/J)$ is $\nabla$-parallel, then
  $[a,j]\in\Gamma(J)$ for all $j\in\Gamma(J)$.
\item If $\bar a, \bar b\in\Gamma(A/J)$ are $\nabla$-parallel, then
  $\overline{[a,b]}$ is also $\nabla$-parallel.
\item If $\bar a\in\Gamma(A/J)$ is $\nabla$-parallel, then
  $\overline{\rho(a)}\in\Gamma(TM/F_M)$ is $\nabla^{F_M}$-parallel,
  where
  \[\nabla^{F_M}\colon \Gamma(F_M)\times\Gamma(TM/F_M)\to\Gamma(TM/F_M)\]
  is the Bott connection associated to $F_M$.
\end{enumerate}
Section~\ref{ex_iis} shows that $(A,F_M,J,\nabla)$ is an
infinitesimal ideal system if and only if $(A,F_M\oplus
J^\circ,\iota)$ is a Dirac bialgebroid with $\iota\colon F_M\oplus
J^\circ\to TM\oplus A^*$ the inclusion, and the Lie algebroid bracket
$[\cdot\,,\cdot]_\nabla$ on $F_M\oplus J^\circ$ defined by
$[(X_1,\alpha_1),(X_2,\alpha_2)]_{\nabla}=([X_1,X_2],\nabla^*_{X_1}\alpha_2-\nabla_{X_2}^*\alpha_1)$.
In order to prove this, we construct a natural Lie algebroid structure
on $\bar A:=(F_M\oplus A)/\graf(-\rho\an{J}:J\to F_M)$, which exists
if and only if the quadruple $(A,F_M,J,\nabla)$ is an infinitesimal
ideal system. We find this result interesting in its own right because
it is a new manner of forming the quotient of an algebroid by an
infinitesimal ideal system, as a generalisation of the quotient of a
Lie algebra by an ideal. The ambient Courant algebroid $C$ is in this example
the Courant algebroid $\bar A\oplus \bar A^*$ defined by the Lie
bialgebroid $(\bar A, \bar A^*)$ with trivial Lie algebroid structure
on $\bar A^*$.

Note that the two families $(A,A^*,(\rho_\star,\id_{A^*}))$ and
$(A,TM,(\id_{TM},\sigma^t))$ of Dirac bialgebroids can be seen as
extreme cases of Dirac bialgebroids, like Poisson and presymplectic
groupoids are the extreme cases of Dirac groupoids. In the first
family, the interesting information is in the pair $(A,A^*)$ and in
the second one, it is rather in the embedding
$(\id_{TM},\sigma^t)\colon TM\to TM\oplus A^*$.

\medskip

Let us finally point out that a Dirac groupoid is an
\emph{LA-groupoid}; the space $\mathsf D$ has a Lie groupoid structure
over $U$ and a Lie algebroid structure over $G$, and the structures
are compatible \cite{Mackenzie92}.  If the Lie algebroid structure is
integrable, the Dirac groupoid should integrate to a presymplectic
double Lie groupoid (``presymplectic'' in the sense of
\cite{BuCrWeZh04}). Hence, this paper gives a class of examples of
Courant algebroids (including the ones defined by Lie bialgebroids)
that could be seen as integrating to ``presymplectic $2$-groupoids'',
if following Mehta and Tang's proposal to integrate Courant algebroids
defined by Lie bialgebroids to ``symplectic $2$-groupoids''
\cite{MeTa11}.  

\subsection*{Methodology and outline of the paper}
First Section~\ref{preliminaries} recalls some background on Lie
groupoids and their tangent and cotangent prolongations, on Courant
algebroids and Dirac structures, on Dorfman connections and dull
algebroids, and on splittings of VB-algebroids and $2$-term
representations up to homotopy.  

Section~\ref{def_Manin_IM} introduces the notions of $A$-Manin
pairs and Dirac bialgebroids, before showing how they are
equivalent. Then Section~\ref{manin} proves that the Manin pair
associated in \cite{Jotz13b} to a Dirac groupoid $(G,
\mathsf D)$ is an $A$-Manin pair, and so that the triple $(A,U,\iota)$
is a Dirac bialgebroid.

Section~\ref{dir_alg} studies LA-Dirac structures on Lie
algebroids. First we summarise the results in \cite{Jotz13a} on linear
splittings of LA-Dirac structures and we describe the two $2$-term
representations up to homotopy describing in a linear splitting the
two VB-algebroid sides.  Then we describe how this data allows us to
construct an $A$-Manin pair associated to each LA-Dirac structure, and
how this defines a bijection between $A$-Manin pairs and LA-Dirac
structures over $A$.

Section~\ref{lie_algebroid} recalls a result of \cite{Ortiz13}
about the one-to-one correspondence between Dirac groupoids and
integrable Dirac algebroids, i.e.~integrable Lie algebroids
with compatible Dirac structures.  Then we describe explicitly the
Lie algebroid $A(\mathsf D)\to U$ of the multiplicative Dirac
structure in terms of the Dorfman connection
$\ldr{}^{U}\colon\Gamma(U)\times\Gamma(U^*)\to\Gamma(U^*)$ that is dual to
the Lie algebroid bracket on sections of $U$.

Section~\ref{conclusion}
finally shows that the Dirac bialgebroid associated to the Dirac
algebroid $(A,\mathsf D_A)$ that is canonically isomorphic to
$(A,A(\mathsf D))$ is the same as the Dirac bialgebroid constructed in
Section~\ref{manin}.  Section~\ref{recipe} gives as a summary a sketch
of the strategy to integrate an integrable Dirac bialgebroid to a
Dirac groupoid. 

 Section~\ref{examples} shows how the notion
of Dirac bialgebroid encompasses the notions of Lie bialgebroids,
IM-2-forms and infinitesimal ideal systems.  \medskip

Let us repeat the different steps that lead us to our main theorem:
\begin{enumerate}
\item[Section~\ref{def_Manin_IM}] Equivalence classes of $A$-Manin
  pairs are in one-to-one correspondence with equivalence classes of
  Dirac bialgebroids.
\item[Section~\ref{manin}] A Dirac groupoid defines an equivalence
  class of $A$-Manin pairs and so of Dirac bialgebroids.
\item[Section~\ref{dir_alg}] There is a bijection
  between LA-Dirac structures and equivalence classes of Dirac
  bialgebroids.
\item[Section~\ref{cristian}] The Lie algebroid of a Dirac groupoid is
  a Dirac algebroid. This defines a bijection between
  source simply connected Dirac groupoids and integrable Dirac
  algebroids \cite{Ortiz13}.
\item[Section~\ref{explicit_computation}] Explicit computation of the
  LA-Dirac structure defined as in
  Section~\ref{cristian} by a multiplicative Dirac structure.
\item[Section~\ref{conclusion}] The Dirac bialgebroids defined by a
  Dirac groupoid and by the corresponding LA-Dirac structure coincide.
\end{enumerate}

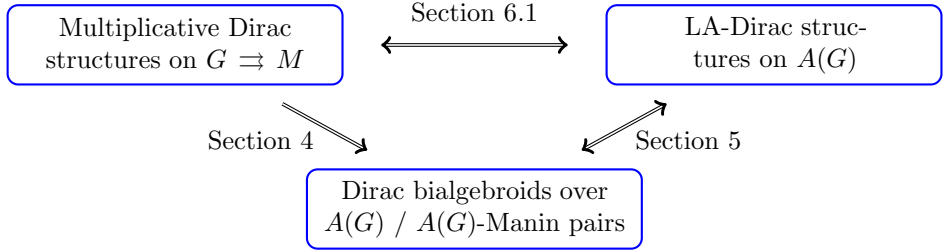
\begin{tikzpicture}[>=to]
\tikzstyle{block} = [rectangle, draw=blue, thick, text width=12em, text centered, rounded corners, minimum height=3em]
\matrix [column sep=-5mm,row sep=11mm]
{\node [block] (Mult) {Multiplicative Dirac structures on $G\rr M$};&&\node [block] (Mor) {LA-Dirac structures on $A(G)$};\\
&\node [block] (Bi) {Dirac bialgebroids over $A(G)$ / $A(G)$-Manin pairs}; &\\
};
\path [draw, double, <->, shorten >=5mm, shorten <=5mm] (Mult) -- node[above=2mm, midway]{Section~\ref{cristian}} (Mor); 
\path [draw,double,->, shorten >=5mm, shorten <=5mm](Mult) -- node [below=2mm, near start]{Section~\ref{manin}\qquad \qquad }(Bi); 
\path [draw, double, <->, shorten >=5mm, shorten <=5mm](Bi) -- node [below=2mm, near end]{\qquad Section~\ref{dir_alg}} (Mor);
\end{tikzpicture}

\subsection*{Notation and conventions}
In the following, $G\rr M$ is always a (Hausdorff) Lie groupoid over
the base $M$. The space of units $M$ is always considered to be a
subset of the space $G$ of arrows. The structure maps of a Lie
groupoid is written as $(\epsilon, \s, \tg, \m, \mathsf i)$, unless
specified otherwise.  The Lie algebroid of a Lie groupoid is
constructed using right-invariant vector fields on the groupoid, and
the anchor is the restriction of $T\tg$ to $T^\s G\an{M}$.  To
simplify the notation, the Lie algebroid $(A(G),
[\cdot\,,\cdot]_{\mathsf A(G)}, \rho_{A(G)})$ of a Lie groupoid $G$ is
always written as $(A, [\cdot\,,\cdot], \rho)$.

Let $M$ be a smooth manifold. We denote by $\mx(M)$ and
$\Omega^1(M)$ the spaces of smooth sections of the tangent and
the cotangent bundle, respectively. For an arbitrary vector bundle
$E\to M$, the space of sections of $E$ is written as
$\Gamma(E)$.
We write in general $q_E\colon E\to M$ for vector bundle projections, except
for $p_M=q_{TM}\colon TM\to M$, $c_M=q_{T^*M}\colon T^*M \to M$ and
$\pi_M=q_{TM\oplus T^*M}\colon TM\oplus T^*M\to M$.

Given a section $\varepsilon$ of $E^*$, we always write
$\ell_\varepsilon\colon E\to \R$ for the linear function associated to
it, i.e.~the function defined by $e_m\mapsto \langle \varepsilon(m),
e_m\rangle$ for all $e_m\in E$.  We write $\phi^t\colon B^*\to
A^*$ for the dual morphism to a morphism $\phi\colon A\to B$ of vector
bundles over the identity, and we write $F^*\omega$ for the
pullback of a form $\omega\in\Omega(N)$ under a smooth map $F\colon
M\to N$ of manifolds.

Let $A$ be a Lie algebroid.  For each $a\in\Gamma(A)$, we have two
derivations over $\rho(a)$:
\begin{equation}\label{ldr_a}
\begin{split}
\ldr{a}\colon \Gamma(A\oplus T^*M)\to\Gamma(A\oplus T^*M),&\quad \ldr{a}(a',\theta)=([a,a'], \ldr{\rho(a)}\theta)\\
\ldr{a}\colon \Gamma(TM\oplus A^*)\to\Gamma(TM\oplus A^*),& \quad\ldr{a}(X,\alpha)=([\rho(a),X], \ldr{a}\alpha).
\end{split}
\end{equation}
Note also that the anchor $\rho$ of a Lie algebroid $A$ defines a
vector bundle morphism $(\rho,\rho^t)\colon A\oplus T^*M\to TM\oplus
A^*$.

\subsection*{Acknowledgements}
The author would like to thank C.~Ortiz and T.~Drummond for
stimulating discussions at the beginning of this project, and K.~Mackenzie for the equations needed
in Section~\ref{ex_Poisson}.

\section{Preliminaries}\label{preliminaries}
This section collects necessary background on the Lie groupoid
structure on the Pontryagin bundle $TG\oplus T^*G$ of a Lie groupoid
$G\rr M$, on Courant algebroids, Dirac structures and Dorfman
connections, and on double vector bundles, VB-algebroids and
representations up to homotopy.
\subsection{Tangent and cotangent Lie groupoids}
Let $G\rr M$ be a Lie groupoid. Applying the tangent functor to each
of the maps defining $G$ yields a Lie groupoid structure on $TG$ with
base $TM$, source $T\s$, target $T\tg$ and multiplication $T\m\colon
T(G\times _MG)\to TG$.  The identity at $v_p\in T_pM$ is
$1_{v_p}=T_p\epsilon v_p$.  This defines the \emph{tangent
  prolongation of $G\rr M$}.

There is also an induced Lie groupoid structure on $T^*G\rr\, A^*=
(TM)^\circ$ \cite{CoDaWe87} (see also \cite{Pradines88}). The source
map $\hat\s\colon T^*G\to A^*$ is given by
\[\hat\s(\alpha_g)\in A_{\s(g)}^*G \text{ for } \alpha_g\in T_g^*G,\qquad 
\hat\s(\alpha_g)(a_{\s(g)})=\alpha_g(T_{\s(g)}L_g(a_{\s(g)}-T_{\s(g)}\tg
a_{\s(g)}))\] for all $a_{\s(g)}\in A_{\s(g)}G$, and the target map
$\hat\tg\colon T^*G\to A^*$ is given by \[\hat\tg(\alpha_g)\in
A_{\tg(g)}^*G, \qquad \hat\tg(\alpha_g)(a_{\tg(g)})
=\alpha_g(T_{\tg(g)}R_g(a_{\tg(g)}))\] for all $a_{\tg(g)}\in
A_{\tg(g)}G$.  If $\hat\s(\alpha_g)=\hat\tg(\alpha_h)$, then the
product $\alpha_g\star\alpha_h$ is defined by
$(\alpha_g\star\alpha_h)(T\m(v_g,v_h))=\alpha_g(v_g)+\alpha_h(v_h)$
for composable pairs $(v_g,v_h)\in T_{(g,h)}(G\times_M G)$.

The fibred product $TG\times_G T^*G$ inherits an induced Lie groupoid
structure over $TM\times_M A^*$.  We write $\mathbb T\tg$ for the
target map $TG\times_G T^*G\to TM\times_M A^*$, and
$\mathbb T\s\colon TG\times_G T^*G\to TM\times_M A^*$ for the source
map. Here again, we write $p_g\star p_h$ for the product of compatible
$p_g,p_h\in TG\oplus T^*G$.

Note that the restriction to $M$ of the kernel of $\TT\s$ can be
identified as follows with $A\oplus T^*M$. The elements of $T_gG\times
T_g^*G$ that are sent by $\TT\s$ to $0_{\s(g)}\in T_{\s(g)}M\times
A^*_{\s(g)}$ are of the form $\left((R_g)_*a_{\tg(g)},
  (T_g\tg)^t\theta_{\tg(g)}\right)$ with $a_{\tg(g)}\in
A_{\tg(g)}=T^\s_{\tg(g)}M$ and $\theta_{\tg(g)}\in T^*_{\tg(g)}M$.  We
identify $\left(a_{m}, (T_m\tg)^t\theta_{m}\right)\in\ker(\TT\s)_m$
with $(a_m, \theta_m)$ in $A_m\times T^*_mM$. Given a section
$\tau=(a,\theta)$ of $A\oplus T^*M$, the \emph{right invariant
  section} $\tau^r\in\Gamma(\ker\TT\s)$ defined by $\tau$ is given
by $\tau^r(g)=\tau_{\tg(g)}\times(0_g,0_g)=(a^r(g),
T_g\tg^t\theta(\tg(g)))\in \ker(\TT\s)_g$ for all $g\in G$.
Furthermore, the restriction of $\TT\tg$ to $\ker(\TT\s)\an{M}$ is,
via the identification above, the map $(\rho,\rho^t)\colon A\oplus
T^*M\to TM\oplus A^*$.

\medskip

\subsection{Courant algebroids and Dirac structures}
A \textbf{Courant algebroid} \cite{LiWeXu97, Roytenberg99} over a manifold $M$ is a vector
bundle $C\to M$ equipped with a fibrewise nondegenerate
symmetric bilinear form $\langle\cdot\,,\cdot\rangle$, a bilinear
bracket $[\cdot\,,\cdot]$ on the smooth sections $\Gamma(C)$,
and a vector bundle map $\rho\colon C\to TM$ over the identity
called the anchor, which satisfy the following conditions\footnote{The following conditions 
\begin{enumerate}
\setcounter{enumi}{3}
\item   $\rho(\lb c_1, c_2\rb) = [\rho(c_1), \rho(c_2)]$,
\item    $\lb c_1,  f c_2\rb =  f \lb c_1 , c_2\rb + (\rho(c_1 ) f )c_2$
\end{enumerate}
are then also satisfied. They are often part of the definition in the
literature, but \cite{Uchino02} observed that they
follow from (1)-(3).  We quickly give here a simple manner to get
(4)-(5) from (1)-(3).  To get (5), replace $c_2$ by $f c_2$ in
(2). Then replace $c_2$ by $f c_2$ in (1) in order to get (4).  }
\begin{enumerate}
\item $\lb c_1, \lb c_2, c_3\rb\rb = \lb\lb c_1, c_2\rb, c_3\rb + \lb c_2, \lb c_1, c_3\rb\rb$,
\item    $\rho(c_1 )\langle c_2, c_3\rangle = \langle\lb c_1, c_2\rb, c_3\rangle 
+ \langle c_2, \lb c_1 , c_3\rb\rangle$,
\item $\lb c_1, c_2\rb +\lb c_2, c_1\rb =\mathcal D\langle c_1 , c_2\rangle$
\end{enumerate}
for all $c_1, c_2, c_3\in\Gamma(C)$ and $ f\in
C^\infty(M)$.
Here, we  use the notation $\mathcal D := \rho^t\circ\dr \colon 
C^\infty(M)\to\Gamma(C)$, using $\langle\cdot\,,\cdot\rangle$
to identify $C$ with $C^*$:
$\langle \mathcal D f, c\rangle=\rho(c)( f)$
for all $ f\in
C^\infty(M)$ and $c\in\Gamma(C)$.

If $C_1$ and $C_2$ are Courant algebroids over the same base $M$, then
a \textbf{Courant morphism} is a vector bundle morphism $\Phi\colon
C_1 \to C_2$ over the identity on $M$, such that
$\rho_2\circ\Phi=\rho_1$, $\langle c,c'\rangle_{C_1}=\langle \Phi(c),
\Phi(c')\rangle_{C_2}$ and $\Phi\lb c, c'\rb_{C_1}=\lb \Phi(c),
\Phi(c')\rb_{C_2}$ for all $c,c'\in\Gamma(C_1)$.

\begin{example}\label{ex_pontryagin}
The direct sum $TM\oplus T^*M$ 
with the projection on $TM$ as anchor map, $\rho=\pr_{TM}$, 
the symmetric bracket 
$\langle\cdot\,,\cdot\rangle$
given by 
\begin{equation}
\langle(v_m,\theta_m), (w_m,\omega_m)\rangle=\theta_m(w_m)+\omega_m(v_m)
\label{sym_bracket}
\end{equation}
for all $m\in M$, $v_m,w_m\in T_mM$ and $\theta_m,\omega_m\in T_m^*M$
and the \textbf{Courant-Dorfman bracket} 
given by 
\begin{equation*}
\lb (X_1,\theta_1), (X_2,\theta_2)\rb=\left([X_1,X_2], \ldr{X_1}\theta_2-\ip{X_2}\dr\theta_1\right)
\end{equation*}
for all $(X_1,\theta_1), (X_2, \theta_2)\in\Gamma(TM\oplus T^*M)$,
yield the standard example of  a Courant algebroid (often called 
the \emph{standard Courant algebroid over $M$}).
The map $\mathcal D\colon  C^\infty(M)\to \Gamma(TM\oplus T^*M)$
is given by $\mathcal D f=(0, \dr f)$.
\end{example}

\begin{example}\label{deg_cou_algebroid}
  Let $G\rr M$ be a Lie groupoid.  The Courant-Dorfman bracket and the
  pairing on sections of $TG\oplus T^*G\to G$ restrict to
  right-invariant sections (defined in the previous section) and give
  rise to the following structure on $A\oplus T^*M$.

The direct sum $A \oplus T^*M\to M$, with the
anchor $\rho\circ\pr_A\colon A\oplus T^*M\to TM$, the pairing $\langle
(a_1,\theta_1), (a_2,\theta_2)\rangle_d=\theta_2(\rho(a_1))+\theta_1(\rho(a_2))$ and the bracket\footnote{``d'' stands for
  degenerate.}
$[\cdot\,,\cdot]_d$ on $\Gamma(A\oplus T^*M)$,
\[ [(a_1,\theta_1), (a_2,\theta_2)]_d
=([a_1,a_2], \ldr{\rho(a_1)}\theta_2-\ip{\rho(a_2)}\dr\theta_1)
\]
is a \emph{degenerate Courant algebroid}. 
That is, $(A\oplus T^*M,
\rho\circ\pr_A, \langle\cdot\,,\cdot\rangle_d,
[\cdot\,,\cdot]_d)$ satisfies the axioms (1)--(5) of
a Courant algebroid, except the non-degeneracy of the pairing (see
\cite{Jotz15}). 
\end{example}

A \textbf{Dirac structure}\cite{CoWe88} $\mathsf D\subseteq C$
is a subbundle satisfying
\begin{enumerate}
\item $\mathsf D^\perp=\mathsf D$ relative to the pairing on $C$,
\item $[\Gamma(\mathsf D), \Gamma(\mathsf D)]\subseteq \Gamma(\mathsf
  D)$.
\end{enumerate}
The rank of the vector bundle $\mathsf D$ is then half the rank of $C$,
and the triple $(\mathsf D\to M, \rho\an{\mathsf D},
[\cdot\,,\cdot]\an{\Gamma(\mathsf D)\times\Gamma(\mathsf D)})$ is a
Lie algebroid over $M$.  Dirac structures appear naturally in several
contexts in geometry and geometric mechanics (see for instance
\cite{Bursztyn11} for an introduction to the geometry and applications
of Dirac structures).  Three classes of great interest in this
paper are collected in the following example.
\begin{example}
  Let $M$ be a smooth manifold and $F\subseteq TM$ an involutive
  subbundle. Then $\mathsf D_F=F\oplus F^\circ\subseteq TM\oplus T^*M$
  is a Dirac structure on $M$.

  Now let $\pi\in\Gamma\left(\bigwedge^2 TM\right)$ be a Poisson
  bivector field.  Then $\mathsf D_\pi\subseteq TM\oplus T^*M$
  defined by
  \[\mathsf{D}_\pi=\operatorname{graph}(\pi^\sharp\colon T^*M\to TM,
  \pi^\sharp(\theta_m)=\pi(\theta_m,\cdot))
\]
is a Dirac structure on $M$. 

Finally let $\omega\in\Omega^2(M)$ be a closed $2$-form on $M$.
Then $\mathsf D_\omega\subseteq TM\oplus T^*M$ defined by
  \[\mathsf{D}_\omega=\operatorname{graph}(\omega^\flat\colon TM\to T^*M,
  \omega^\flat(v_m)=\omega(v_m,\cdot))
\]
is a Dirac structure on $M$. 
\end{example}

\subsection{Dorfman connections and dull algebroids}
In the following, an \emph{anchored} vector bundle is a vector bundle
$Q\to M$ with a vector bundle morphism $\rho_Q\colon Q\to TM$ over the
identity.  An anchored vector bundle $(Q\to M, \rho_Q)$ and a vector
bundle $B\to M$ are said to be \emph{paired} if there exists a
fibrewise pairing $\langle\cdot\,,\cdot\rangle\colon Q\times_M B\to
\R$ and a map $\dr_B\colon C^\infty(M)\to \Gamma(B)$ such that
\begin{equation}\label{compatibility_anchor_pairing}
  \langle q, \dr_B f\rangle=\rho_Q(q)( f)
\end{equation}
for all $q\in\Gamma(Q)$ and $ f\in C^\infty(M)$.  Then $(Q\to M,
\rho)$ and $(B,\dr_B,\langle\cdot\,,\cdot\rangle)$ are said to be
\textbf{paired by $\langle\cdot\,,\cdot\rangle$}.

\begin{definition}\label{the_def}
  Let $(Q\to M,\rho_Q)$ be an anchored vector bundle that is paired
  with $(B\to M, \dr_B, \langle\cdot\,,\cdot\rangle)$.  A
  \textbf{Dorfman ($Q$-)connection on $B$} is an $\R$-linear map
\begin{equation*}
  \Delta\colon \Gamma(Q)\to \operatorname{Der}(B)
\end{equation*} 
such that\footnote{Note that the map $\langle q, \cdot\rangle\cdot
  \dr_B f$ can be seen as a section of $\operatorname{Hom}(B,B)$,
  i.e.~as a derivation over $0\in\mx(M)$.}
\begin{enumerate}
\item $\Delta_q$ is a derivation over $\rho_Q(q)\in\mx(M)$,
\item $\Delta_{ f q}b= f\Delta_qb+\langle q, b\rangle \cdot \dr_B  f$ and 
\item $\Delta_{q}(\dr_B f)=\dr_B(\rho_Q(q) f)$
\end{enumerate}
for all $ f\in C^\infty(M)$, $q,q'\in\Gamma(Q)$, $b\in\Gamma(B)$.
\end{definition}

\begin{remark}
  If the pairing $\langle\cdot\,,\cdot\rangle\colon Q\times_MB\to
  \R$ is nondegenerate, then $B\simeq Q^*$ and the map
  $\dr_{B}=\dr_{Q^*}\colon C^\infty(M)\to\Gamma(Q^*)$ is
  \emph{defined} by \eqref{compatibility_anchor_pairing}: we have then
  $\dr_{Q^*} f=\rho_Q^t\dr f$ for all $ f\in
  C^\infty(M)$.

  The map $\Delta^*\colon \Gamma(Q)\times\Gamma(Q)\to \Gamma(Q)$ that is
  dual to $\Delta$ in the sense of dual derivations, i.e. $\langle
  \Delta^*_{q_1}q_2, \tau\rangle=\rho_Q(q_1)\langle
  q_2,\tau\rangle -\langle q_2, \Delta_{q_1}\tau\rangle$ for all
  $q_1,q_2\in \Gamma(Q)$ and $\tau\in\Gamma(Q^*)$ is then a dull
  bracket on $\Gamma(Q)$ as in the following definition.
\end{remark}

\begin{definition}
A \textbf{dull algebroid} 
is an anchored
 vector bundle  $(Q\to M, \rho_Q)$ 
with a bracket
$[\cdot\,,\cdot]_Q$ on $\Gamma(Q)$ such that
\begin{equation}\label{anchor_preserves_bracket}
\rho_Q[q_1, q_2]_Q=[\rho_Q(q_1),\rho_Q(q_2)]
\end{equation}
and  (the Leibniz identity)
\begin{equation*} [ f_1 q_1,  f_2
  q_2]_Q= f_1 f_2[q_1,
  q_2]_Q+ f_1\rho_Q(q_1)( f_2)q_2- f_2\rho_Q(q_2)( f_1)q_1
\end{equation*}
for all $ f_1, f_2\in C^\infty(M)$, $q_1, q_2\in\Gamma(Q)$.
\end{definition}
That is, a dull algebroid is a \textbf{Lie algebroid} if its bracket
is in addition skew-symmetric and satisfies the Jacobi identity.  The
following examples are crucial in this paper.
\begin{example}
\begin{enumerate}
\item Let $(A\to M, \rho,[\cdot\,,\cdot])$ be a Lie algebroid.
The Dorfman connection that is dual to the bracket on sections of $A$ is the Lie derivative
\[\ldr{}^A:\Gamma(A)\times\Gamma(A^*)\to\Gamma(A^*).
\]

\item Let $C$ be a Courant algebroid and $D$ a Dirac structure in $C$.
  The \textbf{Bott-Dorfman connection}
\[ \Delta^D\colon \Gamma(D)\times\Gamma(C/D)\to\Gamma(C/D),
\]  is defined by
\[\Delta^D_d\bar c=\overline{\lb d, c\rb}, \quad \text{ for } d\in
\Gamma(D), c\in\Gamma(C).
\] Modulo the identification $D^*\simeq C/D$, this Dorfman connection
is the Lie derivative $\Delta^D=\ldr{}^D\colon
\Gamma(D)\times\Gamma(D^*)\to\Gamma(D^*)$ that is dual to the Lie
algebroid structure on $D$.
\end{enumerate}
\end{example}

Let $A\to M$ be a vector bundle. The vector bundle $TM\oplus A^*$ is
always anchored by the projection $\pr_{TM}$ to $TM$, and paired
with $A\oplus T^*M$ via the obvious pairing and the map $\dr_{A\oplus
  T^*M}\colon C^\infty(M)\to\Gamma(A\oplus T^*M)$, $f\mapsto (0,\dr f)$. We
are particularly interested in Dorfman connections
\[\Delta\colon\Gamma(TM\oplus A^*)\times\Gamma(A\oplus T^*M)\to
\Gamma(A\oplus T^*M).
\]
Assume that a subbundle $U\subseteq TM\oplus A^*$ has a
Lie algebroid structure with the anchor $\pr_{TM}$. The Lie algebroid
bracket on $\Gamma(U)$ can be extended to a dull bracket on $TM\oplus
A^*$.  The Dorfman connection $\Delta\colon\Gamma(TM\oplus
A^*)\times\Gamma(A\oplus T^*M)\to \Gamma(A\oplus T^*M)$ that is dual
to this bracket satisfies $\Delta_u\tau\in \Gamma(U^\circ)$ for all
$u\in\Gamma(U)$ and $\tau\in\Gamma(U^\circ)$ and quotients hence to a Dorfman connection
\[\bar\Delta\colon \Gamma(U)\times\Gamma(A\oplus
T^*M/U^\circ)\to\Gamma(A\oplus T^*M/U^\circ).
\]
Note that $U^*\simeq A\oplus T^*M/U^\circ$ and the Dorfman connection
$\bar\Delta$ is equal to
$\ldr{}^U\colon\Gamma(U)\times\Gamma(U^*)\to\Gamma(U^*)$, the Lie
derivative associated to the Lie algebroid $U$.
We say that the Dorfman connection $\Delta$ is an \textbf{extension} of $\ldr{}^U$.

\subsection{VB-algebroids and representations up to homotopy}\label{background_DVB_etc}
We briefly recall the definitions of double vector bundles, of their
\emph{linear} and \emph{core} sections, and of their \emph{linear
  splittings} and \emph{lifts}. We refer to
\cite{Pradines77,Mackenzie05,GrMe10a} for more detailed treatments.
A \textbf{double vector bundle} is a commutative square
\begin{equation*}
\begin{xy}
\xymatrix{
D \ar[r]^{\pi_B}\ar[d]_{\pi_A}& B\ar[d]^{q_B}\\
A\ar[r]_{q_A} & M}
\end{xy}
\end{equation*}
of vector bundles such that $(d_1+_Ad_2)+_B(d_3+_Ad_4)=(d_1+_Bd_3)+_A(d_2+_Bd_4)$
for $d_1,d_2,d_3,d_4\in D$ with $\pi_A(d_1)=\pi_A(d_2)$,
$\pi_A(d_3)=\pi_A(d_4)$ and $\pi_B(d_1)=\pi_B(d_3)$,
$\pi_B(d_2)=\pi_B(d_4)$.
Here, $+_A$ and $+_B$ are the additions in $D\to A$ and $D\to B$,
respectively.
The vector bundles $A$
and $B$ are called the \textbf{side bundles}. The \textbf{core} $C$ of
a double vector bundle is the intersection of the kernels of $\pi_A$
and of $\pi_B$, which has a 
natural vector bundle structure on $C$ over
$M$. The
inclusion $C \hookrightarrow D$ is denoted by
$
C_m \ni c \longmapsto \overline{c} \in \pi_A^{-1}(0^A_m) \cap \pi_B^{-1}(0^B_m).
$

The space of sections $\Gamma_B(D)$ is generated as a
$C^{\infty}(B)$-module by two distinguished classes of sections (see
\cite{Mackenzie11}), the \emph{linear} and the \emph{core sections}
which we now describe.  For a section $c\colon M \rightarrow C$, the
corresponding \textbf{core section} $c^\dagger\colon B \rightarrow D$
is defined as $c^\dagger(b_m) = \tilde{0}_{\vphantom{1}_{b_m}} +_A
\overline{c(m)}$, $m \in M$, $b_m \in B_m$.  If not specified
otherwise we denote the corresponding core section $A\to D$ by
$c^\dagger$ also, relying on the argument to distinguish between
them. The space of core sections of $D$ over $B$ is written
$\Gamma_B^c(D)$.

A section $\xi\in \Gamma_B(D)$ is called \textbf{linear} if $\xi\colon
B \rightarrow D$ is a bundle morphism from $B \rightarrow M$ to $D
\rightarrow A$ over a section $a\in\Gamma(A)$.  The space of linear
sections of $D$ over $B$ is denoted by $\Gamma^\ell_B(D)$.  Given
$\psi\in \Gamma(B^*\otimes C)$, there is a linear section
$\tilde{\psi}\colon B\to D$ over the zero section $0^A\colon M\to A$
given by
$\widetilde{\psi}(b_m) = \tilde{0}_{b_m}+_A \overline{\psi(b_m)}$.
We call $\widetilde{\psi}$ a \textbf{core-linear section}. 

\begin{example}\label{trivial_dvb}
  Let $A, \, B, \, C$ be vector bundles over $M$ and consider
  $D=A\times_M B \times_M C$. With the vector bundle structures
  $D=q^{!}_A(B\oplus C) \to A$ and $D=q_B^{!}(A\oplus C) \to B$, one
  finds that $(D; A, B; M)$ is a double vector bundle called the
  \textbf{decomposed double vector bundle with
    core $C$}. The core sections are given by
$$
c^\dagger\colon b_m \mapsto (0^A_m, b_m, c(m)), \text{ where } m \in M, \, b_m \in
B_m, \, c \in \Gamma(C),
$$
and similarly for $c^\dagger\colon A\to D$.  The space of linear
sections $\Gamma^\ell_B(D)$ is naturally identified with
$\Gamma(A)\oplus \Gamma(B^*\otimes C)$ via
$$
(a, \psi): b_m \mapsto (a(m), b_m, \psi(b_m)), \text{ where } \psi \in
\Gamma(B^*\otimes C), \, a\in \Gamma(A).
$$

In particular, the fibred product $A\times_M B$ is a double vector
bundle over the sides $A$ and $B$, with core $M\times 0$.
\end{example}
A \textbf{linear splitting} of $(D; A, B; M)$ is an injective morphism
of double vector bundles $\Sigma\colon A\times_M B\hookrightarrow D$
over the identity on the sides $A$ and $B$.  Any double vector bundle
admits a linear splitting \cite{GrRo09,Pradines77}.
A linear splitting $\Sigma$ of $D$ is also equivalent to a splitting
$\sigma_A$ of the short exact sequence of $C^\infty(M)$-modules
\begin{equation}\label{fat_seq_gamma}
0 \longrightarrow \Gamma(B^*\otimes C) \hookrightarrow \Gamma^\ell_B(D) 
\longrightarrow \Gamma(A) \longrightarrow 0,
\end{equation}
where the third map is the map that sends a linear section $(\xi,a)$
to its base section $a\in\Gamma(A)$.  The splitting $\sigma_A$ is
called a \textbf{horizontal lift}. Given $\Sigma$, the horizontal lift
$\sigma_A\colon \Gamma(A)\to \Gamma_B^\ell(D)$ is given by
$\sigma_A(a)(b_m)=\Sigma(a(m), b_m)$ for all $a\in\Gamma(A)$ and
$b_m\in B$.  By the symmetry of a linear splitting, we find that a
lift $\sigma_A\colon \Gamma(A)\to\Gamma_B^\ell(D)$ is equivalent to a
lift $\sigma_B\colon \Gamma(B)\to \Gamma_A^\ell(D)$,
$\sigma_B(b)(a(m))=\sigma_A(a)(b(m))$.

\begin{example}\label{td_splitting}
Let $q_E\colon E\to M$ be a vector bundle.  Then the tangent bundle
$TE$ has two vector bundle structures; one as the tangent bundle of
the manifold $E$, and the second as a vector bundle over $TM$. The
structure maps of $TE\to TM$ are the derivatives of the structure maps
of $E\to M$.
\begin{equation*}
\begin{xy}
\xymatrix{
TE \ar[d]_{Tq_E}\ar[r]^{p_E}& E\ar[d]^{q_E}\\
 TM\ar[r]_{p_M}& M}
\end{xy}
\end{equation*} The space $TE$ is a double vector bundle with core bundle
$E \to M$. The map $\bar{}\,\colon E\to p_E^{-1}(0^E)\cap
(Tq_E)^{-1}(0^{TM})$ sends $e_m\in E_m$ to $\bar
e_m=\left.\frac{d}{dt}\right\an{t=0}te_m\in T_{0^E_m}E$.
Hence the core vector field corresponding to $e \in \Gamma(E)$ is the
vertical lift $e^{\uparrow}\colon E \to TE$, i.e.~the vector field with
flow $\phi\colon E\times \R\to E$, $\phi_t(e'_m)=e'_m+te(m)$. An
element of $\Gamma^\ell_E(TE)=\mx^\ell(E)$ is called a \textbf{linear
  vector field}. It is well-known (see e.g.~\cite{Mackenzie05}) that a
linear vector field $\xi\in\mx^l(E)$ covering $X\in\mx(M)$ corresponds
to a derivation $D^*\colon \Gamma(E^*) \to \Gamma(E^*)$ over $X\in
\mx(M)$. The precise correspondence is
given by\footnote{Since its flow is a flow of vector bundle morphisms,
  a linear vector field sends linear functions to linear functions and
  pullbacks to pullbacks.}
\begin{equation}\label{ableitungen}
\xi(\ell_{\varepsilon}) 
= \ell_{D^*(\varepsilon)} \,\,\,\, \text{ and }  \,\,\, \xi(q_E^*f)= q_E^*(X(f))
\end{equation}
for all $\varepsilon\in\Gamma(E^*)$ and $f\in C^\infty(M)$. We write
$\widehat D$ for the linear vector field in $\mx^l(E)$ corresponding
in this manner to a derivation $D$ of $\Gamma(E)$.  The choice of a
linear splitting $\Sigma$ for $(TE; TM, E; M)$ is equivalent to the
choice of a connection on $E$: we can define $\nabla\colon
\mx(M)\times\Gamma(E)\to \Gamma(E)$ by
$\sigma_{TM}(X)=\widehat{\nabla_X}$ for all $X\in\mx(M)$. Conversely,
a connection $\nabla\colon \mx(M)\times\Gamma(E)\to\Gamma(E)$ defines
a lift $\sigma_{TM}^\nabla\colon\mx(M)\to\mx^l(E)$ and a linear
splitting $\Sigma^\nabla\colon TM\times_M E \to TE$.  Given $\nabla$,
it is easy to see using the equalities in \eqref{ableitungen} that the
Lie bracket of vector fields on $E$ is given by $
\left[\sigma_{TM}^\nabla(X),
  \sigma_{TM}^\nabla(Y)\right]=\sigma_{TM}^\nabla[X,Y]-R_\nabla(X,Y)^\uparrow$,
$\left[\sigma_{TM}^\nabla(X), e^\uparrow\right]=(\nabla_Xe)^\uparrow$,
$\left[e_1^\uparrow,e_2^\uparrow\right]=0$, for all $X,Y\in\mx(M)$ and
$e,e_1,e_2\in\Gamma(E)$.
\end{example}

A double vector bundle $(D;A,B;M)$ is a \textbf{VB-algebroid}
(\cite{Mackenzie98x}; see also \cite{GrMe10a}) if there are Lie
algebroid structures on $D\to B$ and $A\to M$, such that the anchor
$\Theta\colon D \to TB$ is a morphism of double vector bundles over
$\rho_A\colon A \to TM$ on one side
and if the Lie bracket is linear:
\begin{equation*} [\Gamma^\ell_B(D), \Gamma^\ell_B(D)] \subset
  \Gamma^\ell_B(D), \qquad [\Gamma^\ell_B(D), \Gamma^c_B(D)] \subset
  \Gamma^c_B(D), \qquad [\Gamma^c_B(D), \Gamma^c_B(D)]= 0.
\end{equation*}
The vector bundle $A\to M$ is then also a Lie algebroid, with anchor
$\rho_A$ and bracket defined as follows: if $\xi_1,
\xi_2\in\Gamma^\ell_B(D)$ are linear over $a_1,a_2\in\Gamma(A)$, then
the bracket $[\xi_1,\xi_2]$ is linear over $[a_1,a_2]$.

\begin{example}\label{example_TAT*A}
  Consider a Lie algebroid $A$.  For $a\in\Gamma(A)$, we have two
  particular types of sections of $TA\to TM$: the linear section
  $Ta\colon TM\to TA$, and the core section $a^\dagger\colon TM\to
  TA$,
  $a^\dagger(x_m)=T_m0^Ax_m+_{p_A}\left.\frac{d}{dt}\right\an{t=0}t\cdot
  a(m)$.  The identities $[Ta, Tb]=T[a,b]$, $[Ta,
  b^\dagger]=[a,b]^\dagger$, $[a^\dagger, b^\dagger]=0$,
  $\rho_{TA}(Ta) =\widehat{[\rho(a),\cdot]}\in\mx^l(TM)$ and
  $\rho_{TA}(a^\dagger)= (\rho(a))^\uparrow\in\mx^c(TM)$ define a
  VB-algebroid structure $(TA\to TM, A\to M)$.

  The cotangent space $(T^*A\to A^*, A\to M)$ also is a
  VB-algebroid. The projection $r_A\colon T^*A\to A^*$ is given by
  $r_A(\alpha_{a_m})(b_m)=\alpha_{a_m}\left(\left.\frac{d}{dt}\right\an{t=0}a_m+tb_m
  \right)$ for all $b_m\in A_m$.  For $\theta\in\Omega^1(M)$ we have
  $\theta^\dagger\in\Gamma_{A^*}^c(T^*A)$,
  $\theta^\dagger(\alpha(m))=\dr_{0^A_m}\ell_\alpha-q_{A}^*\theta(0^A_m)$
  and for $a\in\Gamma(A)$, the section $a^R\in\Gamma^l_{A^*}(T^*A)$ is
  defined by $a^R(\alpha(m))=
  \dr_{a(m)}(\ell_\alpha-q_A^*\langle\alpha, a\rangle)$ for $\alpha\in
  \Gamma(A^*)$.  The Lie algebroid structure on $T^*A\to A^*$ is given
  by $[a_1^R, a_2^R]=[a_1,a_2]^R$, $[a^R,
  \theta^\dagger]=(\ldr{\rho(a)}\theta)^\dagger$, $[\theta_1^\dagger,
  \theta_2^\dagger]=0$,
  $\rho_{T^*A}(a^R)=\widehat{\ldr{a}}\in\mx(A^*)$ and
  $\rho_{T^*A}(\theta^\dagger)=(\rho^*\theta)^\uparrow \in\mx(A^*)$
  for $a,a_i\in\Gamma(A)$ and $\theta,\theta_i\in\Omega^1(M)$.
\medskip

The space
  $TA\times_A T^*A$ has the structure of a double vector bundle:
\begin{equation*}
\begin{xy}
\xymatrix{
TA\oplus T^*A\ar[r]^{\Pi_A}\ar[d]_{c_A}& TM\oplus A^*\ar[d]^{p_M}\\
 A\ar[r]_{q}&M
}
\end{xy}
\end{equation*}
The vertical vector bundle is the Pontryagin bundle $TA\oplus T^*A$ of
$A$ seen as a manifold, and the horizontal projection is defined by
$\Pi_A(v_{a_m}, \alpha_{a_m})=(q_*(v_{a_m}), r_A(\alpha_{a_m}))$.

The Lie algebroid $TA\oplus T^*A\to TM\oplus A^*$ is defined as the
pullback to the diagonals $\Delta_A\to \Delta_M$ of the Lie algebroid
$TA\times T^*A\to TM\times A^*$.  We have the special linear sections
$a^l:=(Ta, a^R)\colon TM\oplus A^*\to TA\oplus T^*A$ for $a\in
\Gamma(A)$ and $(b,\theta)^\dagger:=(b^\dagger, \theta^\dagger)\colon
TM\oplus A^*\to TA\oplus T^*A$ for $(b,\theta)\in\Gamma(A\oplus
T^*M)$.  We write $\Theta\colon TA\oplus T^*A\to T(TM\oplus A^*)$ for the
anchor of $TA\oplus T^*A\to TM\oplus A^*$.
The Lie algebroid $(TA\oplus T^*A, \Theta, [\cdot\,,\cdot])$ is described by the following identities
\begin{equation*}
\begin{split}
[a_1^l, a_2^l]&=[a_1,a_2]^l, \quad [a^l, \tau^\dagger]=(\ldr{a}\tau)^\dagger, \quad [\tau_1^\dagger, \tau_2^\dagger]=0\\
\Theta(a^l)&=\widehat{\ldr{a}}, \quad \Theta(\tau^\dagger)=((\rho,\rho^t)\tau)^\uparrow. 
\end{split}
\end{equation*}
for $a,a_1,a_2\in\Gamma(A)$, $\tau,\tau_1,\tau_2\in\Gamma(A\oplus
T^*M)$. 
\end{example}

Let $E_0,E_1$ be two vector bundles over the same base $M$ as $A$, and
$\partial\colon E_0\to E_1$ a vector bundle morphism. A
\textbf{$2$-term representation up to homotopy of $A$ on
  $\partial\colon E_0\to E_1$} \cite{ArCr12,GrMe10a} is the collection
of two $A$-connections, $\nabla^0$ and $\nabla^1$ on $E_0$
  and $E_1$, respectively, such that $\partial \circ \nabla^0 =
  \nabla^1 \circ \partial$, and an element $R \in \Omega^2(A,
  \Hom(E_1, E_0))$ such that $R_{\nabla^0} = R\circ \partial$,
  $R_{\nabla^1}=\partial \circ R$ and $\dr_{\nabla^{\Hom}}R=0$,
  where $\nabla^{\Hom}$ is the connection induced on $\Hom(E_1,E_0)$
  by $\nabla^0$ and $\nabla^1$ and the operator $\dr_{\nabla^{\Hom}}$ on $\Omega^\bullet(A,\Hom(E_1,E_0))$ is
given by the Koszul formula:
\begin{equation*}
\begin{split}
  \dr_{\nabla^{\Hom}}\omega(a_1,\ldots,a_{k+1})=
&\sum_{i<j}(-1)^{i+j}\omega([a_i,a_j],a_1,\ldots,\hat a_i,\ldots,\hat a_j,\ldots, a_{k+1})\\
  &\qquad +\sum_i(-1)^{i+1}\nabla^{\Hom}_{a_i}(\omega(a_1,\ldots,\hat
  a_i,\ldots,a_{k+1}))
\end{split}
\end{equation*}
for all $\omega\in\Omega^k(A,\Hom(E_1,E_0))$ and $a_1,\ldots,a_{k+1}\in\Gamma(A)$.

\medskip

Consider again a VB-algebroid $(D\to B, A\to M)$. The anchor $\Theta(c^\dagger)$ of a core
section $c^\dagger\in\Gamma_B^c(D)$ is given by
$\Theta(c^\dagger)=(\partial_B c)^\uparrow$, defining a vector bundle
morphism $\partial_B\colon C\to B$. Choose a linear
splitting $\Sigma\colon A\times_MB\to D$. Since the anchor $\Theta$ of a linear
section is linear, for each $a\in \Gamma(A)$ the vector field
$\Theta(\sigma_A(a))\in\mx^l(B)$ defines a derivation of $\Gamma(B)$ with
symbol $\rho(a)$. This defines a linear
connection $\nabla^{B}\colon \Gamma(A)\times\Gamma(B)\to\Gamma(B)$:
$\Theta(\sigma_A(a))=\widehat{\nabla_a^{B}}$
for all $a\in\Gamma(A)$.    Since the bracket of a linear
section with a core section is again a core section, we find a linear connection
$\nabla^{C}\colon\Gamma(A)\times\Gamma(C)\to\Gamma(C)$ such
that $[\sigma_A(a),c^\dagger]=(\nabla_a^{C}c)^\dagger$ for all
$c\in\Gamma(C)$ and $a\in\Gamma(A)$.  The difference
$\sigma_A[a_1,a_2]-[\sigma_A(a_1), \sigma_A(a_2)]$ is a core-linear
section for all $a_1,a_2\in\Gamma(A)$.  This defines a vector valued
form $R\in\Omega^2(A,\operatorname{Hom}(B,C))$ by
$[\sigma_A(a_1), \sigma_A(a_2)]=\sigma_A[a_1,a_2]-\widetilde{R(a_1,a_2)}$,
for all $a_1,a_2\in\Gamma(A)$. For more details on these
constructions, see \cite{GrMe10a}, where the following result is
proved.
\begin{theorem}\label{rajan}
  Let $(D \to B; A \to M)$ be a VB-algebroid and choose a linear
  splitting $\Sigma\colon A\times_MB\to D$.  The triple
  $(\nabla^{B},\nabla^{C},R)$ defined as above is a
  $2$-term representation up to homotopy of $A$ on the complex $\partial_B\colon C\to B$.

  Conversely, let $(D;A,B;M)$ be a double vector bundle with core $C$
  such that $A$ has a Lie algebroid structure, and choose a linear
  splitting $\Sigma\colon A\times_MB\to D$. Then if
  $(\nabla^{B},\nabla^{C},R)$ is a $2$-term representation up to
  homotopy of $A$ on a complex $\partial_B\colon C\to B$, then the
  equations above define a VB-algebroid structure on $(D\to B; A\to
  M)$.

\end{theorem}

\begin{example}\label{tangent_double_2_reps}
  Let $E\to M$ be a vector bundle. By Example~\ref{td_splitting} the
  tangent double $TE$ has a VB-algebroid structure $(TE\to E, TM\to
  M)$. Consider a linear splitting $\Sigma\colon E\times_M TM\to TE$
  and the corresponding linear connection $\nabla\colon
  \mx(M)\times\Gamma(E)\to\Gamma(E)$.  The representation up to homotopy
  corresponding to this splitting is given by $\partial_E=\id_E\colon
  E\to E$, $(\nabla,\nabla,R_\nabla)$.
\end{example}
\begin{example} Let $A$ be a Lie algebroid and consider the
  VB-algebroid $(TA\to TM,A\to M)$; see Example~\ref{example_TAT*A}.
  A linear connection $\nabla\colon\mx(M)\times\Gamma(A)\to\Gamma(A)$
  defines a horizontal lift
  $\sigma_{A}\colon\Gamma(A)\to\Gamma_{TM}^l(TA)$.  The corresponding
  $2$-term representation up to homotopy is on
  $\partial_{TM}=\rho\colon A\to TM$ and given by $(\nabla^{\rm bas}, \nabla^{\rm
    bas}, R_\nabla^{\rm bas})$, where $\nabla^{\rm bas}\colon
  \Gamma(A)\times\Gamma(A)\to\Gamma(A)$ and $\nabla^{\rm bas}\colon
  \Gamma(A)\times\mx(M)\to\mx(M)$ are the basic connections associated
  to $\nabla$:
\begin{align*}
\nabla_a^{\rm bas}a'&=[a,a']+\nabla_{\rho(a')}a, \qquad \nabla_a^{\rm bas}X=[\rho(a),X]+\rho(\nabla_{X}a)
\end{align*}
for all $a,a'\in\Gamma(A)$ and $X\in\mx(M)$ \cite{Fernandes02,ArCr12}. The tensor
$R_{\nabla}^{\rm bas}\in\Omega^2(A,\operatorname{Hom}(TM,A))$ is
defined by
\[R_{\nabla}^{\rm bas}(a_1,a_2)X=-\nabla_X [a_1, a_2] +
\left[\nabla_X a_1, a_2\right] +\left[a_1, \nabla_X a_2\right] -
\nabla_{\nabla_{a_1}^{\rm bas} X } \,a_2 + \nabla_{\nabla_{a_2}^{\rm
    bas} X }\, a_1
\]
for all $a_1,a_2\in\Gamma(A)$ and $X\in\mx(M)$.
\end{example}

\begin{example}
  We now consider the VB-algebroid $(TA\oplus T^*A\to TM\oplus A^*,
  A\to M)$ with core $A\oplus T^*M$.  
 Dorfman connections\footnote{$TM\oplus A^*$ is always anchored by
    $\pr_{TM}$.} $\Delta\colon \Gamma(TM\oplus
  A^*)\times\Gamma(A\oplus T^*M)\to\Gamma(A\oplus T^*M)$ are equivalent to linear splittings
$\Sigma^\Delta\colon (TM\oplus A^*)\times_M A\to TA\oplus T^*A$;
\[\Sigma^\Delta(X,\alpha)(a_m)=\left(T_ma X(m), \dr
  \ell_\alpha(a_m)\right)-\Delta_{(X,\alpha)}(a,0)^\uparrow(a_m),
\] where for $b\in
\Gamma(A)$ and $\theta\in \Omega^1(M)$, the pair
$(b,\theta)^\uparrow(a_m)\in T_{a_m}A\times T^*_{a_m}A$ is defined by
$(b,\theta)^\uparrow(a_m)=\left(b^\uparrow(a_m),
  (q_A^*\theta)(a_m)\right)$ \cite{Jotz13a}.

The horizontal lift
$\sigma^\Delta_A\colon\Gamma(A)\to\Gamma_{TM\oplus A^*}^l(TA\oplus _AT^*A)$ is given by 
$\sigma^\Delta_A (a)(v_m,\alpha_m)=(T_mav_m,\dr_{a_m}
\ell_\alpha)-\Delta_{(X,\alpha)}(a,0)^\uparrow(a_m)$
for any choice of section $(X,\alpha)\in\Gamma(TM\oplus A^*)$ such that
$(X,\alpha)(m)=(v_m,\alpha_m)$, or in other words by 
\[\sigma^\Delta_A (a)=(Ta, a^R)-\widetilde{\Omega_\cdot
a}=a^l-\widetilde{\Omega_\cdot a}
\]
for all $a\in \Gamma(A)$, where
$\Omega\colon \Gamma(TM\oplus A^*)\times\Gamma(A)\to\Gamma(A\oplus T^*M)$ is defined by
\begin{equation}\label{omega}
\Omega_{(X,\alpha)}a=\Delta_{(X,\alpha)}(a,0)-(0,\dr\langle\alpha, a\rangle).
\end{equation}

The two maps $\nabla^{\rm bas}\colon \Gamma(A)\times\Gamma(TM\oplus
A^*)\to \Gamma(TM\oplus A^*)$,
\begin{align}\label{nabla_bas}
  \nabla^{\rm
    bas}_a(X,\alpha)
  &=(\rho,\rho^t)(\Omega_{(X,\alpha)}a)+\ldr{a}(X,\alpha)
\end{align}
and $\nabla^{\rm bas}\colon  \Gamma(A)\times\Gamma(A\oplus T^*M)\to \Gamma(A\oplus T^*M)$,
\begin{align*}\label{nabla_bas1}
\nabla^{\rm
  bas}_a(a',\theta)
&=\Omega_{(\rho,\rho^t)(a',\theta)}a+\ldr{a}(a',\theta)
\end{align*}
are ordinary linear $A$-connections. The formula
\begin{multline*}
  R_\Delta^{\rm bas}(a_1,a_2)(X,\alpha)=-\Omega_{(X,\alpha)}[a_1,a_2]
  +\ldr{a_1}\left(\Omega_{(X,\alpha)}a_2\right)-\ldr{a_2}\left(\Omega_{(X,\alpha)}a_1\right)\\
  + \Omega_{\nabla^{\rm bas}_{a_2}(X,\alpha)}a_1-\Omega_{\nabla^{\rm
      bas}_{a_1}(X,\alpha)}a_2
\end{multline*}
defines a tensor $R_\Delta^{\rm
  bas}\in\Omega^2(A,\operatorname{Hom}(TM\oplus A^*,A\oplus T^*M))$.
We prove in \cite{Jotz13a} that the $2$-term representation up to
homotopy defined by the VB-algebroid $(TA\oplus T^*A\to TM\oplus A^*,
A\to M)$ and the splitting $\Sigma$ is the $2$-term representation of
$A$ on $(\rho,\rho^t)\colon A\oplus T^*M \to TM\oplus A^*$ given by
$(\nabla^{\rm bas},\nabla^{\rm bas},R_\Delta^{\rm bas})$.  Compare
this with the $2$-term representation up to homotopy given by a linear
splitting of $(TA\to TM, A\to M)$ as in Example
\ref{tangent_double_2_reps}.
\end{example}

\section{$A$-Manin pairs and Dirac bialgebroids}\label{def_Manin_IM}

A Manin pair over a manifold $M$ is a pair $(C,U)$ of vector bundles
over $M$, where $C$ has the structure of a Courant algebroid and $U$
is a Dirac structure in $C$ \cite{BuIgSe09}.  
\begin{definition}\label{def_A_Manin}
  Let $(A\to M, \rho, [\cdot\,,\cdot])$ be a Lie algebroid.
\begin{enumerate}
\item An \textbf{$A$-Manin pair over $M$} is a Manin pair $(C,U)$ over $M$, together
  with
  \begin{enumerate}
  \item an injective morphism $\iota\colon U\to TM\oplus A^*$ of
    vector bundles, such that $\rho_U:=\rho_C\an{U}=\pr_{TM}\circ
    \iota$ and 
  \item a (degenerate) Courant morphism $\Phi\colon
    A\oplus T^*M\to C$ such that
\[ \Phi(A\oplus T^*M)+U=C
\]
and $\langle u,\Phi(\tau)\rangle_C=\langle \iota(u),\tau\rangle$ for
all $(u,\tau)\in U\times_M(A\oplus T^*M)$.
\end{enumerate}
\item Let $(U\to M, \rho_U,[\cdot\,,\cdot]_U)$ be a Lie algebroid and
  $\iota\colon U\to TM\oplus A^*$ an injective vector bundle morphism
  that is compatible with the anchors: $\pr_{TM}\circ \iota=\rho_U$.
  The triple $(A,U,\iota)$ is a \textbf{Dirac bialgebroid} (over $A$)
  if $U$ is the Dirac structure and $\iota$ is the injective morphism
  of an $A$-Manin pair.
\item Two $A$-Manin pairs $(C,U)$ and $(C',U')$, and respectively two Dirac bialgebroids
  $(A,U,\iota)$ and $(A,U',\iota')$, are \textbf{equivalent} if
  they define the same Lie algebroid $\iota(U)=\iota'(U')$ in the
  $\pr_{TM}$-anchored vector bundle $TM\oplus A^*$.
\end{enumerate}
\end{definition}

By definition, an $A$-Manin pair determines an unique Dirac
bialgebroid $(A,U,\iota)$.  Conversely, we show that the equivalence
class of an $A$-Manin pair $(C,U)$ with structure maps $\iota,\Phi$ is
completely determined by the equivalence class of the corresponding
Dirac bialgebroid $(A,U,\iota)$.  More precisely, given $A$, $U$
and~$\iota$, we can reconstruct $C$ and $\Phi$ up to
isomorphism.  Consider a Dirac bialgebroid $(A,U,\iota)$ and identify
$U$ with $\iota(U)\subseteq A\oplus T^*M$.  If
$\tau\in U^\circ\subseteq TM\oplus A^*$, then $\Phi(\tau)$ satisfies
\[\langle u, \Phi(\tau)\rangle_C=\langle\tau,u\rangle=0
\]
for all $u\in U$. Since $U$ is a Dirac structure, we find that $\Phi$
restricts to a map $U^\circ\to U$. (Conversely, we find easily that
$\Phi(\tau)\in U$ if and only if $\tau\in U^\circ$.)  Next choose
$\tau_1\in U^\circ $ and $\tau_2\in A\oplus T^*M$.  Then \[\langle
\Phi(\tau_1), \tau_2\rangle=\langle
\Phi(\tau_1),\Phi(\tau_2)\rangle_C= \langle
\tau_1,\tau_2\rangle_d=\langle (\rho,\rho^t)\tau_1, \tau_2\rangle,
\] 
which shows that $\Phi\an{U^\circ}=(\rho,\rho^t)\an{U^\circ}\colon
U^\circ\to U$. In particular, $(\rho,\rho^t)$ sends $U^\circ$ to $U$,
and $U^\circ$ is isotropic in $A\oplus T^*M$.  Consider the vector
bundle map $U\oplus A\oplus T^*M\to C$, $(u,\tau)\mapsto
u+\Phi(\tau)$. By hypothesis, this map is surjective. Its kernel is
the set of pairs $(u,\tau)$ with $u=-\Phi(\tau)$, i.e.~the graph of
$-(\rho,\rho^t)\an{U^\circ}\colon U^\circ\to U$. It follows that $C$
can be identified with
\begin{equation}\label{C_ident}
\frac{U\oplus A\oplus T^*M}{\operatorname{graph}(-(\rho,\rho^t)\an{U^\circ}\colon
U^\circ\to U)}.
\end{equation}
We use the notation $u\oplus \tau$ for
$u+\Phi(\tau)\in C$. The anchor of $C$ is then $c\colon
C\to TM$, $c(u\oplus\tau)=\pr_{TM}(u)+\rho\circ\pr_A(\tau)$, and the
bracket is given by $\langle u_1\oplus\tau_1,
u_2\oplus\tau_2\rangle_C=\langle u_1+\Phi(\tau_1),
u_2+\Phi(\tau_2)\rangle_C= \langle u_1,\tau_2\rangle+\langle u_2,
\tau_1\rangle +\langle \tau_1, (\rho,\rho^t)\tau_2\rangle$.
The map $\mathcal D\colon C^\infty(M)\to\Gamma(C)$
is given by 
$\langle\langle u\oplus\tau, \mathcal D  f\rangle\rangle_C=
(\pr_{TM}(u)+\rho\circ\pr_A(\tau)) f$
for all $u\oplus \tau\in\Gamma(C)$, 
i.e. $\mathcal D f=0\oplus(0,\dr f)$.

We show that the Courant algebroid bracket
is uniquely determined by the Lie algebroid structure on $U$.  Take an
extension $\Delta\colon\Gamma(TM\oplus A^*)\times\Gamma(A\oplus
T^*M)\to\Gamma(A\oplus T^*M)$ of $\ldr{}^U$.  First note that since
the dual dull bracket on $\Gamma(TM\oplus A^*)$ is anchored by
$\pr_{TM}$, we have
\begin{equation}\label{useful}
\Delta_{(X,\alpha)}(a,\theta)=\Delta_u(a,0)+(0,\ldr{X}\theta)
\end{equation}
for $X\in\mx(M)$, $\alpha\in\Gamma(A^*)$, $a\in\Gamma(A)$ and
$\theta\in\Omega^1(M)$.  We define the connection $\nabla^{\rm
  bas}\colon \Gamma(A)\times\Gamma(TM\oplus A^*)\to\Gamma(TM\oplus
A^*)$ associated to $\Delta$ as in \eqref{nabla_bas},
i.e.~$\nabla^{\rm
  bas}_a\nu=(\rho,\rho^t)\bigl(\Delta_\nu(a,0)-(0,\dr\langle
\nu,(a,0)\rangle)\bigr)+\ldr{a}\nu$ for $a\in\Gamma(A)$ and
$\nu\in\Gamma(TM\oplus A^*)$.  We prove that the bracket on $\Gamma(C)$
is given by
\begin{align}\label{bracket_on_C}
  \lb u_1\oplus\tau_1,
  u_2\oplus\tau_2\rb_C&=([u_1,u_2]_U+\nabla_{\pr_A\tau_1}^{\rm
    bas}u_2-\nabla_{\pr_A\tau_2}^{\rm
    bas}u_1)\nonumber\\
  &\hspace*{0.5cm}\oplus([\tau_1,\tau_2]_d+\Delta_{u_1}\tau_2-\Delta_{u_2}\tau_1+(0,\dr\langle\tau_1,u_2\rangle)).
\end{align}
We know that $\lb u_1\oplus 0, u_2\oplus 0\rb_C=[u_1,u_2]_U\oplus 0$, 
$\lb 0\oplus \tau_1, 0\oplus\tau_2\rb_C=0\oplus[\tau_1,\tau_2]_d$
for $u_1,u_2\in\Gamma(U)$ and $\tau_1,\tau_2\in\Gamma(A\oplus T^*M)$.

Note that $C/U$ is isomorphic to $A\oplus T^*M/U^\circ$ via the map
$\Psi\colon C/U\to A\oplus T^*M/U^\circ$, $\overline{u\oplus\tau}\to\overline\tau$. These two vector bundles are
isomorphic to $U^*$ and $\ldr{}^U$ is given by 
$\ldr{u}^U\overline\tau=\overline{\Delta_u\tau}$ in $A\oplus
T^*M/U^\circ$ and by $\ldr{u}^U\overline{u'\oplus \tau}=\overline{\lb u\oplus 0,
u'\oplus\tau\rb_C}$ in $C/U$. Hence
$\ldr{u}^U\overline{0\oplus\tau}=(\Psi\inv\circ\ldr{u}^U\circ\Psi)\overline{0\oplus\tau}=\overline{0\oplus\Delta_u\tau}$
and so
\[\overline{\lb u\oplus 0, 0\oplus\tau\rb_C}=\overline{0\oplus\Delta_u\tau}.\]
We want to compute $v=v(\tau,u)\in\Gamma(U)$
such that $\lb u\oplus 0, 0\oplus \tau\rb_C=v\oplus\Delta_u\tau$.
First note that 
\[\lb u\oplus 0, 0\oplus \tau\rb_C+\lb 0\oplus \tau, u\oplus 0\rb_C
=\mathcal D\langle\langle u\oplus 0, 0\oplus
\tau\rangle\rangle_C=\mathcal D\langle \tau,u\rangle=(0,\dr\langle\tau,u\rangle).
\]
We write $u=(X,\alpha)$ and $\tau=(a,\theta)$.
Then, by the Leibniz property of the Courant algebroid bracket on $C$,
we find for $\tau'=(b,\omega)\in\Gamma(A\oplus T^*M)$:
\begin{align*}
\rho(a)\langle u, \tau'\rangle&\,\,=c(0\oplus\tau)\langle\langle u\oplus
0, 0\oplus\tau'\rangle\rangle_C\\
&\,\,\,=\langle\langle\lb 0\oplus\tau, u\oplus
0\rb, 0\oplus\tau'\rangle\rangle_C
+\langle\langle u\oplus
0, \lb 0\oplus\tau, 0\oplus\tau'\rb\rangle\rangle_C\\
&\,\,\,=\langle\langle (-v)\oplus(-\Delta_u\tau+(0,\dr\langle \tau, u\rangle)) , 0\oplus\tau'\rangle\rangle_C\\
&\qquad 
+\langle\langle u\oplus
0, 0\oplus(\ldr{a}\tau'+(0,-\ip{\rho(b)}\dr\theta))\rangle\rangle_C\\
&\overset{\eqref{useful}}{=}-\langle v,\tau'\rangle-\langle
(\rho,\rho^t)(\Delta_u(a,0)-(0,\dr\langle a,\alpha\rangle)),\tau'\rangle+\cancel{\rho(b)\langle\theta,X\rangle}\\
&\qquad -\cancel{\langle\ldr{X}\theta,\rho(b)\rangle}+\langle
u, \ldr{a}\tau'\rangle
-\cancel{\dr\theta(\rho(b),X)}
\end{align*}
This leads to
$-\langle v,\tau'\rangle=\langle
(\rho,\rho^t)(\Delta_u(a,0)-(0,\dr\langle a,\alpha\rangle))+\ldr{a}u,
\tau'\rangle$
and, since $\tau'$ was arbitrary, we have shown that
$\lb u\oplus 0, 0\oplus \tau\rb=(-\nabla^{\rm
  bas}_a u)\oplus\Delta_u\tau$. Note that \eqref{bracket_on_C} does
not depend on the choice of the extension $\Delta$ of $\ldr{}^U$.
We have proved the following proposition:
\begin{proposition}\label{eq_diracbi_manin}
The  map 
\[\left\{\begin{array}{c}
\text{ Equivalence classes }\\
\text{ of $A$-Manin pairs}
\end{array}\right\}\to \left\{\begin{array}{c}
\text{Equivalence classes of }\\
\text{Dirac bialgebroids over $A$}
\end{array}\right\},
\]
that sends the class of $(C,U)$ (with structure maps $\iota,\Phi$) to the class of $(A,U,\iota)$, is a bijection.
\end{proposition}

\section{The Dirac bialgebroid associated to a Dirac groupoid}\label{manin}
First we recall the definition of a Dirac groupoid, as well as some of
their properties.
\begin{definition}[\cite{Ortiz13}]
  A Dirac groupoid $(G\rr M, \mathsf D)$ is a Lie groupoid
  $G\rightrightarrows M$ with a Dirac structure $\mathsf D$ such that
  $\mathsf D\subseteq TG\oplus T^*G$ is a Lie subgroupoid.
\end{definition}
The Dirac structure $\mathsf D$ is then said to be
\textbf{multiplicative}.  We write $U$ for the set of units of
$\mathsf D$ seen as a groupoid, i.e.~for the vector bundle
$U=\mathsf D\cap (TM\oplus A^*)$.  The inclusion of $U$ in
$TM\oplus A^*$ is always called $\iota\colon U\to TM\oplus A^*$.
We write $K$ for the vector bundle $K=\mathsf D\cap (A\oplus T^*M)$.
We have $K=U^\circ$ and also $(\rho,\rho^t)(K)\subseteq U$ since
$\TT\tg(\mathsf D)\subseteq U$.

Let $u$ be a section of $U$.  Then there exists a smooth section $d$
of $\mathsf D$ such that $d\an{M}=u$ and $\TT\tg\circ d=u\circ \tg$
\cite{Jotz13b}.  We then write $u\sim_{\epsilon}d$ and $d\sim_{\tg}u$.  A
section $d$ of $\mathsf D$ satisfying these two conditions is
called a \textbf{star section}.  In \cite{Jotz13b} we proved the following two results.
\begin{theorem}\label{lie_der_of_xi_section}
  Let $(G\rr M,\mathsf D)$ be a Dirac groupoid, $d\sim_{\tg}u$ a star
  section of $\mathsf D$ and $a\in\Gamma(A)$. Then the Lie derivative
  $\ldr{a^r}d$ can be written as a sum
\begin{equation*}
\ldr{a^r}d=\mathcal L_ad+(\tau_{d,a})^r
\end{equation*}
with $\mathcal L_ad$ a star section and $\tau_{d,a}$ a section of
$A\oplus T^*M$.
\end{theorem}

\begin{theorem}\label{lie_algebroid_dual}
  Let $(G\rr M,\mathsf D)$ be a Dirac groupoid.  Then there is an
  induced bracket $[\cdot\,,\cdot]_U\colon \Gamma(U)\times \Gamma(U)\to
  \Gamma(U)$ defined by $ [u, u']_U=\lb d,d'\rb\an{M}$ for any choice
  of star sections $d\sim_\tg u$, $d'\sim_\tg u'$ of $\mathsf D$.  The
  triple $(U, [\cdot\,,\cdot]_U, \rho_U\colon =\pr_{TM})$ is a Lie algebroid
  over $M$.
\end{theorem}
We shall prove that $(A,U,\iota)$ is a Dirac bialgebroid.  In order to
do this, we describe how an $A$-Manin pair is naturally associated to
a Dirac groupoid $(G\rr M, \mathsf D)$. The Lie algebroid $U$ is the Dirac structure in this
Manin pair.

Set $\mathsf B=: (U\oplus A\oplus T^*M)/K^\tg$, where we see
$K^\tg=(\ker\TT\tg\cap \mathsf D)\an{M}$ as a subbundle of
$U\oplus(A\oplus T^*M)=\mathsf D\an{M}+\ker\TT\s\an{M}$.  We write
sections of $\mathsf B$ as pairs $u\oplus \tau:=u+\tau+K^\tg$,
with $u\in\Gamma(U)$ and $\tau\in\Gamma(A\oplus T^*M)$.  Since
$(\ker\TT\s)^\perp=\ker\TT\tg$ relative to the canonical symmetric
pairing on $TG\oplus T^*G$, we have $\left(U\oplus(A\oplus
  T^*M)\right)^\perp =\left(\ker\TT\s\an{M}+\mathsf
  D\an{M}\right)^\perp =K^\tg$, and so the pairing
$\langle\cdot\,,\cdot\rangle$ on $TG\oplus T^*G$ restricts and
projects to a nondegenerate symmetric pairing on $\mathsf B$:
\begin{equation*}
\begin{split}
  \langle u_1\oplus \tau_1, u_2\oplus \tau_2\rangle_{\mathsf B}
  &=\alpha_2(a_1)+\alpha_1(a_2)+\theta_1(\bar
  X_2+\rho(a_2))+\theta_2(\bar X_1+\rho(a_1)) \\
  &=\langle u_1,\tau_2\rangle+\langle u_2,\tau_1\rangle+\langle
  \tau_1,\tau_2\rangle_d
\end{split}
\end{equation*}
 where $u_i=(\bar X_i,\alpha_i)$
and $\tau_i=(a_i,\theta_i)$.  We then define
$\mathsf b\colon \mathsf B\to TM$ by $\mathsf
b(u\oplus\tau)=\pr_{TM}(u)+\rho\circ\pr_A(\tau)$ and, finally,
$\lb\cdot\,,\cdot\rb_{\mathsf B}\colon \Gamma(\mathsf B)\times
\Gamma(\mathsf B) \to \Gamma(\mathsf B)$ by
\begin{align*}
  \left\lb u_1\oplus \tau_1, u_2\oplus\tau_2\right\rb_{\mathsf B}
  =\left.\left\lb d_1+\tau_1^r, d_2+\tau_2^r\right\rb\right\an{M} +K^\tg
\end{align*}
for all $\tau_1,\tau_2\in\Gamma\left(\ker\TT\s\an{M}\right)$,
$u_1,u_2\in\Gamma(U)$ and star sections $d_i\sim_\tg u_i$ of
$\mathsf D$. Here, $\lb\cdot\,,\cdot\rb$ is the Courant-Dorfman
bracket on sections of $TG\oplus T^*G$.  We have proved in
\cite{Jotz13b} that $(\mathsf B, \mathsf b,
\lb\cdot\,,\cdot\rb_{\mathsf B}, \langle\cdot\,,\cdot\rangle_{\mathsf
  B})$ is a Courant algebroid.

\begin{theorem}\label{Manin_pair}
  Let $(G\rr M, \mathsf D)$ be a Dirac groupoid, and let $A$ be the
  Lie algebroid of $G\rr M$. Then the pair $(\mathsf B, U)$ defined as
  above is naturally an $A$-Manin pair.
\end{theorem}

\begin{proof}
  By construction, $U$ is a subbundle of $TM\oplus A^*$. We first show
  that $U$ is a Dirac structure in $\mathsf B$, when we identify $U$ with
  its image under the injective map $U\to \mathsf B$, $u\mapsto
  u\oplus 0$.  A dimension count shows that
  $2\operatorname{rank}(U)=\operatorname{rank}(\mathsf B)$, and $U$ is
  obviously isotropic relative to
  $\langle\cdot\,,\cdot\rangle_{\mathsf B}$.  By
  Theorem~\ref{lie_algebroid_dual}, $\Gamma(U)$ is closed under the
  bracket on $\Gamma(\mathsf B)$.

  By construction again, we find immediately that the map $\Phi\colon
  A\oplus T^*M\to \mathsf B$, $\tau\mapsto 0\oplus \tau$ is a
  morphism of (degenerate) Courant algebroids: it is easy to check
  that $\lb 0\oplus \tau_1, 0\oplus\tau_2\rb_{\mathsf
    B}=0\oplus[\tau_1,\tau_2]_d$, $\langle 0\oplus
  \tau_1, 0\oplus\tau_2\rangle_{\mathsf B}=\langle \tau_1,
  \tau_2\rangle_d$ and $\mathsf
  b(0\oplus\tau)=\rho\circ\pr_A(\tau)$.  The sum
  $\iota_U+\Phi\colon U\oplus(A\oplus T^*M)\to \mathsf B$,
  $(u,\tau)\mapsto u\oplus \tau$ is surjective, and 
  we have $\langle
  u,\Phi(\tau)\rangle=\langle u\oplus 0, 0\oplus
  \tau\rangle_{\mathsf B}=\langle u,\tau\rangle$.
\end{proof}

\begin{corollary}\label{cor_Dirac_bi}
  The triple $(A,U,\iota)$, with $U$ endowed with the Lie algebroid
  structure defined in Theorem~\ref{lie_algebroid_dual}, is a Dirac
  bialgebroid.
\end{corollary}

Our goal is to show that this defines a bijection
between source simply connected Dirac groupoids and integrable Dirac bialgebroids.

The following proposition shows that the Lie derivative
$\ldr{}^{\mathsf D}$ is right-invariant and projects to the Lie
derivative $\ldr{}^U$. Note that this is in fact the key feature of a
multiplicative Dirac structure, and that it resembles the right-invariance of the
Bott connection corresponding to a multiplicative involutive
distribution on a Lie groupoid \cite{JoOr14}.
\begin{proposition}\label{induced_Dorfman}
  Let $(G\rr M, \mathsf D)$ be a Dirac groupoid.  
\begin{enumerate}
\item The Dorfman connection $\ldr{}^{\mathsf D}=\Delta^{\mathsf
    D}\colon \Gamma(\mathsf D)\times \Gamma(TG\oplus T^*G/\mathsf
  D)\to \Gamma(TG\oplus T^*G/\mathsf D)$ induces as follows a Dorfman
  connection $\Delta\colon \Gamma(U)\times \Gamma(A\oplus T^*M/K)\to
  \Gamma(A\oplus T^*M/K)$.  For a star section $d\sim_\tg u$ of
  $\mathsf D$ and a right-invariant section $\tau^r$ of $TG\oplus
  T^*G$, we have
\[\Delta^{\mathsf D}_d\overline{\tau^r}=(\Delta_u\bar \tau)^r.
\]
\item Note that $A\oplus T^*M/K\simeq U^*$ and the Dorfman connection
  $\Delta$ is exactly the Dorfman connection defined by $U$ in
  $\mathsf B$, or in other words the Lie derivative $\ldr{}^U$ of $\Gamma(U^*)$
  by sections of $U$.
\end{enumerate}
\end{proposition}

\begin{proof}[Proof of Proposition~\ref{induced_Dorfman}]
  For the first claim write $\tau=(a,\theta)$, with $a\in\Gamma(A)$
  and $\theta\in\Omega^1(M)$, and $\pr_{TG}(d)=X$, $\pr_{TM}(u)=\bar
  X$.  Using Theorem~\ref{lie_der_of_xi_section}, we get the equality
\begin{align*}
  \lb d, \tau^r\rb &=-\lb\tau^r, d\rb+(0,\dr \langle
  d,\tau^r\rangle)
  =-\ldr{a^r}d+(0,\ip{X}\dr\tg^*\theta)+(0,\tg^*\dr \langle u,\tau\rangle)\\
&=-\mathcal L_ad-(\tau_{d,a}-(0,\ip{\bar X}\dr\theta+\dr
\langle u,\tau\rangle))^r,
\end{align*}
which implies 
\[\overline{\lb d, \tau^r\rb}=\overline{-(\tau_{d,a}-(0,\ip{\bar X}\dr\theta+\dr
\langle u,\tau\rangle))^r}
\]
in $TG\oplus T^*G/\mathsf D$. This is $0$ if and only if
$\tau_{d,a}-(0,\ip{\bar X}\dr\theta+\dr \langle
u,\tau\rangle)\in\Gamma(K)$.

We show that this does not depend on the choice of $d$ over the section
$u\in\Gamma(U)$.  If $d'\sim_\tg u$ is an other star section of
$\mathsf D$ over $u$, then we have $d-d'\in\Gamma(\mathsf D\cap
\ker\TT\tg)$ and $(d-d')\an{M}=0$. Choose a (local) basis of sections
$\tau_1,\ldots,\tau_r$ of $K^\tg$.  Then $d-d'$ can be written as
$d-d'=\sum_{i=1}^rf_i \tau_i^l$ with $f_1,\ldots,f_r\in C^\infty(G)$
vanishing on $M$. Using this, it is easy to check that $\lb d-d',
\tau^r\rb =-\sum_{i=1}^ra^r(f_i)\tau_i^l\in\Gamma(\mathsf D)$ and so that
$\overline{\lb d-d',\tau^r\rb}=0$.
Set $\Delta_u\bar\tau=\overline{-\tau_{d,a}+(0,\ip{\bar X}\dr\theta+\dr
\langle u,\tau\rangle)}$. Checking that $\Delta$ is a Dorfman
connection is straightforward.

For the second claim let $u'$ be a section of $U$ and $d'$ a star section of $\mathsf
  D$ over $u'$.  We compute
\begin{align*}\langle \Delta_u\bar\tau,u'\rangle
  &=\langle \Delta^{\mathsf D}_d\overline{\tau^r}, d'\rangle\an{M}
  =\langle \lb d,\tau^r\rb,d\rq{}\rangle\an{M}\\
  &=\left( \pr_{TG}(d)\langle \tau^r, d'\rangle -\langle \tau^r, \lb d,d'\rb\rangle\right)\an{M}\\
  &=\pr_{TM}(u)\langle\tau,u'\rangle-\langle\tau,
  [u,u']_U\rangle=\langle\ldr{u}\bar\tau,u'\rangle.\qedhere
\end{align*}
\end{proof}

\section{The Dirac bialgebroid associated to a Dirac algebroid}\label{dir_alg}
Let $A$ be a Lie algebroid. Recall from Example~\ref{example_TAT*A}
that the space $TA\oplus T^*A$ has the structure of a double vector
bundle with sides $TM\oplus A^*$ and $A$ and with core $A\oplus T^*M$.
The vector bundle $TA\oplus T^*A\to A$ is equipped with the standard
Courant algebroid structure, and the vector bundle $TA\oplus T^*A\to
TM\oplus A^*$ is equipped with a Lie algebroid structure, which was
described in Example~\ref{example_TAT*A}.  These two structures are
compatible and $TA\oplus T^*A$ is an LA-Courant algebroid. (LA-Courant
algebroids are defined in \cite{Li-Bland12}. An alternative, perhaps
more handy definition in terms of representations up to homotopy and
Dorfman $2$-representations is given in \cite{Jotz15}.)
\begin{definition}
  A Dirac algebroid is a Lie algebroid $A$ with a Dirac structure
  $\mathsf D_A\subseteq TA\oplus T^*A$ that is also a sub-algebroid of
  $TA\oplus T^*A\to TM\oplus A^*$ over a subbundle
  $U\subseteq TM\oplus A^*$.
\end{definition}
\noindent We also say that $\mathsf D_A$ is an LA-Dirac structure on
$A$. It is in fact a sub-double Lie algebroid of the LA-Courant
algebroid $(TA\oplus T^*A;A, TM\oplus A^*;M)$.

This section shows how LA-Dirac structures are equivalent to $A$-Manin
pairs.

\subsection{Decompositions of LA-Dirac structures}\label{sum_Jotz13a}
Let $A$ be a Lie algebroid.  Recall that in \S\ref{example_TAT*A} we
have described linear splittings of the VB-algebroid
$(TA\oplus T^*A\to TM\oplus A^*, A\to M)$.  Now we consider sub-double
vector bundles
\begin{equation*}
\begin{xy}
\xymatrix{
D\ar[d]\ar[r]&U\ar[d]\\
A\ar[r]&M
}\end{xy}
\qquad \text{ of } \qquad
\begin{xy}
\xymatrix{
TA\oplus T^*A\ar[d]\ar[r]&TM\oplus A^*\ar[d]\\
A\ar[r]&M
}\end{xy}.
\end{equation*}
We denote the core of $D$ by $K\subseteq A\oplus T^*M$.  There exists
a Dorfman connection $\Delta$ such that $D\to A$ is spanned by the
sections $k^\dagger$, for all $k\in\Gamma(K)$ and
$\sigma^\Delta_{TM\oplus A^*}(u)$ for all $u\in\Gamma(U)$ (see
\cite{Jotz13a}).  The Dorfman connection $\Delta$ is then said to be
\textbf{adapted} to $D$ and the double vector subbundle $D$ is
completely determined by the triple $(U,K,\Delta)$.  Conversely we
write $D_{(U,K,\Delta)}$ for the double vector subbundle defined in
this manner by a Dorfman connection $\Delta$ and two subbundles
$U\subseteq TM\oplus A^*$, $K\subseteq A\oplus T^*M$.  The set of
sections of $D\to U$ is spanned as a $C^\infty(U)$-module by the
sections $\sigma_A^\Delta(a)\an{U}$ and $k^\dagger\an{U}$ for all
$a\in\Gamma(A)$ and $k\in\Gamma(K)$. We have
$D_{(U,K,\Delta)}=D_{(U,K,\Delta')}$ if and only if
$(\Delta-\Delta')_u\tau\in\Gamma(K)$ for all $u\in\Gamma(U)$ and
$\tau\in\Gamma(A\oplus T^*M)$.

By the results in \cite{Jotz13a}, $D$ is a Dirac structure $\mathsf
D_A$ over $A$, if and only if $K=U^\circ$, any adapted Dorfman
connection $\Delta$ satisfies $\Delta_uk\in\Gamma(K)$ for all
$u\in\Gamma(U)$ and $k\in\Gamma(K)$ and the induced quotient Dorfman
connection $\bar \Delta\colon \Gamma(U)\times\Gamma(A\oplus T^*M/K)\to
\Gamma(A\oplus T^*M/K)$, $\bar\Delta_u\bar
\tau=\overline{\Delta_u\tau}$ is flat. That is, since $A\oplus
T^*M/U^\circ\simeq U^*$, the Dorfman connection $\bar\Delta$ is in
fact dual to a Lie algebroid structure on $U$ with anchor $\pr_{TM}$.
The Lie algebroid bracket is the restriction to
$\Gamma(U)\times\Gamma(U)$ of the dull bracket
$\lb\cdot\,,\cdot\rb_\Delta$ that is dual to $\Delta$. The quotient
Dorfman connection and so also the obtained Lie algebroid structure on
$U$ do not depend on the choice of the adapted Dorfman connection to
$\mathsf D_A$.

Next, a double vector subbundle as above defines a VB-subalgebroid
$(D\to A, U\to M)$ of $(TA\oplus T^*A\to TM\oplus A^*, A\to M)$ if
and only if $(\rho,\rho^t)(K)\subseteq U$, $\nabla^{\rm
  bas}_ak\in\Gamma(K)$ for all $a\in\Gamma(A)$ and $k\in\Gamma(K)$,
$\nabla^{\rm bas}_au\in\Gamma(U)$ for all $u\in\Gamma(U)$ and
$R_\Delta^{\rm bas}(a,b)u\in\Gamma(K)$ for all $a,b\in\Gamma(A)$ and
$u\in\Gamma(U)$.  We get the following theorem.

\begin{theorem}\cite{Jotz13a}\label{morphic_Dirac_triples}
Consider a Lie algebroid $A$ and a    Dorfman connection
\[\Delta\colon \Gamma(TM\oplus A^*)\times\Gamma(A\oplus T^*M)\to \Gamma(A\oplus T^*M).
\]
Let $U\subseteq TM\oplus A^*$ and $K\subseteq A\oplus T^*M$ be subbundles. Then 
$D_{(U,K,\Delta)}$ is a Dirac structure in $TA\oplus T^*A\to A$ and a 
subalgebroid of $TA\oplus T^*A\to TM\oplus A^*$
over $U$ if and only if: 
\begin{enumerate}
\item $K=U^\circ$
\item $(\rho,\rho^t)(K)\subseteq U$,
\item $(U,\pr_{TM}, \lb\cdot\,,\cdot\rb_\Delta\an{\Gamma(U)\times\Gamma(U)})$ is a Lie algebroid, 
\item $\nabla_a^{\rm bas}u\in\Gamma(U)$ for all $a\in\Gamma(A)$ and $u\in\Gamma(U)$,
\item $R_\Delta^{\rm bas}(a,b)u\in\Gamma(K)$ for all $u\in\Gamma(U)$, $a,b\in\Gamma(A)$.
\end{enumerate}
\end{theorem}

In the next subsection we use Theorem~\ref{morphic_Dirac_triples} to construct the Manin pair associated to the
Dirac bialgebroid. 
\begin{remark}\label{remark_duality}
The equation
\begin{equation}\label{not_dual}
\langle \nabla_{a'}^{\rm bas}\nu, \tau\rangle
+\langle \nu, \nabla_{a'}^{\rm bas}\tau\rangle
=\rho(a')\langle \nu, \tau\rangle-\langle \lb \nu,
(\rho,\rho^t)\tau\rb_\Delta+ \lb 
(\rho,\rho^t)\tau, \nu\rb_\Delta, \tau'\rangle
\end{equation}
for $\tau'=(a',\theta')\in \Gamma(A\oplus T^*M)$,
$\nu\in\Gamma(TM\oplus A^*)$ and $\tau\in\Gamma(A\oplus T^*M)$,  is easily checked.

Hence the two $A$-connections $\nabla^{\rm bas}$ are not in duality,
  but if $K=U^\circ$ and the restriction of the dull bracket to
  $\Gamma(U)\times\Gamma(U)$ is skew-symmetric, then for
  $a\in\Gamma(A)$, $\nabla^{\rm bas}_ak\in\Gamma(K)$ for all
  $k\in\Gamma(K)$ if and only if $\nabla^{\rm bas}_au\in\Gamma(U)$ for
  all $u\in\Gamma(U)$; see \cite{Jotz13a}. 
\end{remark}

Given the linear splitting given by the Dorfman connection adapted to
  $\mathsf D_A$, the Lie algebroid structure $\mathsf D_A\to A$ is described by
the $2$-term representation up to homotopy defined by the vector
bundle morphism $\pr_A\colon K\to A$, the connections
$(\pr_A\circ\,\Omega)\colon \Gamma(U)\times\Gamma(A)\to\Gamma(A)$ and
$\Delta\colon \Gamma(U)\times\Gamma(K)\to\Gamma(K)$, and the curvature
$R_\Delta\in\Omega^2(U, \Hom(A, K))$.

The Lie algebroid structure of the other side $\mathsf D_A\to U$
is described by the complex $(\rho,\rho^t)\colon K\to U$ with the
basic connections $\nabla^{\rm bas}\colon \Gamma(A)\times\Gamma(U)\to
\Gamma(U)$ and $\nabla^{\rm bas}\colon \Gamma(A)\times\Gamma(K)\to
\Gamma(K)$
and the basic curvature $R_\Delta^{\rm bas}\in\Omega^2(A,\Hom(U,K))$.

\begin{remark}
A long but straightforward computation shows that these two $2$-term
representations up to homotopy form a matched pair \cite{GrJoMaMe14}.
One of the seven equations defining the matched pair is exactly
\eqref{bialgebroid1} in Lemma~\ref{difficult_computation}, which is
also the key to the most complicated equation. \eqref{bialgebroid2} is
a useful tool to check that equation as well. Corollary
\ref{eq_for_morphism} is another one, and most of the remaining
equations follow immediately from the definition of the involved
objects.

As a consequence, $(\mathsf
D_A, U, A, M)$ is a double Lie algebroid \cite{GrJoMaMe14}, and the
core $K$ has an induced Lie algebroid structure, which is given by the
restriction of $[\cdot\,,\cdot]_d$ to $\Gamma(K)$, with Lie algebroid
morphisms to $A$ and $U$. To see this, use
\cite[Remark 3.5]{GrJoMaMe14} and Lemmas \ref{basic_like_eq} and
\ref{eq_for_morphism} below.
\end{remark}

\subsection{The $A$-Manin pair associated to an LA-Dirac structure}\label{the_courant_alg}
Consider a triple $(U,U^\circ,\Delta)$ as in Theorem
\ref{morphic_Dirac_triples} and define the vector bundle $C$ with base $M$ as in \eqref{C_ident}.
We write $u\oplus \tau$ for the class in $C$ of a pair $(u,\tau)\in\Gamma(U\oplus(A\oplus T^*M))$. 
It is easy to check that the pairing 
\begin{equation*}
  \langle\langle u_1\oplus \tau_1, u_2\oplus\tau_2\rangle\rangle_C
:=\langle u_1,\tau_2\rangle+\langle u_2,\tau_1\rangle+\langle\tau_1,(\rho,\rho^t)\tau_2\rangle
\end{equation*}
defines a symmetric fibrewise pairing
$\langle\langle\cdot\,,\cdot\rangle\rangle_C$ on $C$.  It is easy to
see that this pairing $\langle\langle\cdot\,,\cdot\rangle\rangle_C$ is
nondegenerate, and that the vector bundle $C$ is isomorphic to
$U\oplus U^*$.  Now set \[c\colon C\to TM, \qquad
c(u\oplus\tau)=\pr_{TM}(u)+\rho\circ\pr_{A}(\tau).\]
\begin{theorem}\label{Courant_algebroid}
Let $A$ be a Lie algebroid and consider an LA-Dirac structure $(\mathsf
D_A, U,A,M)$ over $A$.  Choose a Dirac structure
$\Delta\colon\Gamma(TM\oplus A^*)\times\Gamma(A\oplus
T^*M)\to\Gamma(A\oplus T^*M)$ that is adapted to $D_A$.
Then $C$ in \eqref{C_ident} is a Courant algebroid with anchor 
$c$, pairing $\langle\langle\cdot\,,\cdot\rangle\rangle_C$
and bracket $\lb\cdot\,,\cdot\rb\colon \Gamma(C)\times\Gamma(C)\to\Gamma(C)$,
\begin{multline*}
  \lb u_1\oplus\tau_1, u_2\oplus\tau_2\rb\\
  =\,\left(\lb u_1,u_2\rb_\Delta+\nabla_{\tau_1}^{\rm
      bas}u_2-\nabla_{\tau_2}^{\rm bas}u_1\right)\oplus\left(
    [\tau_1,\tau_2]_{D}+\Delta_{u_1}\tau_2-\Delta_{u_2}\tau_1 +(0,\dr\langle
    \tau_1,u_2\rangle)\right).
\end{multline*}
The map $\mathcal D=\,c^*\circ \dr\colon C^\infty(M)\to \Gamma(C)$
is given by 
$f\mapsto 0\oplus(0,\dr f)$. The bracket does not depend on the choice
of $\Delta$.
\end{theorem}

The proof of this theorem can be found in the appendix.

\begin{remark}
\begin{enumerate}
\item It is easy to check that the Courant algebroid structure does
  not depend on the choice of $\Delta$ adapted to $\mathsf D_A$.
\item This construction has some similarities with the one of matched
  pairs
of Courant algebroids in \cite{GrSt12}. It would be interesting to
understand the relation between the two constructions.
\end{enumerate}
\end{remark}

It is easy to see that the Lie algebroid $(U,\pr_{TM}\an{U},\lb\cdot\,,\cdot\rb_\Delta)$  
is a Dirac structure in $C$, and that $(C,U)$ is an $A$-Manin pair over $M$.
Hence, we get the following corollary.
\begin{corollary}
Let $A$ be a Lie algebroid and consider an LA-Dirac structure $(\mathsf
D_A, U,A,M)$ over $A$.
Then the triple $(A,U,\iota)$, with $\iota$ the inclusion of $U$ in $TM\oplus A^*$, is
a Dirac bialgebroid.
\end{corollary}

Now we can prove the following theorem.
\begin{theorem}\label{inf_cor}
  Let $A$ be a Lie algebroid over a manifold $M$. Then the
  construction above defines a one--one correspondence between
  LA-Dirac structures on $A$ and equivalence classes of Dirac bialgebroids on $A$.
\end{theorem}

\begin{proof}
  We show here that there is a one-one correspondence between pairs
  $(U,K=U^\circ,\Delta)$ as in Theorem~\ref{morphic_Dirac_triples} and
  Dirac bialgebroids $(A,U,\iota)$.  From the definition of the
  bracket, anchor and pairing in Theorem~\ref{Courant_algebroid}, we
  find immediately that $(C,U)$ is an $A$-Manin pair and so that
  $(A,U,\iota)$ is a Dirac bialgebroid.

  Conversely, choose an $A$-Manin pair $(C,U)$. Recall from the proof
  of Proposition~\ref{eq_diracbi_manin} that $C$ can be written as in
  \eqref{C_ident} and that the anchor, pairing and
  brackets are defined as in Theorem~\ref{Courant_algebroid}, with
  $\Delta\colon\Gamma(TM\oplus A^*)\times\Gamma(A\oplus
  T^*M)\to\Gamma(A\oplus T^*M)$ any extension of the Lie derivative
  $\ldr{}^U$ dual to the Lie algebroid structure on $U$.
The proof of Theorem~\ref{Courant_algebroid}
shows that all the conditions in Theorem~\ref{morphic_Dirac_triples}
are then satisfied for the pair $U$ and $\Delta$.
\end{proof}

\section{The Lie algebroid of a multiplicative Dirac structure}\label{lie_algebroid}
In this section, we recall how the Lie algebroid of a Dirac groupoid
is equipped with an LA-Dirac structure \cite{Ortiz13}. Then we compute
explicitly the Dirac algebroid of a Dirac groupoid and we prove that
they induce equivalent Dirac bialgebroids.
\subsection{Dirac groupoids correspond to Dirac algebroids}\label{cristian}
Ortiz shows in \cite{Ortiz13} that modulo 
the canonical isomorphism (see \cite{JoStXu15})
\begin{equation*}
\begin{xy}
\xymatrix{
A(TG\oplus T^*G)\ar[dd]\ar[rd]\ar[rr]^I&&TA\oplus T^*A \ar[rd]\ar[dd]&\\
&TM\oplus A^*\ar[dd]\ar@{=}[rr]& & TM\oplus A^*\ar[dd]\\
A(G)=A\ar[rd]\ar@{=}[rr]&&A\ar[rd]\\
&M \ar@{=}[rr]&&M\\
}\end{xy},
\end{equation*} the Lie algebroid
$A(\mathsf D)\to U$ of a multiplicative Dirac structure $\mathsf D$ on
$G\rr M$ is an LA-Dirac structure
\begin{equation*}
\begin{xy}
\xymatrix{
\mathsf D_A\ar[d]\ar[r]&U\ar[d]\\
A\ar[r]&M
}\end{xy}
\end{equation*} 
 i.e.~$(A,\mathsf D_A)$ is a Dirac algebroid. 
Assume that $G\rr M $ is source simply connected. Ortiz proves that this defines a one-to-one
correspondence between LA-Dirac structures over $A=A(G)$ and multiplicative
Dirac structures on $G\rr M$ \cite{Ortiz13}:
\begin{equation}\label{cor_cristian}
\left\{\begin{array}{c}
\mathsf D_A \text{ LA-Dirac structure}\\
\text{on $A$}\\
\end{array}\right\}
\overset{1:1}{\leftrightarrow}
\left\{\begin{array}{c}
 \mathsf D \text{ multiplicative Dirac structure}\\
\text{ on $G\rr M$}
\end{array}\right\}.
\end{equation}

\subsection{The Lie algebroid of a multiplicative Dirac structure}\label{explicit_computation}
Now consider again a Dirac groupoid $(G\rr M, \mathsf D)$ and consider
an extension $\Delta\colon \Gamma(TM\oplus A^*)\times\Gamma(A\oplus
T^*M)\to \Gamma(A\oplus T^*M)$ of the induced Dorfman connection
$\ldr{}^U$ as in Proposition~\ref{induced_Dorfman}. (Recall that
$U^*\simeq (A\oplus T^*M)/U^\circ$.) Define then $\Omega\colon
\Gamma(TM\oplus A^*)\times\Gamma(A)\to \Gamma(A\oplus T^*M)$ as in
\eqref{omega}.

\begin{lemma}\label{duality}
  For $a\in\Gamma(A)$, the map $\phi_a\colon \Gamma(U)\to\Gamma(U^*)$
  defined by
  \[\phi_a(u)=\overline{\Omega_ua}, \qquad a\in\Gamma(A),
  u\in\Gamma(U)
\] satisfies $\phi_a^t=-\phi_a$.
\end{lemma}

\begin{proof}
The proof is just a computation:
\begin{align*}
  \langle \phi_a(u_1), u_2\rangle &=\langle
  \Delta_{u_1}a-(0,\dr\langle u_1, a\rangle), u_2\rangle= \langle
  \ldr{u_1}^U\overline{(a,0)}, u_2\rangle -\pr_{TM}(u_2)\langle
  u_1, a\rangle\\
  &= \langle \lb u_1\oplus 0, 0\oplus a\rb_{\mathsf B},
  u_2\oplus0\rangle -\pr_{TM}(u_2)\langle u_1,
  a\rangle\\
  &=-\langle \lb 0\oplus a, u_1\oplus0\rb_{\mathsf B},
  u_2\oplus 0\rangle\\
  &=-\rho(a)\langle u_1\oplus 0, u_2\oplus 0\rangle+\langle u_1\oplus
  0, \lb 0\oplus a, u_2\oplus
  0\rb_{\mathsf B}\rangle\\
  &=-\langle u_1, \phi_a(u_2)\rangle
\end{align*}
for $u_1,u_2\in\Gamma(U)$ and $a\in\Gamma(A)$.
\end{proof}

\begin{remark}\label{useful_on_Phi}
  By the proof of Proposition~\ref{induced_Dorfman}, we have
    $\phi_a(u)=\overline{-\tau_{a,d}}$ for any
    star section $d\in\Gamma(\mathsf D)$ over $u\in\Gamma(U)$.
   
\end{remark}

 We compute the
linear and core sections of the Lie algebroid $A(TG\oplus T^*G)\to
TM\oplus A^*$.  Let $a$ be a section of $A$, $m\in M$, $(v_m,\alpha_m)\in
U_m$ and consider the curve $\hat a(v_m,\alpha_m)\colon
(-\varepsilon,\varepsilon)\to TG\oplus T^*G$,
\begin{align*}
  \hat
  a(v_m,\alpha_m)(t)
  &=\left( T_m\Exp(ta)v_m,(T_{\Exp(ta)(m)}L_{\Exp(-ta)})^t\alpha_m\right)
\end{align*}
where $\varepsilon>0$ is so that $\Exp(ta)(m)$ is defined for all
$t\in (-\varepsilon,\varepsilon)$. Here, $L_{\Exp(ta)}$ is the left
multiplication by the bisection $\Exp(ta)$ of $G\rr M$.
In the same manner, given $\sigma\in\Gamma(A\oplus T^*M)$, consider
the curve $\tau^c(v_m,\alpha_m)\colon \R\to TG\oplus T^*G$,
$\tau^c(v_m,\alpha_m)(t)=(v_m,\alpha_m)+t\sigma(m)$.
We define the linear and core sections $a^l,
\tau^\times\in\Gamma(A(TG\oplus T^*G))$ by
\[a^l(v_m,\alpha_m)=\left.\frac{d}{dt}\right\an{t=0}\hat
a(v_m,\alpha_m)(t)\quad \text{ and } \quad
\tau^\times(v_m,\alpha_m)=\left.\frac{d}{dt}\right\an{t=0}\sigma^c(v_m,\alpha_m)(t)\]
for all $(v_m,\alpha_m)\in TM\oplus A^*$.  More generally, if
$\phi\in\Gamma(\operatorname{Hom}(TM\oplus A^*, A\oplus T^*M))$ and
$a\in\Gamma(A)$, then the linear section
$\alpha_{a,\phi}\in\Gamma(A(TG\oplus T^*G))$ is given by
\[\alpha_{a,\phi}(v_m,\alpha_m)
:=a^l(v_m,\alpha_m)+\left.\frac{d}{dt}\right\an{t=0}(v_m,\alpha_m)+t\phi(v_m,\alpha_m)
\]
for all $(v_m,\alpha_m)\in TM\oplus A^*$, where the sum is taken in the
fibre over $(v_m,\alpha_m)$ of $A(TG\oplus T^*G)\to TM\oplus A^*$.

We find out which $a\in\Gamma(A)$ and
$\phi\in\Gamma(\operatorname{Hom}(TM\oplus A^*, A\oplus T^*M))$ lead
to a section $\alpha_{a,\phi}$ that restricts to a section of
$A(\mathsf D)$ on $U$. 

\begin{proposition}
  For all $a\in\Gamma(A)$, we have
  $\alpha_{a,\phi}\an{U}\in\Gamma(\mathsf A(\mathsf D))$ if and only if
  $\overline{\phi\an{U}}=\phi_a$.
\end{proposition}

\begin{proof}
  Since $\mathsf D=\mathsf D^\perp$, a vector field $X\in\mx(TG\oplus
  T^*G)$ restricts to a section of $T\mathsf D$ if and only if
  $\dr\ell_d(X)=0$ on $\mathsf D$ for all $d\in\Gamma(\mathsf D)$,
  where $\ell_d\colon TG\oplus T^*G\to \R$ is the linear function
  associated to $d$.

  Choose $u_m\in U$, a star-section $d\in\Gamma(\mathsf D)$,
  $d\sim_{\tg}u'$, defined in a neighbourhood of $m\in M$ and compute:
\begin{align*}
  &\,\dr l_d(u_m)(\alpha_{a,\phi}(u_m)) =\left.\frac{d}{dt}\right\an{t=0}
  \left\langle \hat a(u_m)(t), d(\Exp(ta)(m)) \right\rangle +
  \left\langle u_m+t \phi(u_m), d(m)
  \right\rangle\\
  &\,\,=\,\left.\frac{d}{dt}\right\an{t=0} \left\langle u_m,
    (L_{\Exp(ta)}^*d)(m) \right\rangle+ t\cdot \left\langle \phi(u_m),
    \TT\tg (d(m))
  \right\rangle\\
&\,\,=\langle u_m, (\ldr{a^r}d)(m)\rangle+ \langle \phi(u_m), u'(m)\rangle
=\langle u_m, \overline{\tau_{a,d}}(m)\rangle +\langle \phi(u_m), u'(m)\rangle.
\end{align*}
Since $\phi_a(u')=-\overline{\tau_{d,a}}$ (Remark
\ref{useful_on_Phi}) we find that this vanishes for all star sections
$d\in\Gamma(D)$ if and only if ${\overline{\phi\an{U}}}^t=-\phi_a$.  By
Lemma~\ref{duality}, this is equivalent to $\overline
{\phi\an{U}}=\phi_a$.
\end{proof}

\begin{corollary}
  The subbundle $A(\mathsf D)\to U$ of $A(TG\oplus T^*G)\to TM\oplus
  A^*$ is spanned by the sections $\alpha_{a,\Omega_\cdot a}\an{U}$ and
  $\tau^\times\an{U}$, for all $a\in\Gamma(A)$ and
  $\tau\in\Gamma(U^\circ)$.
\end{corollary}

This corollary allows us to describe in the next section $\mathsf D_A=I(A(\mathsf D))$
in terms of the Lie algebroid $U$.

\subsection{Conclusion and main result}\label{conclusion}
Consider again a Dirac groupoid $(G\rr M, \mathsf D)$ and an extension
$\Delta\colon \Gamma(TM\oplus A^*)\times\Gamma(A\oplus T^*M)\to
\Gamma(A\oplus T^*M)$
of the induced Dorfman connection $\ldr{}^U$ as in
Proposition~\ref{induced_Dorfman}.  The Dirac groupoid
$(G\rr M, \mathsf D)$ induces an $A$-Manin pair $(\mathsf B, U)$ as in
Theorem \ref{Manin_pair} and hence a Dirac bialgebroid $(A,U,\iota)$.
By Theorem~\ref{inf_cor} there is then an LA-Dirac structure on $A$
corresponding to $(A, U, \iota)$. Our goal is to show how this
LA-Dirac structure is, via the canonical isomorphism
$A(TG\oplus T^*G)\simeq TA\oplus T^*A$, the Lie algebroid
$A(\mathsf D)\to U$ of the multiplicative Dirac structure
$\mathsf D\rr U$.

Note that given a Dirac groupoid $(G\rr M, \mathsf D)$, we have found
two Lie algebroid structures on the bundle $U$ of units of $\mathsf
D$. The first one is $(\pr_{TM}, [\cdot\,,\cdot]_U)$, which was
defined in Theorem~\ref{lie_algebroid_dual}, the Lie algebroid
structure of $U$ as a Dirac structure in $\mathsf B$.  The second Lie
algebroid structure on $U$, $(\pr_{TM},
\lb\cdot\,,\cdot\rb_{\Delta'}\an{\Gamma(U)\times\Gamma(U)})$ was defined
by any Dorfman connection $\Delta'$ that is adapted to the LA-Dirac structure
$\mathsf D_A=I(A(\mathsf D))$ in $TA\oplus T^*A\to A$. In order to
prove our main theorem, it remains to show that given a Dirac groupoid
$(G\rr M, \mathsf D)$, we have
$[\cdot\,,\cdot]_U=\lb\cdot\,,\cdot\rb_{\Delta'}\an{\Gamma(U)\times\Gamma(U)}$
for any Dorfman connection $\Delta'$ that is adapted to 
$\mathsf D_A=I(A(\mathsf D))$. For this, it suffices to show that
$\Delta$ is adapted to $\mathsf D_A$.

The canonical isomorphism $I\colon A(TG\oplus T^*G)\to TA\oplus T^*A$
of double vector bundles can quickly be described as follows.  The
linear section $\alpha_{a,\phi}$ for $a\in\Gamma(A)$ and
$\phi\in\Gamma(\operatorname{Hom}(TM\oplus A^*,A\oplus T^*M))$ is sent
by $I$ to $(Ta,a^R))-\widetilde{\phi}$. In particular, the image of
$\alpha_{a,\Omega(\cdot,a)}$ is $\sigma_A^\Delta(a)$.  The image under
$I$ of the core section $\tau^\times$ is $\tau^\dagger$.  As a
consequence, any extension $\Delta$ of the Dorfman connection
$\ldr{}^U\colon \Gamma(U)\times\Gamma(A\oplus
T^*M/U^\circ)\to\Gamma(A\oplus T^*M/U^\circ)$ that is dual to
$[\cdot\,,\cdot]_{U}$ is adapted to $\mathsf D_A=A(\mathsf D)$.  In
other words, the double vector subbundle $(\mathsf D_A:=I(A(\mathsf
D)); U,A;M)$ of $(TA\oplus T^*A;TM\oplus A^*, A;M)$ corresponds to the
triple $(U,U^\circ,\Delta)$.

This yields the following proposition.
\begin{proposition}\label{same_prop}
  Let $(G\rr M, \mathsf D)$ be a Dirac groupoid and let $A$ be the Lie
  algebroid of $G\rr M$.  Then the bijection defined by Theorem~\ref{inf_cor}
associates the LA-Dirac structure $\mathsf
  D_A:=I(A(\mathsf D))$ to the class of the Dirac bialgebroid
  $(A,U,\iota)$, with $U$ endowed with the Lie algebroid structure
  $(\pr_{TM}, [\cdot\,,\cdot]_{U})$.
\end{proposition}

Corollary~\ref{cor_Dirac_bi}, Propositions~\ref{same_prop} and \ref{inf_cor}
together with \eqref{cor_cristian} now yield our main theorem.
\begin{theorem}\label{main1}
  Let $G\rr M$ be a Lie groupoid with Lie algebroid $A$. Then the
  construction in Theorem~\ref{Manin_pair} defines a one-to-one
  correspondence between multiplicative Dirac structures on $G\rr M$
  and equivalence classes of Dirac bialgebroids over $A$.
\end{theorem}
Hence, it does make sense to say that a Dirac bialgebroid
$(A,U,\iota)$ is \emph{integrable} if the Lie algebroid $A$ integrates
to a source-simply connected Lie groupoid.

Section~\ref{examples} shows how the infinitesimal descriptions of
Poisson groupoids, presymplectic groupoids and multiplicative
distributions on Lie groupoids are special cases of Theorem
\ref{main1}.

\subsection{Concrete integration recipe}\label{recipe}
For the convenience of the reader, let us quickly sketch how a
Dirac groupoid is reconstructed from an integrable Dirac bialgebroid
$(A,U,\iota)$.
\begin{enumerate}
\item[Step 1:] Let $(G\rr , M)$ be the source-simply connected Lie
  groupoid integrating the Lie algebroid $A$ \cite{CrFe03} and
  identify $U$ with the Lie algebroid $U\simeq\iota(U)\subseteq
  TM\oplus A^*$ with anchor $\pr_{TM}$.
\item[Step 2:] Extend the bracket on $\Gamma(U)$ to a dull bracket on
  $TM\oplus A^*$, i.e.~using the Leibniz identity, but possibly
  loosing the skew-symmetry and the Jacobi identity. This dull bracket
  is dual to a Dorfman connection $\Delta\colon \Gamma(TM\oplus
  A^*)\times\Gamma(A\oplus T^*M)\to\Gamma(A\oplus T^*M)$.
\item[Step 3:] Take the double vector subbundle $(\mathsf
  D_A;U,A;M)\subseteq (TA\oplus T^*A;TM\oplus A^*,A;M)$ that is
  defined by the triple $(U,U^\circ,\Delta)$, hence with core
  $U^\circ\subseteq A\oplus T^*M$.
\item[Step 4:] The space $\mathsf D_A$ is an LA-Dirac structure on
  $A$, which integrates by the results in \cite{Ortiz13} to a
  multiplicative Dirac structure $\mathsf D$ on $G\rr M$.
\end{enumerate}

\section{Examples}\label{examples}

This section describes the $A$-Manin pairs and Dirac bialgebroids in
the special cases of Dirac Lie groups, Poisson groupoids, presymplectic groupoids,
and multiplicative distributions on Lie groupoids.
Of course, the obtained infinitesimal descriptions show some
redundancy, as Lie bialgebroids, IM-$2$-forms and infinitesimal ideal
systems are very special cases of our general notion of Dirac
bialgebroids.  For instance, the Lie bialgebroid encoding a Poisson
groupoid is now replaced by the pair of dual algebroids together with
an inclusion (not the trivial one) of $A^*$ in $TM\oplus A^*$.  On the
other hand, the examples show that Lie bialgebroids, IM-2-forms and
infinitesimal ideal systems are the corner cases of Dirac
bialgebroids, just as Poisson structures, presymplectic structures and
distributions are the corner cases of Dirac structures.

\subsection{Dirac Lie groups}
Let $G$ be a Lie group with Lie algebra $\lie g$.  The papers
\cite{Ortiz08,Jotz11} show that multiplicative Dirac structures on $G$
correspond to pairs $(\lie i, \delta)$ where $\lie i\subseteq \lie g$
is an ideal and $\delta\colon\lie g\to\bigwedge^2\lie g/\lie i$ is a
Lie algebra $1$-cocycle\footnote{The space $\bigwedge^2\lie g/\lie i$
  is a $\lie g$-module via the action
\[ x\cdot \left((y+\lie i)\wedge(z+\lie i)\right)=([x,y]+\lie i)\wedge (z+\lie i)+(y+\lie i)\wedge([x,z]+\lie i).
\]}, such that the dual $\delta^t\colon \lie i^\circ\times\lie i^\circ\to
\lie g^*$ defines a Lie bracket on $\lie i^\circ$.  Equivalently, the
pair $(\lie g/\lie i, \lie i^\circ)$ is a Lie bialgebra. There is then
a quadratic Lie algebra structure on $\lie g/\lie i\times\lie i^\circ$
\cite[Theorem 1.12]{LuWe90}, or in other words $\lie g/\lie
i\times\lie i^\circ$ is a Courant algebroid over a point.  The pair
$(\lie g/\lie i\times\lie i^\circ, \lie i^\circ)$ is a $\lie g$-Manin
pair and $(\lie g, \lie i^\circ, \iota:\lie i^\circ\hookrightarrow\lie
g^*)$ is then a Dirac bialgebroid.

Conversely, note that a Dirac bialgebroid over a point is a triple $(\lie g, \lie
p, \iota\colon \lie p\hookrightarrow \lie g^*)$ with $\lie g, \lie p$
Lie algebras and $\iota$ an injective vector space morphism such that
there exists a quadratic Lie algebra $\lie m$ with a Lie algebra
morphism $\Phi\colon \lie g\to \lie m$ such that
\begin{enumerate}
\item $\Phi$ has isotropic image,
\item $\lie p$ is a Lagrangian subalgebra of $\lie m$,
\item $\langle \Phi(x), \xi\rangle=\langle \iota(x), \xi\rangle$ for $x\in\lie g$ and $\xi\in\lie p$,
\item $\Phi(\lie g)+\lie p=\lie m$.
\end{enumerate}
We let the reader verify that $\lie p^\circ\subseteq \lie g$ must be
an ideal and $(\lie g/\lie p^\circ, \lie p)$ a Lie bialgebra.  Hence, simply
connected Dirac Lie groups are in one-to-one correspondence with Dirac
bialgebroids over points, which we call \textbf{Dirac bialgebras}.

\subsection{Poisson groupoids}\label{ex_Poisson}
The infinitesimal description of a Poisson groupoid $(G\rr M, \pi)$ is
known to be its Lie bialgebroid $(A,A^*)$ \cite{MaXu94,MaXu00}.  In
that case, our Dirac bialgebroid is
$(A,A^*,(\rho_\star,\operatorname{id}_{A^*})\colon A^*\to TM\oplus
A^*)$. That is, $A^*$ is identified with the graph of its anchor
$\rho_\star$.  The $A$-Manin pair defining this Dirac bialgebroid is
$(A\oplus A^*, A^*)$, where $A\oplus A^*\to M$ is the Courant
algebroid defined by the Lie bialgebroid \cite{LiWeXu97}.  (Note that
we have already shown in \cite{Jotz13b} that the Courant algebroid
$\mathsf B$ defined by the Dirac groupoid $(G\rr M, \pi)$ is, up to an
isomorphism, $A\oplus A^*$.) The map $\Phi\colon A\oplus T^*M\to
A\oplus A^*$ is $(a,\theta)\mapsto
(a+\rho_\star^t\theta,\rho^t\theta)$.

The double $A\oplus A^*$ being a Courant algebroid is equivalent to
$(A,A^*)$ being a Lie bialgebroid.  By definition of the Courant
bracket, the subbundle $A^*$ is then automatically a Dirac structure
in $A\oplus A^*$. The surjectivity and pairing conditions for the
$A$-Manin pair are then obvious. The compatibility of $\Phi$ with
pairings and anchors follows easily from the equality
$\rho\circ\rho_\star^t=-\rho_\star\circ\rho^t$, which is always
satisfied for Lie bialgebroids.  We only have to check that the map
$\Phi$ preserves the brackets. First, for $a_1,a_2\in\Gamma(A)$,
\begin{align*} \lb \Phi(a_1,0), \Phi(a_2,0)\rb_{A\oplus
      A^*} &=\lb (a_1,0), (a_2,0)\rb_{A\oplus
      A^*}=([a_1,a_2]_A,0)\\
&=\Phi([(a_1,0),(a_2,0)]_d).
  \end{align*} We then have on the one hand for
  $\theta\in\Omega^1(M)$:
  \begin{align*} \lb \Phi(a,0), \Phi(0,\theta)\rb_{A\oplus A^*}&=\lb
    (a,0), (\rho_\star^t\theta,\rho^t\theta)\rb_{A\oplus A^*}=([a,
    \rho_\star^t\theta]_A-\ip{\rho^t\theta}\dr_{A^*}a,
    \ldr{a}\rho^t\theta) .\end{align*} On the other hand, we
  have \[\Phi([(a,0), (0,\theta)]_d)=\Phi(0,\ldr{\rho(a)}\theta)=(\rho_\star^t\ldr{a}\theta,
  \rho^t\ldr{\rho(a)}\theta).\] The equation
  $\rho^t\ldr{\rho(a)}\theta=\ldr{a}\rho^t\theta$ is verified using
  the definitions of $\ldr{a}$ and $\ldr{\rho(a)}$ and the
  compatibility of the bracket with the anchor. Then
  $\rho_\star^t\ldr{a}\theta=[a,
  \rho_\star^t\theta]_A-\ip{\rho^t\theta}\dr_{A^*}a$ for all
  $\theta\in\Omega^1(M)$ and $a\in\Gamma(A)$ is equivalent to Equation
  (11) in \cite[Lemma 12.1.8]{Mackenzie05}, hence always satisfied if
  $(A,A^*)$ is a Lie bialgebroid. Finally we have \begin{align*} &\lb
    \Phi(0,\theta), \Phi(0,\omega)\rb_{A\oplus A^*}
    =\lb (\rho_\star^t\theta,\rho^t\theta), (\rho_\star^t\omega,\rho^t\omega)\rb_{A\oplus A^*}\\
    =\,&\left([\rho_\star^t\theta,\rho_\star^t\omega]_A+\ldr{\rho^t\theta}\rho_\star^t\omega
      -\ip{\rho^t\omega}\dr_{A^*}\rho_\star^t\theta,
      [\rho^t\theta,\rho^t\omega]_{A^*}+\ldr{\rho_\star^t\theta}\rho^t\omega
      -\ip{\rho_\star^t\omega}\dr_{A}\rho^t\theta\right)\end{align*}
  and on the other hand $\Phi([(0,\theta),(0,\omega)]_d)=\Phi(0)=0$.  Hence, by symmetry of the Lie bialgebroid
  $(A,A^*)$, we only have to check that $
  [\rho_\star^t\theta,\rho_\star^t\omega]_A=-\ldr{\rho^t\theta}\rho_\star^t\omega
  +\ip{\rho^t\omega}\dr_{A^*}\rho_\star^t\theta $ for all
  $\theta,\omega\in\Omega^1(M)$. If $\omega=\dr f$ with some $f\in
  C^\infty(M)$, this is \cite[Lemma 12.1.5]{Mackenzie05}. To see this,
  note that
  $\rho^t_\star\dr((\rho\circ\rho_\star)^t\theta(f))=-\ldr{\rho^t\theta}(\dr_\star
  f)$.  The general case (with $\omega$ not necessarily exact) follows
  easily.

  Let us now quickly describe a Dorfman connection adapted to the
  Dirac algebroid $(A, A(D_\pi))$. Via the canonical isomorphism
  $TA\oplus T^*A\simeq A(TG\oplus T^*G)$, the Lie algebroid of the
  Dirac structure $D_\pi$ is $D_{\pi_A}$, where $\pi_A$ is the linear
  Poisson structure on $A$ that is equivalent to the Lie algebroid
  structure on $A^*$ \cite{MaXu00}. The sides of $D_{\pi_A}$ are $A$
  and $U=\operatorname{graph}(\rho_\star\colon A^*\to TM)$ and its
  core is $K=\operatorname{graph}(-\rho^t_\star\colon T^*M\to A)$.
  Choose any connection
  $\nabla\colon\mx(M)\times\Gamma(A)\to\Gamma(A)$, denote by
  $\nabla^*\colon\mx(M)\times\Gamma(A^*)\to\Gamma(A^*)$ the dual
  connection and construct the basic $A^*$-connection on $A^*$:
  $\nabla^{*\rm bas}\colon\Gamma(A^*)\times\Gamma(A^*)\to\Gamma(A^*)$,
  $\nabla^{*\rm
    bas}_{\alpha_1}\alpha_2=\nabla_{\rho_\star\alpha_2}\alpha_1+[\alpha_1,\alpha_2]_{A^*}$
  for all $\alpha_1,\alpha_2\in\Gamma(A^*)$.  The Dorfman connection
  $\Delta\colon \Gamma(TM\oplus A^*)\times\Gamma(A\oplus T^*M)\to
  \Gamma(A\oplus T^*M)$ defined by
\[\Delta_{(X,\alpha)}(a,\theta)=\left(\langle a, \nabla_\cdot^{*\rm
    bas}\alpha\rangle+\nabla_Xa-\rho_\star^t\langle \nabla^*_\cdot \alpha,
  a\rangle,
\ldr{X}\theta+ \langle \nabla^*_\cdot \alpha, a\rangle
\right)
\]
is adapted to $D_{\pi_A}$.
The corresponding dull bracket $\lb\cdot\,,\cdot\rb_\Delta$ on sections of $TM\oplus A^*$ is 
given by
\begin{multline}
\lb(X_1,\alpha_1), (X_2,\alpha_2)\rb_\Delta\\
=\left([X_1, X_2],
  \nabla_{X_1}\alpha_2-\nabla_{X_2}\alpha_1+\nabla_{\rho_\star(\alpha_2)}\alpha_1-\nabla_{\rho_\star(\alpha_1)}\alpha_2+[\alpha_1,\alpha_2]_{A^*},
\right)
\end{multline}
which restricts to
$ \lb (\rho_\star(\alpha_1), \alpha_1),  (\rho_\star(\alpha_2),
\alpha_2)\rb_\Delta=(\rho_\star[\alpha_1,\alpha_2]_{A^*}, [\alpha_1,\alpha_2]_{A^*})$
on sections of $U$. For more details, see \cite{Jotz13a}.

\subsection{Presymplectic groupoids}\label{ex_presymplectic}
Presymplectic groupoids are described infinitesimally in
\cite{BuCrWeZh04, BuCaOr09} by IM-2-forms:
  Let $(A\to M,[\cdot\,,\cdot],\rho)$ be a Lie algebroid. A vector
  bundle map $\sigma\colon A\to T^*M$ is an IM-2-form if
\begin{enumerate}
\item $\langle \rho(a_1),\sigma(a_2)\rangle=-\langle\rho(a_2),\sigma(a_1)\rangle$ and 
\item $\sigma[a_1,a_2]=\ldr{\rho(a_1)}\sigma(a_2)-\ip{\rho(a_2)}\dr\sigma(a_1)$
\end{enumerate}  
for all $a_1,a_2\in\Gamma(A)$.

We show that the Dirac bialgebroid corresponding to this example is
\linebreak
$(A,TM, (\operatorname{id}_{TM},\sigma^t)\colon TM\to TM\oplus A^*)$.
Consider a vector bundle morphism $\sigma\colon A\to T^*M$ on a Lie
algebroid $A\to M$. We want to define an $A$ Manin-pair
$(TM\oplus T^*M, TM)$, where $TM\oplus T^*M$ is the standard Courant
algebroid (see also \cite{Jotz13b}, where we check that $\mathsf B$ is
in this case isomorphic to $TM\oplus T^*M$), $TM$ is seen as
isomorphic to
$\operatorname{graph}(\sigma^t\colon TM\to A^*)\subseteq TM\oplus A^*$
and the map $\Phi\colon A\oplus T^*M\to TM\oplus T^*M$ is
$(a,\theta)\mapsto (\rho(a), \sigma(a)+\theta)$. We quickly check that
$\Phi$ is a morphism of degenerate Courant algebroids if and only if
$\sigma$ is an IM-2-form.

The map $\Phi$ is obviously compatible with the anchors. Choose
$(a_1,\theta_1),(a_2,\theta_2)\in\Gamma(A\oplus T^*M)$.  Then
  \begin{align*}
    \langle\Phi(a_1,\theta_1),\Phi(a_2,\theta_2)\rangle&=\langle(\rho(a_1),
    \sigma(a_1)+\theta_1), (\rho(a_2), \sigma(a_2)+\theta_2)\rangle\\
    &=\langle (a_1,\theta_1),
    (a_2,\theta_2)\rangle_d+\langle\sigma(a_1),\rho(a_2)\rangle+\langle\sigma(a_2),\rho(a_1)\rangle
\end{align*}
and 
\begin{align*}
&\lb \Phi(a_1,\theta_1), \Phi(a_2,\theta_2)\rb=\lb (\rho(a_1),
  \sigma(a_1)+\theta_1), (\rho(a_2), \sigma(a_2)+\theta_2)\rb\\
&=(\rho[a_1,a_2],\ldr{\rho(a_1)}(\sigma(a_2)+\theta_2)-\ip{\rho(a_2)}\dr(\sigma(a_1)+\theta_1))\\
&=\Phi[(a_1,\theta_1), (a_2,\theta_2)]_d+(0, \ldr{\rho(a_1)}\sigma(a_2)-\ip{\rho(a_2)}\dr\sigma(a_1)-\sigma[a_1,a_2]).
\end{align*}
Hence, $\Phi$ is a morphism of Courant algebroids if and only if
$\langle \rho(a_1),\sigma(a_2)\rangle=-\langle\rho(a_2),\sigma(a_1)\rangle$
and $\sigma[a_1,a_2]=\ldr{\rho(a_1)}\sigma(a_2)-\ip{\rho(a_2)}\dr\sigma(a_1)$
for all $a_1,a_2\in\Gamma(A)$.

The Dirac algebroid corresponding to a presymplectic groupoid $(G\rr
M, \omega)$ is $(A,A(D_\omega))\simeq (A, D_{\sigma^*\omega_{\rm
    can}})$, where $\omega_{\rm can}$ is the canonical symplectic
structure on $T^*M$. The vector bundles $U$ and $K$ over $M$
are $U=\operatorname{graph}(-\sigma^t\colon TM\to A^*)$ and 
$K=\operatorname{graph}(\sigma\colon A\to T^*M)$. The Dorfman 
connection $\Delta\colon  \Gamma(TM\oplus A^*)\times\Gamma(A\oplus T^*M)\to
\Gamma(A\oplus T^*M)$
defined by
\[\Delta_{(X,\alpha)}(a,\theta)=(\nabla_Xa,
\ldr{X}(\theta-\sigma(a))+\langle \nabla_\cdot^*(\sigma^tX+\alpha), a\rangle
+\sigma(\nabla_Xa)). 
\]
The dual dull bracket $\lb\cdot\,,\cdot\rb_\Delta$ on sections
of $TM\oplus A^*$ is given by 
\[\lb (X_1,\alpha_1), (X_2,\alpha_2)\rb_\Delta=([X_1,X_2], \nabla_{X_1}^*(\alpha_2+\sigma^tX_2)-\nabla_{X_2}^*(\alpha_1+\sigma^tX_1)-\sigma^t[X_1,X_2])
\]
and restricts to $ \lb (X_1,-\sigma^tX_1), (X_2,-\sigma^tX_2)\rb_\Delta=([X_1,X_2],
-\sigma^t[X_1, X_2])$
on sections of $U$.

\subsection{Multiplicative distributions}\label{ex_iis}
Multiplicative distributions on Lie groupoids are described
infinitesimally in \cite{Hawkins08,JoOr14} via infinitesimal ideal
systems:
\begin{definition}\label{def_inf_id_system}\cite{JoOr14}
  Let $(q\colon A\to M, \rho,[\cdot\,,\cdot])$ be a Lie algebroid,
  $F_M\subseteq TM$ an involutive subbundle, $J\subseteq A$ a
  subbundle such that $\rho(J)\subseteq F_M$ and $\nabla$ a flat
  partial $F_M$-connection on $A/J$ with the following
  properties\footnote{We say by abuse of notation that a section
    $a\in\Gamma(A)$ is $\nabla$-parallel if its class $\bar a$ in
    $\Gamma(A/J)$ is $\nabla$-parallel.}:
\begin{enumerate}
\item If $a\in\Gamma(A)$ is $\nabla$-parallel, then $[a,j]\in\Gamma(J)$ 
for all $j\in\Gamma(J)$.
\item If $a_1,a_2\in\Gamma(A)$ are $\nabla$-parallel, then $[a_1,a_2]$ is also
  $\nabla$-parallel.
\item If $a\in\Gamma(A)$ is  $\nabla$-parallel, then $\rho(a)$ is
  $\nabla^{F_M}$-parallel, where
  \[\nabla^{F_M}\colon \Gamma(F_M)\times\Gamma(TM/F_M)\to\Gamma(TM/F_M)\]
  is the Bott connection associated to $F_M$.
\end{enumerate}
The triple $(F_M,J,\nabla)$ is an infinitesimal ideal
system in $A$. 
\end{definition}
We prove the following theorem.
\begin{theorem}
  Let $(q\colon A\to M, \rho,[\cdot\,,\cdot])$ be a Lie algebroid,
  $F_M\subseteq TM$ an involutive subbundle, $J\subseteq A$ a
  subbundle such that $\rho(J)\subseteq F_M$ and $\tilde \nabla\colon
  \mx(M)\times\Gamma(A)\to\Gamma(A)$ be a connection such that
  $\tilde\nabla_Xj\in\Gamma(J)$ for all $j\in\Gamma(J)$ and
  $X\in\Gamma(F_M)$, and which thus defines a connection $\nabla\colon
  \Gamma(F_M)\times\Gamma(A/J)\to\Gamma(A/J)$.  Then
  $(A,F_M,J,\nabla)$ is an infinitesimal ideal system if and only if
  $(A, F_M\oplus J^\circ,\iota)$ is a Dirac bialgebroid, where
  $\iota\colon F_M\oplus J^\circ\hookrightarrow TM\oplus A^*$ is the
  inclusion and $F_M\oplus J^\circ$ has the anchor
  $\pr_{TM}$ and the bracket
\[[(X_1,\alpha_1),(X_2,\alpha_2)]_{F_M\oplus J^\circ}=([X_1,X_2], \tilde\nabla^*_{X_1}\alpha_2-\tilde\nabla^*_{X_2}\alpha_1),
\]
$X_1,X_2\in\Gamma(F_M)$, $\alpha_1,\alpha_2\in\Gamma(J^\circ)$.
\end{theorem}
First note that $F_M\oplus J^\circ$ is a Lie algebroid with this
structure if and only if $\nabla$ is flat.  To prove this theorem, we
construct a Manin pair associated to the infinitesimal ideal
system.

\medskip We have shown in \cite{JoOr14} that if the quotients are
smooth and $\nabla$ has no holonomy, there is an induced Lie algebroid
structure on $(A/J)/\nabla\to M/F_M$ such that the canonical
projection $A\to (A/J)/\nabla$ over $M\to M/F_M$ is a fibration of Lie
algebroids. In the case of a Lie algebra $\lie g$, an ideal system is
just an ideal $\lie i$, and this quotient is the usual quotient of the
Lie algebra by its ideal $\lie i$.  We show that there is an
alternative construction of a quotient Lie algebroid, that simplifies
as well to the usual quotient $\lie g/\lie i$ in the Lie algebra case.

The paper \cite{DrJoOr15} shows that an infinitesimal ideal system can
alternatively be defined as follows.  Let $A$ be a Lie algebroid over
the base $M$, $J$ a subbundle of $A$ and $F_M$ an involutive subbundle
of $TM$ such that $\rho(J)\subseteq F_M$. Let $\tilde\nabla\colon
\mx(M)\times\Gamma(A)\to\Gamma(A)$ be a connection such that
$\tilde\nabla_Xj\in\Gamma(J)$ for all $j\in\Gamma(J)$ and
$X\in\Gamma(F_M)$, and which thus defines a connection $\nabla\colon
\Gamma(F_M)\times\Gamma(A/J)\to\Gamma(A/J)$.\footnote{Conversely,
  given $\nabla$, there always exists a connection $\tilde\nabla$
  projecting in this manner to $\nabla$. We say that $\tilde\nabla$
  is an \emph{extension} of $\nabla$.}  The triple $(F_M,J,\nabla)$ is
an infinitesimal ideal system if
\begin{enumerate}
\item $\nabla$ is flat,
\item $\nabla^{\rm bas}_aX\in\Gamma(F_M)$,
\item $\nabla^{\rm bas}_aj\in\Gamma(J)$ and
\item $R_{\tilde\nabla}^{\rm bas}(a_1,a_2)(X)\in\Gamma(F_M)$
\end{enumerate}
for all $X\in\Gamma(F_M)$, $a_1,a_2\in\Gamma(A)$ and $j\in\Gamma(J)$,
where the two connections $\nabla^{\rm bas}$ and the tensor
$R_{\tilde\nabla}^{\rm bas}$ are defined by $\tilde\nabla$ as in
Example~\ref{tangent_double_2_reps}.

Now consider the direct sum $F_M\oplus A$ of vector bundles over
$M$. Since the anchor $\rho$ restricts to a map $\rho\an{J}\colon J\to
F_M$, the vector bundle
\[\bar A:=\frac{F_M\oplus A}{\operatorname{graph}(-\rho\an{J})}\to M,\]
inherits the anchor $\bar \rho(X\oplus a)=X+\rho(a)$. (Given
$X\in\Gamma(F_M)$ and $a\in\Gamma(A)$, we write $X\oplus a$ for the
class of $(X,a)$ in $\bar A$.)  Define $[\cdot\,,\cdot]_{\bar
  A}\colon\Gamma(\bar A)\times\Gamma(\bar A)\to\Gamma(\bar A)$ by
\begin{equation*} [X_1\oplus a_1, X_2\oplus a_2]_{\bar A}=([X_1,X_2]+\nabla_{a_1}^{\rm
    bas}X_2-\nabla^{\rm
    bas}_{a_2}X_1)\oplus([a_1,a_2]+\tilde\nabla_{X_1}a_2-\tilde\nabla_{X_2}a_1).
\end{equation*}
This is skew-symmetric and well-defined because $[X\oplus a,
(-\rho(j))\oplus j]_{\bar A}=(-\rho(\tilde\nabla_Xj+\nabla_a^{\rm
  bas}j))\oplus(\tilde\nabla_Xj+\nabla_a^{\rm bas}j)$ and
$\tilde\nabla_Xj+\nabla_a^{\rm bas}j\in\Gamma(J)$ for all
$j\in\Gamma(J)$.  It is easy to check that it is compatible with the
anchor, i.e.~that is satisfies the Leibniz identity.  To see that the Jacobi
identity is satisfied, one can check that
\begin{align*}
  &[[ X_1\oplus a_1, X_2\oplus a_2]_{\bar A}, X_3\oplus a_3]_{\bar
    A}+[[ X_2\oplus a_2, X_3\oplus a_3]_{\bar A}, X_1\oplus a_1
  ]_{\bar A}\\
  &\hspace*{5cm}+[[ X_3\oplus a_3, X_1\oplus a_1]_{\bar A}, X_2\oplus
  a_2]_{\bar A}=(-\rho(j))\oplus j,
\end{align*}
with
\begin{align*}
  j:=\left[R_{\tilde\nabla}^{\rm
      bas}(a_1,a_2)X_3-R_{\tilde\nabla}(X_1, X_2) a_3\right]+\rm
  c.p. \in\Gamma(J).
\end{align*}
The maps $ A\to \bar A$, $a\mapsto 0\oplus a$ and $F_M\to \bar A$,
$X\mapsto X\oplus 0$ are easily seen to be morphisms of Lie algebroids
over the identity.  Note that the second map is injective.  We let the
reader complete the proof of the following theorem. 
\begin{theorem}
Let $(q\colon A\to M, \rho,[\cdot\,,\cdot])$ be a Lie algebroid,
  $F_M\subseteq TM$ an involutive subbundle, $J\subseteq A$ a
  subbundle and $\nabla$ a
  partial $F_M$-connection on $A/J$. Consider any extension
$\tilde\nabla:\mx(M)\times\Gamma(A)\to\Gamma(A)$ of $\nabla$.
Then $(A,F_M,J^\circ,\nabla)$ is an infinitesimal ideal system if and only if 
\[\bar A=\frac{F_M\oplus A}{\operatorname{graph}(-\rho\an{J})}\to M\]
with the anchor $\bar \rho(X\oplus a)=X+\rho(a)$ and the bracket
$[X_1\oplus a_1, X_2\oplus a_2]_{\bar A}=([X_1,X_2]+\nabla_{a_1}^{\rm
  bas}X_2-\nabla^{\rm
  bas}_{a_2}X_1)\oplus([a_1,a_2]+\tilde\nabla_{X_1}a_2-\tilde\nabla_{X_2}a_1)$
is a Lie algebroid, which does not
depend on the choice of the extension $\tilde\nabla$ of the connection $\nabla$.
\end{theorem}

Consider now the Courant algebroid $\bar A\oplus \bar A^*$ defined by
the trivial Lie bialgebroid $(\bar A, \bar A^*)$, i.e.~with the
trivial Lie algebroid structure on $\bar A^*$. The vector bundle $\bar
A^*$ can be described as follows: sections of $\bar A^*$ are pairs
$(\bar\theta,\rho^t\theta+\alpha)\in\Gamma(F_M^*\oplus A^*)$ with
$\theta\in\Omega^1(M)$ and $\alpha\in\Gamma(J^\circ)$. Here,
$\bar\theta$ is the class of $\theta$ in $T^*M/F_M^\circ\simeq F_M^*$.
Hence, there is a natural inclusion of $J^\circ$ in $\bar A^*$, and so
a natural inclusion $i$ of $F_M\oplus J^\circ$ in $\bar A\oplus \bar A^*$.
Elements of $i(F_M\oplus J^\circ)\subseteq \bar A\oplus \bar A^*$ can be
written $(X\oplus 0, (0,\alpha))$ with $X\in F_M$ and $\alpha\in J^\circ$.
It is easy to see that $F_M\oplus J^\circ$ is isotropic in $\bar
A\oplus \bar A^*$. To see that it is maximal isotropic, take $(X'\oplus
a,(\bar\theta, \rho^t\theta+\alpha'))\in \bar A\oplus \bar A^*$ such that
$\langle (X' \oplus a,(\bar\theta, \rho^t\theta+\alpha')), (X\oplus 0,
(0,\alpha))\rangle=0$ for all $(X\oplus 0, (0,\alpha))\in F_M\oplus
J^\circ$. Since this is $\theta(X)+\alpha(a)$, we find that $a\in J$ and
$\theta\in F_M^\circ$. Hence, $(X'\oplus a,(\bar\theta,
\rho^t\theta+\alpha'))= \bigl((X'+\rho(a))\oplus 0,(0,
\rho^t\theta+\alpha'))\in F_M\oplus J^\circ$ (recall that
$\rho(J)\subseteq F_M$ and so $\rho^t(F_M^\circ)\subseteq J^\circ$).
The Courant bracket of $(X_1\oplus0, (0,\alpha_1))$ and $(X_2\oplus0,
(0,\alpha_2))\in F_M\oplus J^\circ$ equals
\[\lb (X_1\oplus0, (0,\alpha_1)), (X_2\oplus0, (0,\alpha_2))\rb_{\bar A\oplus\bar
  A^*}=([X_1,X_2]\oplus0, (0,\tilde\nabla_{X_1}^*\alpha_2-\tilde\nabla^*_{X_2}\alpha_1)),
\]
which is again an element of $F_M\oplus J^\circ$. This shows that
$F_M\oplus J^\circ$ is a Dirac structure in $\bar A\oplus\bar A^*$. 

Next we define the map $\Phi\colon  A\oplus T^*M\to \bar A\oplus \bar A^*$,
$\Phi(a,\theta)=(0\oplus a, (\bar\theta,\rho^t\theta))$. The equality
$i(F_M\oplus J^\circ)+ \Phi(A\oplus T^*M)=\bar A\oplus \bar A^*$
is immediate, as well as the equality $\langle
i(X,\alpha),\Phi(a,\theta)\rangle=\langle (X\oplus0,(0,\alpha)),
(0\oplus a,(\bar\theta,\rho^t\theta))\rangle=\theta(X)+\alpha(a)$ for all
$(X,\alpha)\in F_M\oplus J^\circ$ and $(a,\theta)\in A\oplus T^*M$. To
see that $\Phi$ is a morphism of (degenerate) Courant algebroids, note
first that
\[\ldr{0\oplus a}(\bar\theta,\rho^t\theta)=\left(\overline{\ldr{\rho(a)}\theta}, \rho^t(\ldr{\rho(a)}\theta)\right)\]
and
\[\ip{0\oplus a}\dr_{\bar A}(\bar\theta,\rho^t\theta)=\left(\overline{\ip{\rho(a)}\dr\theta}, \rho^t(\ip{\rho(a)}\dr\theta)\right)\]
for $a\in\Gamma(A)$ and $\theta\in\Omega^1(M)$.  Then we can compute
\begin{align*}
  \lb \Phi(a,\theta), \Phi(b,\omega)\rb_{\bar A\oplus \bar A^*}&=\lb
  (0\oplus a, (\bar\theta,\rho^t\theta)), (0\oplus b,
  (\bar\omega,\rho^t\omega))\rb_{\bar A\oplus \bar A^*}\\
  &=\left(0\oplus[a,b],(\overline{\ldr{\rho(a)}\omega-\ip{\rho(b)}\dr\theta}, \rho^t(\ldr{\rho(a)}\omega-\ip{\rho(b)}\dr\theta))\right)\\
  &=\Phi([(a,\theta),(b,\omega)]_d).
\end{align*}
The compatibility of $\Phi$ with the anchors is immediate and
compatibility with the pairing is checked as follows:
\begin{align*}
  \langle \Phi(a,\theta), \Phi(b,\omega)\rangle
  &=\langle (0\oplus a, (\rho^t\theta,\bar\theta)), (0\oplus b, (\rho^t\omega,\bar\omega))\rangle=\theta(\rho(b))+\omega(\rho(a))\\
  &=\langle (a,\theta), (\rho,\rho^t)(b,\omega)\rangle=\langle
  (a,\theta), (b,\omega)\rangle_d.
\end{align*}
We have hence shown that $(\bar A\oplus\bar A^*, i(F_M\oplus J^\circ))$ is an
$A$-Manin pair.

The Dirac algebroid corresponding to a Lie groupoid with a
multiplicative distribution $(G\rr M, F\oplus F^\circ)$ is $(A, A(
F\oplus F^\circ))\simeq (A, F_A\oplus F_A^\circ)$, with $F_A\simeq
A(F)$. This double vector subbundle of $TA$ has sides $A$ and
$F_M\subseteq TM$, and core $J\subseteq A$. Choose any connection
$\nabla\colon \mx(M)\times\Gamma(A)\to\Gamma(A)$ that is adapted to
$F_A$,
i.e.~$T_ma(v_m)-\left.\frac{d}{dt}\right\an{t=0}a(m)+t\nabla_{v_m}a\in
F_A(a(m))$ for all $v_m\in F_M$ and $a\in\Gamma(A)$ (see
\cite{JoOr14}), and consequently $\nabla_Xj\in\Gamma(J)$ for all
$X\in\Gamma(F_M)$ and $j\in\Gamma(J)$.  Then the Dorfman connection
$\Delta\colon\Gamma(TM\oplus A^*)\times\Gamma(A\oplus
T^*M)\to\Gamma(A\oplus T^*M)$,
\[\Delta_{(X,\alpha)}(a,\theta)=(\nabla_Xa, \ldr{X}\theta+\langle
\nabla_\cdot^*\alpha, a\rangle)
\]
is adapted to $F_A\oplus F_A^\circ$ and the dual dull bracket given by 
$\lb (X_1,\alpha_1), (X_2,\alpha_2)\rb=([X_1,X_2], \nabla^*_{X_1}\alpha_2-\nabla^*_{X_2}\alpha_1)
$
restricts to the Lie algebroid bracket found above on sections of $F_M\oplus J^\circ$.

\appendix

\section{Proof of Theorem~\ref{Courant_algebroid}}\label{proof_courant}
Note that the equality 
\begin{equation}\label{intertwine_bas}
\nabla_a^{\rm bas}\circ(\rho,\rho^t)=(\rho,\rho^t)\circ\nabla_a^{\rm bas}
\end{equation}
is easily verified. 
The connection $\nabla_a^{\rm
  bas}\colon\Gamma(A)\times\Gamma(A\oplus T^*M)\to \Gamma(A\oplus
T^*M)$ defines a connection $\nabla_a^{\rm bas}\colon\Gamma(A\oplus
T^*M)\times\Gamma(A\oplus T^*M)\to \Gamma(A\oplus T^*M)$, $\nabla^{\rm
  bas}_\tau\tau'=\nabla_{\pr_A(\tau)}^{\rm bas}\tau'$. Recall also
\eqref{useful}
$\Delta_{(X,\alpha)}(a,\theta)=\Delta_u(a,0)+(0,\ldr{X}\theta)$ for
all $(X,\alpha)\in\Gamma(TM\oplus A^*)$ and
$(a,\theta)\in\Gamma(A\oplus T^*M)$.
First we prove a few lemmas.
\begin{lemma}\label{basic_like_eq}
The equality
\begin{equation}\label{basicLike_eqq}
[\tau_1,\tau_2]_d=\Delta_{(\rho,\rho^t)(\tau_1)}\tau_2-\nabla_{\tau_2}^{\rm bas}\tau_1
\end{equation}
holds for all $\tau_1,\tau_2\in\Gamma(A\oplus T^*M)$.
\end{lemma}

\begin{proof}Write $\tau_1=(a_1,\theta_1)$ and $\tau_2=(a_2,\theta_2)\in\Gamma(A\oplus T^*M)$.
Then:
\begin{align*}
\nabla_{\tau_2}^{\rm bas}\tau_1&=\nabla_{a_2}^{\rm bas}\tau_1=\Omega_{(\rho,\rho^t)\tau_1}a_2+\ldr{a_2}\tau_1\\
&=\Delta_{(\rho,\rho^t)\tau_1}\tau_2-(0,\dr\langle\theta_1,\rho(a_2)\rangle)-(0,\ldr{\rho(a_1)}\theta_2)+([a_2,a_1], \ldr{\rho(a_2)}\theta_1)\\
&=\Delta_{(\rho,\rho^t)\tau_1}\tau_2-[\tau_1, \tau_2]_d.\qedhere
\end{align*}
\end{proof}
\begin{lemma}\label{complicated_eq}
  Let $A$ be a Lie algebroid and choose a Dorfman connection
  $\Delta\colon \Gamma(TM\oplus A^*)\times\Gamma(A\oplus
  T^*M)\to\Gamma(A\oplus T^*M)$.  Then, for all $\nu\in\Gamma(TM\oplus
  A^*)$ and $\tau,\tau'\in\Gamma(A\oplus T^*M)$:
\begin{equation}\label{complicated_eqq}
  \langle (\rho,\rho^t)\Delta_{\nu}\tau-\lb \nu,(\rho,\rho^t)\tau\rb_\Delta-\nabla_{\tau}^{\rm bas}\nu,\tau'\rangle=\langle\nabla_{\tau'}^{\rm bas}\nu, \tau\rangle.
\end{equation}
\end{lemma}
This yields the following corollary.
\begin{corollary}\label{eq_for_morphism}
  Assume that $(U,K,\Delta)$ defines an LA-Dirac structure in $TA\oplus
  T^*A$. Then, for all $u\in\Gamma(U)$
and $k\in\Gamma(K)$:
\begin{equation*}
(\rho,\rho^t)\Delta_{u}k=\lb u,(\rho,\rho^t)k\rb_\Delta+\nabla_{k}^{\rm bas}u.
\end{equation*}
\end{corollary}
\begin{proof}
By Lemma~\ref{complicated_eq}, we have 
\[\langle (\rho,\rho^t)\Delta_{u}k-\lb u,(\rho,\rho^t)k\rb_\Delta
-\nabla_{k}^{\rm bas}u,\tau\rangle=\langle\nabla_{\tau}^{\rm bas}u, k\rangle\]
for all $\tau\in\Gamma(A\oplus T^*M)$.  Since $\nabla_\tau^{\rm
  bas}$ preserves $\Gamma(U)$ by Theorem~\ref{morphic_Dirac_triples},
the right-hand side vanishes.
\end{proof}

\begin{proof}[Proof of Lemma~\ref{complicated_eq}]

We write $\tau=(a,\theta)$ and $\nu=(X,\alpha)$. Then for any
$\tau'=(a',\theta')\in\Gamma(A\oplus T^*M)$ we have
\begin{align*}
&\langle (\rho,\rho^t)\Delta_{\nu}\tau-\lb \nu,(\rho,\rho^t)\tau\rb_\Delta-\nabla_{a}^{\rm bas}\nu, \tau'\rangle-\langle\nabla_{a'}^{\rm bas}\nu, \tau\rangle\\
\overset{\eqref{not_dual}}=&\langle \Delta_\nu\tau-\Omega_\nu a, (\rho,\rho^t)\tau'\rangle
-\langle \ldr{a}\nu,\tau'\rangle+\langle \nu, \nabla_{a'}^{\rm bas}\tau\rangle
-\rho(a')\langle \nu, \tau\rangle+\lb 
(\rho,\rho^t)\tau, \nu\rb_\Delta, \tau'\rangle\\
=\,\,\,\,&\langle (0, \ldr{X}\theta+\dr\langle\alpha, a\rangle), (\rho,\rho^t)\tau'\rangle
-\langle \ldr{a}\nu,\tau'\rangle
-\rho(a')\langle \nu, \tau\rangle+\rho(a)\langle  \nu, \tau'\rangle\\
&\hspace*{7cm}-\langle \nu,
\Delta_{(\rho,\rho^t)\tau}\tau'- \nabla_{a'}^{\rm bas}\tau\rangle\\
\overset{\eqref{basicLike_eqq}}=&\langle \ldr{X}\theta, \rho(a')\rangle+\rho(a')\langle\alpha, a\rangle
+\langle  \nu, \ldr{a}\tau'\rangle -\rho(a')\langle \nu, \tau\rangle-\langle \nu,[\tau,\tau']_d\rangle\nonumber\\
=\,\,\,\,&\langle \ldr{X}\theta, \rho(a')\rangle-\rho(a')\langle\theta, X\rangle
+\cancel{\langle  \nu, \ldr{a}(a',\theta')\rangle}-\langle \nu,
(\cancel{[a,a'],\ldr{\rho(a)}\theta'}-\ip{\rho(a')}\dr\theta)\rangle\nonumber\\
=\,\,\,\,&\langle\ip{X}\dr\theta, \rho(a')\rangle+\langle
\ip{\rho(a')}\dr\theta, X\rangle=0.\qedhere
\end{align*}
\end{proof}

We also need the
following lemma, which is proved at the end of this section.

\begin{lemma}\label{difficult_computation}
Let $A\to M$ be a Lie algebroid and $(K,U,\Delta)$ an LA-Dirac
triple. Then  Condition (5) of Theorem 
\ref{morphic_Dirac_triples} is equivalent to:
 \begin{equation}\label{bialgebroid1}
\nabla^{\rm bas}_\tau\lb u, v\rb_\Delta-\lb \nabla^{\rm bas}_\tau u,
v\rb_\Delta-\lb u, \nabla^{\rm bas}_\tau v\rb_\Delta+\nabla^{\rm
  bas}_{\Delta_u\tau}v-\nabla^{\rm bas}_{\Delta_v\tau}u=-(\rho,\rho^t)R_\Delta(u,v)\tau
\end{equation}
for all $u,v\in\Gamma(U)$ and $\tau\in\Gamma(A\oplus T^*M)$.

The equality
\begin{equation}\label{bialgebroid2}
\begin{split}
&\Delta_u[\tau_1,\tau_2]_d-[\Delta_u\tau_1,\tau_2]_d-[\tau_1,\Delta_u\tau_2]_d
+\Delta_{\nabla_{a_1}^{\rm bas}u}\tau_2-\Delta_{\nabla_{a_2}^{\rm
    bas}u}\tau_1+(0,\dr\langle\tau_1,\nabla^{\rm bas}_{a_2}u\rangle)\\
=\,&-R_\Delta^{\rm bas}(a_1,a_2)u
\end{split}
\end{equation}
holds for all $\tau_1,\tau_2\in\Gamma(A\oplus T^*M)$ with
$\pr_A(\tau_i)=:a_i$ and all $u\in\Gamma(U)$.

\end{lemma}

\begin{proof}[Proof of Theorem~\ref{Courant_algebroid}]
We start by checking that $\lb\cdot\,,\cdot\rb$ is well-defined.
Choose $k\in\Gamma(K)$, $\tau\in\Gamma(A\oplus T^*M)$ with $\pr_A(k)=a$, $\pr_A(\tau)=a'$, and 
$u\in\Gamma(U)$. We have then 
\begin{align*}
\lb u\oplus\tau, (\rho,\rho^t)(-k)\oplus k\rb
=\,&\left(-\lb u,(\rho,\rho^t)k\rb_\Delta-\nabla_{a'}^{\rm bas}(\rho,\rho^t)k-\nabla_{a}^{\rm bas}u\right)\\
&\oplus\left( [\tau,k]_d+\Delta_{u}k+\Delta_{(\rho,\rho^t)k}\tau
-(0,\dr\langle \tau,(\rho,\rho^t)k\rangle)\right).
\end{align*}
Using Lemma~\ref{basic_like_eq}, we see immediately that the second
term of this sum equals $\nabla_{a'}^{\rm bas}k+\Delta_uk$.  
We verify that this is a section of $K$. By
Remark~\ref{remark_duality} or the considerations before
Theorem~\ref{morphic_Dirac_triples}, we know that $\nabla_{a'}^{\rm
  bas}k\in\Gamma(K)$, and 
$\Delta_uk\in\Gamma(K)$ follows from $K=U^\circ$ and the fact that
$\lb\cdot\,,\cdot\rb_\Delta$, which is dual to $\Delta$, preserves
sections of $U$.

Recall \eqref{intertwine_bas} and $(\rho,\rho^t)\Delta_{u}k=\lb
u,(\rho,\rho^t)k\rb_\Delta+\nabla_{a'}^{\rm bas}u$ by Lemma
\ref{eq_for_morphism}.  These yield
\[\lb u\oplus\tau, (\rho,\rho^t)(-k)\oplus k\rb=(-(\rho,\rho^t)(\nabla_a^{\rm bas}k+\Delta_uk))\oplus(\nabla_a^{\rm bas}k+\Delta_uk),
\]
which is $0$ in $C$.  The equality $\lb (\rho,\rho^t)(-k)\oplus k,
u\oplus\tau \rb=0$ follows with $\lb (\rho,\rho^t)(-k)\oplus k, u\oplus
\tau\rb_C+\lb u\oplus\tau, (\rho,\rho^t)(-k)\oplus k\rb_C =\mathcal
D\langle\langle u\oplus\tau, (\rho,\rho^t)(-k) \rangle\rangle_C $
proved below and $\langle\langle u\oplus\tau, (\rho,\rho^t)(-k)\oplus
k\rangle\rangle_C=\langle u, k\rangle=0$ since $U=K^\circ$.

Choose $u$, $u_i\in\Gamma(U)$  and
$\tau=(a,\theta)$, $\tau_i=(a_i,\theta_i)\in\Gamma(A\oplus T^*M)$, $i=1,2,3$.
We check $(1)-(3)$ in the definition of a Courant algebroid, 
in reverse order.
First we have 
\begin{align*}
&\lb u_1\oplus \tau_1,  u_2\oplus \tau_2\rb_C+\lb u_2\oplus \tau_2,  u_1\oplus \tau_1\rb_C\\
=&(\lb u_1, u_2\rb_\Delta+\lb u_2, u_1\rb_\Delta)\oplus 
([\tau_1,\tau_2]_d+[\tau_2,\tau_1]_d+(0,
\dr\langle\tau_1,u_2\rangle+\dr\langle \tau_2, u_1\rangle)\\
=&0\oplus (0, \dr(\langle\tau_1,\tau_2\rangle_d+\langle\tau_1,u_2\rangle+\langle \tau_2, u_1\rangle))
=\mathcal D\langle\langle u_1\oplus\tau_1, u_2\oplus\tau_2\rangle\rangle_C.
\end{align*}

We denote by $\Skew_\Delta(\nu_1,\nu_2)$
the anti-commutator $\lb \nu_1,\nu_2\rb_\Delta+\lb \nu_2, \nu_1\rb_\Delta$, $\nu_1,\nu_2\in\Gamma(TM\oplus A^*)$.
We check the equality $c(u_1\oplus\tau_1)\langle\langle
u_2\oplus\tau_2, u_3\oplus\tau_3\rangle\rangle_C =\langle\langle \lb
u_1\oplus \tau_1, u_2\oplus\tau_2 \rb_C,
u_3\oplus\tau_3\rangle\rangle_C +\langle\langle u_2\oplus\tau_2, \lb
u_1\oplus \tau_1,u_3\oplus\tau_3\rb_C\rangle\rangle_C$.  The bracket
$\lb u_1\oplus \tau_1, u_2\oplus \tau_2\rb$ equals $(\lb u_1,
u_2\rb_\Delta+\nabla_{\tau_1}^{\rm bas}u_2-\nabla_{\tau_2}^{\rm
  bas}u_1) \oplus([\tau_1,
\tau_2]_d+\Delta_{u_1}\tau_2-\Delta_{u_2}\tau_1+(0, \dr\langle \tau_1,
u_2\rangle))$. Using \eqref{basicLike_eqq}, we can replace
$[\tau_1,\tau_2]_d$ by $\nabla_{\tau_1}^{\rm
  bas}\tau_2-\Delta_{(\rho,\rho^t)\tau_2}\tau_1+(0,\dr\langle \tau_1,
(\rho,\rho^t)\tau_2\rangle)$ in this expression and we find that
$\langle\langle \lb u_1\oplus \tau_1, u_2\oplus \tau_2\rb,
u_3\oplus\tau_3\rangle\rangle_C$ equals 
\begin{align}\label{eq1}
&\langle \lb u_1, u_2\rb_\Delta+\nabla_{\tau_1}^{\rm
  bas}u_2-\nabla_{\tau_2}^{\rm bas}u_1, \tau_3\rangle
+X_3\langle \tau_1, u_2\rangle+ X_3\langle \tau_1,
(\rho,\rho^t)\tau_2\rangle\nonumber\\
&+\langle  \nabla_{\tau_1}^{\rm
  bas}\tau_2-\Delta_{(\rho,\rho^t)\tau_2}\tau_1+\Delta_{u_1}\tau_2-\Delta_{u_2}\tau_1,
u_3\rangle+\rho(a_3)\langle \tau_1, u_2\rangle\nonumber\\
&+ \rho(a_3)\langle \tau_1,
(\rho,\rho^t)\tau_2\rangle+\langle  \nabla_{\tau_1}^{\rm
  bas}\tau_2-\Delta_{(\rho,\rho^t)\tau_2}\tau_1+\Delta_{u_1}\tau_2-\Delta_{u_2}\tau_1,
(\rho,\rho^t)\tau_3\rangle.\nonumber
\end{align} 
Using \eqref{complicated_eqq} we replace 
$\langle -\nabla^{\rm
  bas}_{\tau_2}u_1+(\rho,\rho^t)\Delta_{u_1}\tau_2,\tau_3\rangle$ in this equation by
$\langle \nabla^{\rm bas}_{\tau_3}u_1,\tau_2\rangle+\langle \lb
u_1,(\rho,\rho^t)\tau_2\rb_\Delta,\tau_3\rangle$.

Now we sum $\langle\langle \lb u_1\oplus \tau_1, u_2\oplus \tau_2\rb,
  u_3\oplus\tau_3\rangle\rangle_C$ with $ \langle\langle \lb u_1\oplus
  \tau_1, u_3\oplus \tau_3\rb,
  u_2\oplus\tau_2\rangle\rangle_C$.
The terms $\langle \lb u_1,u_2\rb_\Delta, \tau_3\rangle$, $\langle \lb u_1,u_3\rb_\Delta, \tau_2\rangle$,
$\langle \Delta_{u_1}\tau_2,u_3\rangle$ and $\langle \Delta_{u_1}\tau_3, u_2\rangle$
add up to 
\[X_1\langle \langle u_2\oplus\tau_2, u_3\oplus\tau_3\rangle\rangle_C
-X_1\langle \tau_2,(\rho,\rho^t)\tau_3\rangle.
\]
The terms $X_3\langle\tau_1,u_2\rangle$ and
$-\langle\Delta_{u_3}\tau_1,u_2\rangle$ add up to $\langle \tau_1,\lb
u_2,u_3\rb_\Delta\rangle$ and the terms $X_2\langle\tau_1,u_3\rangle$
and $-\langle\Delta_{u_2}\tau_1,u_3\rangle$ add up to $\langle
\tau_1,\lb u_3,u_2\rb_\Delta\rangle$, which cancels out $\langle
\tau_1,\lb u_2,u_3\rb_\Delta\rangle$ since
$\lb\cdot\,,\cdot\rb_\Delta$ is skew-symmetric on sections of $U$. In a similar manner,
\begin{equation}\label{4_1}
X_3\langle
\tau_1,(\rho,\rho^t)\tau_2\rangle-\langle\Delta_{u_2}\tau_1,(\rho,\rho^t)\tau_3\rangle
+X_2\langle
\tau_1,(\rho,\rho^t)\tau_3\rangle-\langle\Delta_{u_3}\tau_1,(\rho,\rho^t)\tau_2\rangle
\end{equation}
is $\langle \tau_1, \lb u_3,(\rho,\rho^t)\tau_2\rb_\Delta+\lb
u_2,(\rho,\rho^t)\tau_3\rb_\Delta\rangle$, and 
\begin{equation}\label{4_2}
-\langle \Delta_{(\rho,\rho^t)\tau_2}\tau_1,u_3\rangle+\rho(a_3)\langle\tau_1, u_2\rangle
-\langle \Delta_{(\rho,\rho^t)\tau_3}\tau_1,u_2\rangle+\rho(a_2)\langle\tau_1, u_3\rangle
\end{equation}
is 
$\langle \tau_1, \lb (\rho,\rho^t)\tau_2, u_3\rb_\Delta+\lb
(\rho,\rho^t)\tau_3,u_2\rb_\Delta\rangle$. As a consequence, \eqref{4_1} and \eqref{4_2}
add up to
\[\langle \tau_1, \Skew_\Delta(u_2,(\rho,\rho^t)\tau_3)+\Skew_\Delta(u_3,(\rho,\rho^t)\tau_2)\rangle.
\] 
By \eqref{not_dual}, this adds up with 
\[\langle \nabla^{\rm bas}_{\tau_1}u_2,\tau_3\rangle+\langle\nabla^{\rm bas}_{\tau_1}\tau_2,u_3\rangle
+\langle \nabla^{\rm bas}_{\tau_1}u_3,\tau_2\rangle+\langle\nabla^{\rm bas}_{\tau_1}\tau_3,u_2\rangle
\]
to $\rho(a_1)\langle u_2,\tau_3\rangle+\rho(a_1)\langle
u_3,\tau_2\rangle$.  Similarly, using \eqref{not_dual} and \eqref{intertwine_bas}, the terms\linebreak
$\langle \nabla^{\rm bas}_{\tau_1}\tau_2,(\rho,\rho^t)\tau_3\rangle$
and $\langle \nabla^{\rm
  bas}_{\tau_1}\tau_3,(\rho,\rho^t)\tau_2\rangle$ add up to
\[\rho(a_1)\langle \tau_2,(\rho,\rho^t)\tau_3\rangle-\langle
\tau_1,\Skew_\Delta((\rho,\rho^t)\tau_2,(\rho,\rho^t)\tau_3)\rangle.\]
The terms $\rho(a_3)\langle\tau_1,(\rho,\rho^t)\tau_2\rangle$,
$\rho(a_2)\langle \tau_1,(\rho,\rho^t)\tau_3\rangle$,
$-\langle\Delta_{(\rho,\rho^t)\tau_2}\tau_1,(\rho,\rho^t)\tau_3\rangle$
and\linebreak
$-\langle\Delta_{(\rho,\rho^t)\tau_3}\tau_1,(\rho,\rho^t)\tau_2\rangle$ sum up
to $\langle\tau_1, \Skew_\Delta((\rho,\rho^t)\tau_2,(\rho,\rho^t)\tau_3)\rangle$.

Putting everything together, we get
\begin{equation*}
\begin{split}
  &c(u_1\oplus\tau_1)\langle \langle u_2\oplus\tau_2,
  u_3\oplus\tau_3\rangle\rangle_C-X_1\langle \tau_2, (\rho,\rho^t)\tau_3\rangle\\
&+\langle \nabla^{\rm bas}_{\tau_3}u_1,\tau_2\rangle+\langle \lb
u_1,(\rho,\rho^t)\tau_2\rb_\Delta,\tau_3\rangle +\langle \nabla^{\rm
  bas}_{\tau_2}u_1,\tau_3\rangle+\langle \lb
u_1,(\rho,\rho^t)\tau_3\rb_\Delta,\tau_2\rangle.
\end{split}
\end{equation*}
By \eqref{complicated_eqq} the third, fourth and fifth term add up to
$\langle\Delta_{u_1}\tau_2,(\rho,\rho^t)\tau_3\rangle$, which cancels out the second and the sixth.

\medskip

For the Jacobi identity, we check that
\begin{align*}
&\operatorname{Jac}_{\lb\cdot\,,\cdot\rb_\Delta}(u_1\oplus\tau_1,u_2\oplus\tau_2,u_3\oplus\tau_3)\\
&:=\lb \lb u_1\oplus\tau_1, u_2\oplus \tau_2\rb,
u_3\oplus \tau_3\rb+\lb u_2\oplus \tau_2, \lb u_1\oplus\tau_1, u_3\oplus \tau_3\rb
\rb\\
&\hspace*{5cm}-\lb u_1\oplus \tau_1, \lb u_2\oplus\tau_2,
u_3\oplus \tau_3\rb\rb=-(\rho,\rho^t)(k)\oplus k,
\end{align*}
with
$k:=\left[R_\Delta^{\rm bas}(a_1,a_2)u_3-R_\Delta(u_1,
  u_2)\tau_3\right]+\rm c.p.$

To see that $k$ is a section of $K$, recall from Theorem~\ref{morphic_Dirac_triples}
that since  $u_i=(X_i,\alpha_i)\in\Gamma(U)$, $i=1,2,3$, we have
$R_\Delta(u_1,u_2)\tau_3+{\rm c.p.}\in\Gamma(K)$.
By the same theorem, 
 we find $R_\Delta^{\rm bas}(a_1,a_2)u_3+{\rm c.p.}\in\Gamma(K)$.
We write $\tau_i=(a_i,\theta_i)$ for $i=1,2,3$. Since $\lb
u_1\oplus\tau_1, u_2\oplus \tau_2\rb$
is
\[\left(\lb u_1,u_2\rb_\Delta+\nabla_{\tau_1}^{\rm
        bas}u_2-\nabla_{\tau_2}^{\rm bas}u_1\right)\oplus\left(
      [\tau_1,\tau_2]_d+\Delta_{u_1}\tau_2-\Delta_{u_2}\tau_1
      +(0,\dr\langle \tau_1,u_2\rangle)\right),
\]
its bracket 
$\lb \lb u_1\oplus\tau_1, u_2\oplus \tau_2\rb, u_3\oplus \tau_3\rb$
with $u_3\oplus\tau_3$ is 
\begin{align*}
 &\Bigl(\lb \lb u_1,u_2\rb_\Delta,u_3\rb_\Delta
+\left\lb\nabla_{\tau_1}^{\rm
        bas}u_2-\nabla_{\tau_2}^{\rm bas}u_1, u_3\right\rb_\Delta+\nabla^{\rm bas}_{[\tau_1,\tau_2]_d+\Delta_{u_1}\tau_2-\Delta_{u_2}\tau_1
      +(0,\dr\langle \tau_1,u_2\rangle)}u_3\\
&\hspace*{4cm} -\nabla_{\tau_3}^{\rm bas}\left(\lb u_1,u_2\rb_\Delta+\nabla_{\tau_1}^{\rm
        bas}u_2-\nabla_{\tau_2}^{\rm bas}u_1\right)\Bigr)\\
&\oplus\Bigl([[\tau_1,\tau_2]_d,\tau_3]_d+[\Delta_{u_1}\tau_2-\Delta_{u_2}\tau_1, \tau_3]_d
+[(0,\dr\langle \tau_1,u_2\rangle),\tau_3]_d\\
& \quad+\Delta_{\lb u_1,u_2\rb_\Delta+\nabla_{\tau_1}^{\rm
        bas}u_2-\nabla_{\tau_2}^{\rm bas}u_1}\tau_3-\Delta_{u_3}([\tau_1,\tau_2]_d+\Delta_{u_1}\tau_2-\Delta_{u_2}\tau_1
      +(0,\dr\langle \tau_1,u_2\rangle))\\
&\quad+(0, \dr\langle [\tau_1,\tau_2]_d+\Delta_{u_1}\tau_2-\Delta_{u_2}\tau_1
      +(0,\dr\langle \tau_1,u_2\rangle), u_3\rangle)
\Bigr).
\end{align*}
Since $[(0,\dr\langle \tau_1,u_2\rangle),\tau_3]_d=0$
and \[\pr_A\left([\tau_1,\tau_2]_d+\Delta_{u_1}\tau_2-\Delta_{u_2}\tau_1
      +(0,\dr\langle \tau_1,u_2\rangle\right)
=[a_1,a_2]+\pr_A\left(\Omega_{u_1}a_2-\Omega_{u_2}a_1\right),\]
this is
\begin{align*}
  &\lb \lb u_1\oplus\tau_1, u_2\oplus \tau_2\rb, u_3\oplus \tau_3\rb\\
  =&\Bigl(\lb \lb u_1,u_2\rb_\Delta,u_3\rb_\Delta
+\left\lb\nabla_{\tau_1}^{\rm
        bas}u_2-\nabla_{a_2}^{\rm bas}u_1, u_3\right\rb_\Delta+\nabla^{\rm bas}_{[a_1,a_2]+\pr_A\left(\Omega_{u_1}a_2-\Omega_{u_2}a_1\right)}u_3\\
&\hspace*{4cm} -\nabla_{a_3}^{\rm bas}\left(\lb u_1,u_2\rb_\Delta+\nabla_{a_1}^{\rm
        bas}u_2-\nabla_{a_2}^{\rm bas}u_1\right)\Bigr)\\
&\oplus\Bigl([[\tau_1,\tau_2]_d,\tau_3]_d+[\Delta_{u_1}\tau_2-\Delta_{u_2}\tau_1, \tau_3]_d\\
& \quad+\Delta_{\lb u_1,u_2\rb_\Delta+\nabla_{a_1}^{\rm
        bas}u_2-\nabla_{a_2}^{\rm bas}u_1}\tau_3-\Delta_{u_3}([\tau_1,\tau_2]_d+\Delta_{u_1}\tau_2-\Delta_{u_2}\tau_1
      +(0,\dr\langle \tau_1,u_2\rangle))\\
&\quad+(0, \dr\langle [\tau_1,\tau_2]_d+\Delta_{u_1}\tau_2-\Delta_{u_2}\tau_1
      +(0,\dr\langle \tau_1,u_2\rangle), u_3\rangle)
\Bigr).
\end{align*}
In the same manner, we get
\begin{align*}
  &\lb  u_2\oplus\tau_2, \lb u_1\oplus \tau_1,  u_3\oplus \tau_3\rb\rb\\
 & =\Bigl(\lb  u_2,\lb u_1,u_3\rb_\Delta\rb_\Delta
+\left\lb u_2,\nabla_{a_1}^{\rm
        bas}u_3-\nabla_{a_3}^{\rm bas}u_1 \right\rb_\Delta\\
&\quad +\nabla_{a_2}^{\rm bas}\left(\lb u_1,u_3\rb_\Delta+\nabla_{a_1}^{\rm
        bas}u_3-\nabla_{a_3}^{\rm bas}u_1\right)
-\nabla^{\rm bas}_{[a_1,a_3]+\pr_A\left(\Omega_{u_1}a_3-\Omega_{u_3}a_1\right)}u_2\Bigr)\\
&\oplus\Bigl([\tau_2,[\tau_1, \tau_3]_d]_d+[ \tau_2, \Delta_{u_1}\tau_3-\Delta_{u_3}\tau_1]_d
+[\tau_2, (0, \dr\langle \tau_1,u_3\rangle)]_d\\
& \quad+\Delta_{u_2}([\tau_1,\tau_3]_d+\Delta_{u_1}\tau_3-\Delta_{u_3}\tau_1
      +(0,\dr\langle \tau_1,u_3\rangle))
-\Delta_{\lb u_1,u_3\rb_\Delta+\nabla_{a_1}^{\rm
        bas}u_3-\nabla_{a_3}^{\rm bas}u_1}\tau_2\\
&\quad+(0, \dr\langle \tau_2, \lb u_1, u_3\rb_\Delta+\nabla_{a_1}^{\rm bas}u_3-\nabla_{a_3}^{\rm bas}u_1\rangle)
\Bigr).
\end{align*}
We start by studying the $U$-part of $\operatorname{Jac}_{\lb\cdot\,,\cdot\rb_\Delta}(u_1\oplus\tau_1,u_2\oplus\tau_2,u_3\oplus\tau_3)$.
By the computations above, this equals
\begin{align*}
&\cancel{\lb \lb u_1,u_2\rb_\Delta,u_3\rb_\Delta}+\cancel{\lb  u_2,\lb u_1,u_3\rb_\Delta\rb_\Delta}-\cancel{\lb  u_1,\lb u_2,u_3\rb_\Delta\rb_\Delta}\\
&-R_{\nabla^{\rm bas}}(a_1,a_2)u_3+R_{\nabla^{\rm bas}}(a_1,a_3)u_2-R_{\nabla^{\rm bas}}(a_2,a_3)u_1\\
&+\left\lb\nabla_{a_1}^{\rm
        bas}u_2-\nabla_{a_2}^{\rm bas}u_1, u_3\right\rb_\Delta+\nabla^{\rm bas}_{\pr_A\left(\Omega_{u_1}a_2-\Omega_{u_2}a_1\right)}u_3
 -\nabla_{a_3}^{\rm bas}\lb u_1,u_2\rb_\Delta\\
&
+\left\lb u_2,\nabla_{a_1}^{\rm
        bas}u_3-\nabla_{a_3}^{\rm bas}u_1 \right\rb_\Delta+\nabla_{a_2}^{\rm bas}\lb u_1,u_3\rb_\Delta
-\nabla^{\rm bas}_{\pr_A\left(\Omega_{u_1}a_3-\Omega_{u_3}a_1\right)}u_2\\
&
-\left\lb u_1,\nabla_{a_2}^{\rm
        bas}u_3-\nabla_{a_3}^{\rm bas}u_2 \right\rb_\Delta-\nabla_{a_1}^{\rm bas}\lb u_2,u_3\rb_\Delta+\nabla^{\rm bas}_{\pr_A\left(\Omega_{u_2}a_3-\Omega_{u_3}a_2\right)}u_1
\end{align*}
Since $(U,\pr_{TM},\lb\cdot\,,\cdot\rb_\Delta)$
is a Lie algebroid, the first three terms vanish together.
Next note that since $(\nabla^{\rm bas}, \nabla^{\rm bas},
R_\Delta^{\rm bas})$ is a representation up to homotopy of $A$ on
$(\rho,\rho^t)\colon A\oplus T^*M\to TM\oplus A^*$,
the basic curvature satisfies
\begin{equation}
  R_{\nabla^{\rm bas}}(a_1,a_2)\nu=(\rho,\rho^t)\circ R^{\rm bas}_\Delta(a_1,a_2)\nu
\end{equation}
for all $a_1,a_2\in\Gamma(A)$, $\nu\in\Gamma(TM\oplus A^*)$.
Using this, \eqref{bialgebroid1},  (4) in Theorem
\ref{morphic_Dirac_triples}, the skew-symmetry of $\lb
\cdot\,,\cdot\rb_\Delta$ on $\Gamma(U)$, and the equality
$\pr_A\Omega_{u_1}a_2=\pr_A\Delta_{u_1}\tau_2$, we find easily that the $U$-part
of $\operatorname{Jac}_{\lb\cdot\,,\cdot\rb_\Delta}(u_1\oplus\tau_1,u_2\oplus\tau_2,u_3\oplus\tau_3)$
is 
\begin{align*}
&-(\rho,\rho^t)(R_{\Delta}^{\rm bas}(a_1,a_2)u_3+R_{\Delta}^{\rm
    bas}(a_3,a_1)u_2+R_{\Delta}^{\rm bas}(a_2,a_3)u_1)\\
&+(\rho,\rho^t)(R_\Delta(u_1,u_2)\tau_3+R_\Delta(u_2,u_3)\tau_1+R_\Delta(u_3,u_1)\tau_2).
\end{align*}
Now we look more carefully at the $A\oplus T^*M$-part of
$\operatorname{Jac}_{\lb\cdot\,,\cdot\rb_\Delta}(u_1\oplus\tau_1,u_2\oplus\tau_2,u_3\oplus\tau_3)$.
Sorting out the terms yields that this is
\begin{align*}
&\cancel{[[\tau_1,\tau_2]_d,\tau_3]_d}+\cancel{[\tau_2,[\tau_1, \tau_3]_d]_d}-\cancel{[\tau_1,[\tau_2, \tau_3]_d]_d}\\
&-\Delta_{u_1}[\tau_2,\tau_3]_d
+[\Delta_{u_1}\tau_2, \tau_3]_d+[ \tau_2, \Delta_{u_1}\tau_3]_d
+\Delta_{\nabla_{a_3}^{\rm bas}u_1}\tau_2
-\Delta_{\nabla_{a_2}^{\rm bas}u_1}\tau_3-(0, \dr\langle\tau_2,\nabla_{a_3}^{\rm bas}u_1\rangle)\\
&+\Delta_{u_2}[\tau_1,\tau_3]_d-[\Delta_{u_2}\tau_1, \tau_3]_d -[ \tau_1, \Delta_{u_2}\tau_3]_d
-\Delta_{\nabla_{a_3}^{\rm bas}u_2}\tau_1+\Delta_{\nabla_{a_1}^{\rm
        bas}u_2}\tau_3+(0, \dr\langle\tau_1,\nabla_{a_3}^{\rm bas}u_2\rangle)\\
&-\Delta_{u_3}[\tau_1,\tau_2]_d-[ \tau_2, \Delta_{u_3}\tau_1]_d+[ \tau_1, \Delta_{u_3}\tau_2]_d+\Delta_{\nabla_{a_2}^{\rm
        bas}u_3}\tau_1-\Delta_{\nabla_{a_1}^{\rm
        bas}u_3}\tau_2-(0, \dr\langle \tau_1,\nabla_{a_2}^{\rm bas}u_3\rangle)\\
&-R_\Delta(u_1,u_2)\tau_3-R_\Delta(u_2,u_3)\tau_1-R_\Delta(u_3,u_1)\tau_2\\
& -\Delta_{u_3}(0,\dr\langle \tau_1,u_2\rangle)+(0, \dr\langle [\tau_1,\tau_2]_d+\Delta_{u_1}\tau_2-\Delta_{u_2}\tau_1
      +(0,\dr\langle \tau_1,u_2\rangle), u_3\rangle)\\
&+[\tau_2, (0, \dr\langle \tau_1,u_3\rangle)]_d+\Delta_{u_2}(0,\dr\langle \tau_1,u_3\rangle)+(0, \dr\langle \tau_2, \lb u_1, u_3\rb_\Delta+\nabla_{a_1}^{\rm bas}u_3\rangle)\\
&-[\tau_1, (0, \dr\langle \tau_2,u_3\rangle)]_d-\Delta_{u_1}(0,\dr\langle \tau_2,u_3\rangle)-(0, \dr\langle \tau_1, \lb u_2, u_3\rb_\Delta\rangle).
\end{align*}
By the Jacobi identity for $[\cdot\,,\cdot]_d$
and three times \eqref{bialgebroid2}
together with 
$-[ \tau_2, \Delta_{u_3}\tau_1]_d=[\Delta_{u_3}\tau_1, \tau_2]_d-(0,\dr\langle (\rho,\rho^t)\Delta_{u_3}\tau_1,\tau_2\rangle)$, 
we get
 \begin{align*}
&R_\Delta^{\rm bas}(a_2,a_3)u_1-R_\Delta^{\rm bas}(a_1,a_3)u_2+R_\Delta^{\rm bas}(a_1,a_2)u_3\\
&-R_\Delta(u_1,u_2)\tau_3-R_\Delta(u_2,u_3)\tau_1-R_\Delta(u_3,u_1)\tau_2\\
& -\cancel{(0,\dr X_3\langle \tau_1,u_2\rangle)}+(0, \dr\langle [\tau_1,\tau_2]_d+\cancel{\Delta_{u_1}\tau_2}
-\cancel{\Delta_{u_2}\tau_1}, u_3\rangle)+ \cancel{(0, \dr X_3\langle \tau_1,u_2\rangle)}\\
&+(0, \dr\rho(a_2)\langle \tau_1,u_3\rangle)+\cancel{(0,\dr X_2\langle \tau_1,u_3\rangle)}
+(0, \dr\langle \tau_2, \cancel{\lb u_1, u_3\rb_\Delta}+\nabla_{a_1}^{\rm bas}u_3\rangle)\\
&- (0, \dr \rho(a_1)\langle \tau_2,u_3\rangle)-\cancel{(0,\dr X_1\langle \tau_2,u_3\rangle)}
-\cancel{(0, \dr\langle \tau_1, \lb u_2, u_3\rb_\Delta\rangle)}\\
&-(0,\dr\langle(\rho,\rho^t)\Delta_{u_3}\tau_1, \tau_2\rangle)\\
=\,&R_\Delta^{\rm bas}(a_2,a_3)u_1+R_\Delta^{\rm bas}(a_3,a_1)u_2+R_\Delta^{\rm bas}(a_1,a_2)u_3\\
&-R_\Delta(u_1,u_2)\tau_3-R_\Delta(u_2,u_3)\tau_1-R_\Delta(u_3,u_1)\tau_2+(0,\dr f),
\end{align*}
with (using \eqref{basicLike_eqq})
\begin{align*}
 f=&\langle\Delta_{(\rho,\rho^t)\tau_1}\tau_2-\nabla^{\rm bas}_{\tau_2}\tau_1,u_3\rangle+\rho(a_2)\langle \tau_1,u_3\rangle
+\langle \tau_2, \nabla_{a_1}^{\rm bas}u_3-(\rho,\rho^t)\Delta_{u_3}\tau_1\rangle\\
&-\rho(a_1)\langle \tau_2,u_3\rangle\\
=\,\,\,\,&-\langle\nabla^{\rm bas}_{\tau_2}\tau_1,u_3\rangle+\rho(a_2)\langle \tau_1,u_3\rangle
+\langle \tau_2, \nabla_{a_1}^{\rm bas}u_3-(\rho,\rho^t)\Delta_{u_3}\tau_1-\lb(\rho,\rho^t)\tau_1,u_3\rb_\Delta\rangle\\
\overset{\eqref{complicated_eqq}}
=&-\langle\nabla^{\rm bas}_{\tau_2}\tau_1,u_3\rangle+\rho(a_2)\langle \tau_1,u_3\rangle
-\langle\nabla_{\tau_2}^{\rm bas}u_3,\tau_1\rangle
-\langle \Skew_\Delta(u_3, (\rho,\rho^t)\tau_1),a_2\rangle\\
\overset{\eqref{not_dual}}
=&0.\qedhere
\end{align*}
\end{proof}

We have used Theorem~\ref{difficult_computation} for the Jacobi
identity. For completeness, we prove this theorem here.
\begin{proof}[Proof of Theorem~\ref{difficult_computation}]
  We start with the first equation. For simplicity, we write here
  $a$ for $(a,0)$, $\Delta_ua:=\Delta_u(a,0)$ and $\dr f:=(0,\dr f)$
  for $a\in\Gamma(A)$, $ f\in C^\infty(M)$ and $u\in\Gamma(U)$.
  First, we find that for $a_1,a_2\in\Gamma(A)$ and
  $u_1=(X_1,\alpha_1),u_2=(X_2,\alpha_2)\in\Gamma(U)$, the equation $0=\langle
  R_\Delta^{\rm bas}(a_1,a_2)u_1,u_2\rangle$ can be written
\begin{align*}
  0=&-\langle
  \Delta_{u_1}[a_1,a_2]-\dr\langle\alpha_1,[a_1,a_2]\rangle,u_2\rangle\\
&  +\rho(a_1)\langle\Delta_{u_1}a_2-\dr\langle\alpha_1,a_2\rangle,u_2\rangle-\langle\Delta_{u_1}a_2-\dr\langle\alpha_1,a_2\rangle,\ldr{a_1}u_2\rangle\\
  &-\rho(a_2)\langle\Delta_{u_1}a_1-\dr\langle\alpha_1,a_1\rangle,u_2\rangle+\langle\Delta_{u_1}a_1-\dr\langle\alpha_1,a_1\rangle,\ldr{a_2}u_2\rangle\\
  &+\langle \Delta_{(\rho,\rho^t)\Omega_{u_1}a_2+\ldr{a_2}u_1}a_1-\dr\langle(\rho,\rho^t)\Omega_{u_1}a_2+\ldr{a_2}u_1,a_1\rangle,u_2\rangle\\
  &-\langle \Delta_{(\rho,\rho^t)\Omega_{u_1}a_1+\ldr{a_1}u_1}a_2-\dr\langle(\rho,\rho^t)\Omega_{u_1}a_1+\ldr{a_1}u_1,a_2\rangle,u_2\rangle\\
  =&-\langle
  \Delta_{u_1}[a_1,a_2],u_2\rangle+X_2\langle\alpha_1,[a_1,a_2]\rangle\\
&+\rho(a_1)\langle\Delta_{u_1}a_2,u_2\rangle-\rho(a_1)X_2\langle\alpha_1,a_2\rangle -\langle\Delta_{u_1}a_2,\ldr{a_1}u_2\rangle+[\rho(a_1),X_2]\langle\alpha_1,a_2\rangle\\
&-\rho(a_2)\langle\Delta_{u_1}a_1,u_2\rangle+\rho(a_2)X_2\langle\alpha_1,a_1\rangle+\langle\Delta_{u_1}a_1,\ldr{a_2}u_2\rangle-[\rho(a_2),X_2]\langle\alpha_1,a_1\rangle\\
  &+\langle
  \Delta_{(\rho,\rho^t)\Omega_{u_1}a_2}a_1,u_2\rangle+[\rho(a_2),X_1]\langle
  a_1,\alpha_2\rangle -\langle
  a_1,\lb\ldr{a_2}u_1,u_2\rb_\Delta\rangle\\
& -X_2\langle(\rho,\rho^t)\Omega_{u_1}a_2+\ldr{a_2}u_1,a_1\rangle\\
  &-\langle
  \Delta_{(\rho,\rho^t)\Omega_{u_1}a_1}a_2,u_2\rangle-[\rho(a_1),X_1]\langle
  a_2,\alpha_2\rangle +\langle
  a_2,\lb\ldr{a_1}u_1,u_2\rb_\Delta\rangle\\
&  +X_2\langle(\rho,\rho^t)\Omega_{u_1}a_1+\ldr{a_1}u_1,a_2\rangle.
\end{align*}
We use 
\begin{align*}
&X_2\langle\alpha_1,[a_1,a_2]\rangle+\rho(a_2)X_2\langle\alpha_1,a_1\rangle-[\rho(a_2),X_2]\langle\alpha_1,a_1\rangle-X_2\langle\ldr{a_2}u_1,a_1\rangle\\
=\,&X_2(-\langle\alpha_1,[a_2,a_1]\rangle+\rho(a_2)\langle\alpha_1,a_1\rangle-\langle\ldr{a_2}\alpha_1,a_1\rangle)=0,
\end{align*}
similarly
\begin{align*}
-\rho(a_1)X_2\langle\alpha_1,a_2\rangle+[\rho(a_1),X_2]\langle\alpha_1,a_2\rangle+X_2\langle\ldr{a_1}u_1,a_2\rangle%
=-X_2\langle \alpha_1, [a_1,a_2]\rangle
\end{align*}
and, by \eqref{basicLike_eqq} and \eqref{not_dual},
\begin{align*}
  &\langle \Delta_{(\rho,\rho^t)\Omega_{u_1}a_2}a_1,u_2\rangle-X_2\langle(\rho,\rho^t)\Omega_{u_1}a_2,a_1\rangle\\
  =\,&\langle \nabla_{a_1}^{\rm bas}\Omega_{u_1}a_2+[\Omega_{u_1}a_2,a_1]_d, u_2\rangle-X_2\langle(\rho,\rho^t)\Omega_{u_1}a_2,a_1\rangle\\
  =\,&\rho(a_1)\langle\Omega_{u_1}a_2, u_2\rangle-\langle
       \Omega_{u_1}a_2, \nabla_{a_1}^{\rm bas}u_2\rangle\\
&-\langle\Skew_\Delta(u_2,(\rho,\rho^t)\Omega_{u_1}a_2),a_1\rangle-\langle\ldr{a_1}\Omega_{u_1}a_2, u_2\rangle\\
  =\,&-\langle \Omega_{u_1}a_2, (\rho,\rho^t)\Omega_{u_2}a_1\rangle-\langle\Skew_\Delta(u_2,(\rho,\rho^t)\Omega_{u_1}a_2),a_1\rangle\\
\end{align*}
to get
\begin{align*}
0=&-X_2\langle \alpha_1, [a_1,a_2]\rangle-X_1\langle\alpha_2, [a_1,a_2]\rangle+\langle [a_1,a_2], \lb u_1,u_2\rb_\Delta\rangle\\
&-\langle a_1,\lb\ldr{a_2}u_1,u_2\rb_\Delta\rangle
+\langle a_2,\lb\ldr{a_1}u_1,u_2\rb_\Delta\rangle\\
&+\rho(a_1)\langle\Delta_{u_1}a_2,u_2\rangle
-X_1\langle a_2,\ldr{a_1}u_2\rangle+\langle a_2, \lb u_1,\ldr{a_1}u_2\rb_\Delta\rangle\\
&-\rho(a_2)\langle\Delta_{u_1}a_1,u_2\rangle+X_1\langle a_1,\ldr{a_2}u_2\rangle-\langle a_1, \lb u_1,\ldr{a_2}u_2\rb_\Delta\rangle\\
&+[\rho(a_2),X_1]\langle a_1,\alpha_1\rangle-[\rho(a_1),X_1]\langle a_2,\alpha_2\rangle\\
&-\langle \Omega_{u_1}a_2, (\rho,\rho^t)\Omega_{u_2}a_1\rangle-\langle\Skew_\Delta(u_2,(\rho,\rho^t)\Omega_{u_1}a_2),a_1\rangle\\
&+\langle \Omega_{u_1}a_1, (\rho,\rho^t)\Omega_{u_2}a_2\rangle+\langle\Skew_\Delta(u_2,(\rho,\rho^t)\Omega_{u_1}a_1),a_2\rangle.
\end{align*}
But since 
\begin{align*}
&\rho(a_1)\langle\Delta_{u_1}a_2,u_2\rangle-[\rho(a_1),X_1]\langle a_2,\alpha_2\rangle
-X_1\langle a_2,\ldr{a_1}u_2\rangle\\
&=-\rho(a_1)\langle a_2,\lb u_1,u_2\rb_\Delta\rangle+X_1\langle[a_1,a_2],\alpha_2\rangle\\
&=-\langle [a_1,a_2],\lb u_1,u_2\rb_\Delta\rangle
-\langle a_2,\ldr{a_1}\lb u_1,u_2\rb_\Delta\rangle+X_1\langle[a_1,a_2],\alpha_2\rangle,
\end{align*}
we find
\begin{equation}\label{ugly_but_useful}
\begin{split}
  0=&-X_2\langle \alpha_1, [a_1,a_2]\rangle+X_1\langle\alpha_1,
  [a_1,a_2]\rangle-\langle
  [a_1,a_2], \lb u_1,u_2\rb_\Delta\rangle\\
  &-\langle a_1,\lb\ldr{a_2}u_1,u_2\rb_\Delta\rangle
  +\langle a_2,\lb\ldr{a_1}u_1,u_2\rb_\Delta\rangle+\langle a_2, \lb
  u_1,\ldr{a_1}u_2\rb_\Delta\rangle\\
  &-\langle a_1,
  \lb u_1,\ldr{a_2}u_2\rb_\Delta\rangle
  -\langle a_2, \ldr{a_1}\lb u_1,u_2\rb_\Delta\rangle+\langle a_1, \ldr{a_2}\lb u_1,u_2\rb_\Delta\rangle\\
  &-\langle \Omega_{u_1}a_2, (\rho,\rho^t)\Omega_{u_2}a_1\rangle-\langle\Skew_\Delta(u_2,(\rho,\rho^t)\Omega_{u_1}a_2),a_1\rangle\\
  &+\langle \Omega_{u_1}a_1,
  (\rho,\rho^t)\Omega_{u_2}a_2\rangle+\langle\Skew_\Delta(u_2,(\rho,\rho^t)\Omega_{u_1}a_1),a_2\rangle.
\end{split}
\end{equation}
Now, writing $\tau=(a,\theta)$, we compute:
\begin{align*}
&\nabla^{\rm bas}_\tau\lb u_1, u_2\rb_\Delta-\lb \nabla^{\rm bas}_\tau u_1,
u_2\rb_\Delta-\lb u_1, \nabla^{\rm bas}_\tau u_2\rb_\Delta+\nabla^{\rm
  bas}_{\Delta_{u_1}\tau}{u_2}-\nabla^{\rm bas}_{\Delta_{u_2}\tau}{u_1}\\
=\,\,\,\,&(\rho,\rho^t)\Omega_{\lb {u_1},{u_2}\rb_\Delta}a+\ldr{a}\lb {u_1}, {u_2}\rb_\Delta
-\lb(\rho,\rho^t)\Omega_{u_1}a+\ldr{a}{u_1}, {u_2}\rb_\Delta\\
&-\lb {u_1}, (\rho,\rho^t)\Omega_{u_2}a+\ldr{a}{u_2}\rb_\Delta+\nabla^{\rm bas}_{ \pr_A\Omega_{u_1}a}{u_2}-\nabla^{\rm bas}_{\pr_A\Omega_{u_2}a}{u_1}\\
\overset{\eqref{complicated_eqq}}=&(\rho,\rho^t)\bigl(\Delta_{\lb {u_1},{u_2}\rb_\Delta}a-\dr\langle a,\lb {u_1},{u_2}\rb_\Delta\rangle\bigr)
+\ldr{a}\lb {u_1}, {u_2}\rb_\Delta-\lb\ldr{a}{u_1}, {u_2}\rb_\Delta\\
&-\lb {u_1},
  \ldr{a}{u_2}\rb_\Delta+(\rho,\rho^t)\Delta_{{u_2}}\Omega_{u_1}a-\langle\nabla^{\rm
  bas}_\cdot {u_2},\Omega_{u_1}a\rangle
  -(\rho,\rho^t)\Delta_{{u_1}}\Omega_{u_2}a\\
&+\langle\nabla^{\rm bas}_\cdot {u_1},\Omega_{u_2}a\rangle-\Skew_\Delta((\rho,\rho^t)\Omega_{u_1}a,{u_2})\\
  =\,\,\,\,&-(\rho,\rho^t)R_\Delta({u_1},{u_2})a-\rho^t\dr\langle
             a,\lb {u_1},{u_2}\rb_\Delta\rangle+\ldr{a}\lb {u_1},
             {u_2}\rb_\Delta-\lb\ldr{a}{u_1}, {u_2}\rb_\Delta\\
&-\lb {u_1}, \ldr{a}{u_2}\rb_\Delta-\rho^t\dr(X_2\langle \alpha_1,a\rangle)-\langle\nabla^{\rm bas}_\cdot
  {u_2},\Omega_{u_1}a\rangle
  +\rho^t\dr (X_1\langle \alpha_2,a\rangle)\\
&+\langle\nabla^{\rm
  bas}_\cdot {u_1},\Omega_{u_2}a\rangle-\Skew_\Delta((\rho,\rho^t)\Omega_{u_1}a,{u_2}).
\end{align*}
We write the right-hand side of this equation
\[-(\rho,\rho^t)R_\Delta({u_1},u_2)a+\tau'
\]
with $\tau'\in\Gamma(TM\oplus A^*)$. Since 
$\pr_{TM}\bigl(\ldr{a}\lb u_1, u_2\rb_\Delta-\lb\ldr{a}u_1, u_2\rb_\Delta-\lb u_1, \ldr{a}u_2\rb_\Delta\bigr)$
is $[\rho(a),[X_1,X_2]]-[[\rho(a),X_1],X_2]-[X_1,[\rho(a),X_2]]$, which
vanishes, we observe that
$\pr_{TM}\tau'=0$. Hence, we just have to show that $\langle \tau', a'\rangle=0$
for any section $a'\in\Gamma(A)$.
We have 
\begin{align*}
\langle \tau', {a'}\rangle=\,&-\rho({a'})\langle a,\lb u_1,{u_2}\rb_\Delta\rangle+\langle\ldr{a}\lb u_1, {u_2}\rb_\Delta-\lb\ldr{a}u_1, {u_2}\rb_\Delta-\lb u_1, \ldr{a}{u_2}\rb_\Delta, {a'}\rangle\\
&-\rho({a'})X_2\langle \alpha_1,a\rangle-\langle\nabla^{\rm bas}_{a'} {u_2},\Omega_{u_1}a\rangle+\rho({a'})X_1\langle \alpha_2,a\rangle
+\langle\nabla^{\rm bas}_{a'} u_1,\Omega_{u_2}a\rangle\\
&-\langle\Skew_\Delta((\rho,\rho^t)\Omega_{u_1}a,{u_2}),{a'}\rangle\\
=\,&\langle [a,{a'}],\lb u_1,{u_2}\rb_\Delta\rangle-\langle
     a,\ldr{{a'}}\lb u_1,{u_2}\rb_\Delta\rangle\\
&+\langle\ldr{a}\lb u_1, {u_2}\rb_\Delta-\lb\ldr{a}u_1, {u_2}\rb_\Delta-\lb u_1, \ldr{a}{u_2}\rb_\Delta, {a'}\rangle\\
&-\rho({a'})X_2\langle \alpha_1,a\rangle-\langle(\rho,\rho^t)\Omega_{u_2}{a'}+\ldr{{a'}}{u_2},\Omega_{u_1}a\rangle\\
&+\rho({a'})X_1\langle \alpha_2,a\rangle+\langle(\rho,\rho^t)\Omega_{u_1}{a'}+\ldr{{a'}}u_1,\Omega_{u_2}a\rangle\\
&-\langle\Skew_\Delta((\rho,\rho^t)\Omega_{u_1}a,{u_2}),{a'}\rangle.
\end{align*}
Since the definition of $\Omega$ and the duality of $\Delta$ and $\lb\cdot\,,\cdot\rb_\Delta$ yield
\begin{align*}
\langle\ldr{{a'}}u_1,\Omega_{u_2}a\rangle+\rho({a'})X_1\langle \alpha_2,a\rangle
&=\langle\ldr{{a'}}u_1, \Delta_{u_2}a\rangle+X_1\rho({a'})\langle \alpha_2, a\rangle\\
&=X_2\langle\ldr{{a'}}u_1, a\rangle-\langle\lb u_2,\ldr{{a'}}u_1\rb_\Delta, a\rangle+X_1\rho({a'})\langle \alpha_2, a\rangle, 
\end{align*}
we get
\begin{align*}
\langle \tau', {a'}\rangle=\,&\langle [a,{a'}],\lb
                               {u_1},{u_2}\rb_\Delta\rangle-\langle
                               a,\ldr{{a'}}\lb
                               {u_1},{u_2}\rb_\Delta\rangle\\
&+\langle\ldr{a}\lb u_1, {u_2}\rb_\Delta-\lb\ldr{a}u_1, {u_2}\rb_\Delta-\lb u_1, \ldr{a}{u_2}\rb_\Delta, {a'}\rangle\\
&+\langle(\rho,\rho^t)\Omega_{u_1}{a'},\Omega_{u_2}a\rangle-\langle(\rho,\rho^t)\Omega_{u_2}{a'},\Omega_{u_1}a\rangle-\langle\Skew_\Delta((\rho,\rho^t)\Omega_{u_1}a,{u_2}),{a'}\rangle\\
&-X_1\langle\ldr{{a'}}\alpha_2, a\rangle+\langle\lb {u_1},\ldr{{a'}}{u_2}\rb_\Delta, a\rangle-X_2\rho({a'})\langle \alpha_1, a\rangle\\
&+X_2\langle\ldr{{a'}}\alpha_1, a\rangle-\langle\lb {u_2},\ldr{{a'}}{u_1}\rb_\Delta, a\rangle+X_1\rho({a'})\langle \alpha_2, a\rangle\\
=\,&\langle [a,{a'}],\lb {u_1},{u_2}\rb_\Delta\rangle-\langle a,\ldr{{a'}}\lb {u_1},{u_2}\rb_\Delta-\lb {u_1},\ldr{{a'}}{u_2}\rb_\Delta+\lb {u_2},\ldr{{a'}}{u_1}\rb_\Delta\rangle\\
&+\langle\ldr{a}\lb {u_1}, {u_2}\rb_\Delta-\lb\ldr{a}{u_1}, {u_2}\rb_\Delta-\lb {u_1}, \ldr{a}{u_2}\rb_\Delta, {a'}\rangle\\
&+\langle(\rho,\rho^t)\Omega_{u_1}{a'},\Omega_{u_2}a\rangle-\langle(\rho,\rho^t)\Omega_{u_2}{a'},\Omega_{u_1}a\rangle\\
&-\langle\Skew_\Delta((\rho,\rho^t)\Omega_{u_1}a,{u_2}),{a'}\rangle-X_1\langle\alpha_2,[a,{a'}]\rangle+X_2\langle\alpha_1,[a,{a'}]\rangle.
\end{align*}
We have 
\[-\langle a,\lb {u_2},\ldr{{a'}}{u_1}\rb_\Delta\rangle=\langle a,\lb \ldr{{a'}}{u_1},
{u_2}\rb_\Delta\rangle-\langle\Skew_\Delta(\ldr{{a'}}{u_1}, {u_2}),a\rangle
\]
and 
since $\nabla^{\rm bas}_{a'}{u_1}\in\Gamma(U)$ by hypothesis, we find 
\[-\Skew_\Delta(\ldr{{a'}}{u_1},
{u_2})=-\Skew_\Delta(\ldr{{a'}}{u_1}-\nabla^{\rm
  bas}_{a'}{u_1}, {u_2})=\Skew_\Delta((\rho,\rho^t)\Omega_{u_1}{a'}, {u_2}).
\]
This shows that $\langle \tau',{a'}\rangle$ is equal to the right-hand side
of \eqref{ugly_but_useful} (with $a=a_1$, $a'=a_2$), which vanishes.

\bigskip

For the second equation we write $u=(X,\alpha)$ and $\tau_i=(a_i,\theta_i)$ for $i=1,2$: 
\begin{align*}
  &\Delta_u[\tau_1,\tau_2]_d-[\Delta_u\tau_1,\tau_2]_d-[\tau_1,\Delta_u\tau_2]_d
  +\Delta_{\nabla_{a_1}^{\rm bas}u}\tau_2-\Delta_{\nabla_{a_2}^{\rm
      bas}u}\tau_1+(0,\dr\langle\tau_1,\nabla^{\rm bas}_{a_2}u\rangle)\\
  =\,&\Omega_u[a_1,a_2]+(0,\ldr{X}(\ldr{\rho(a_1)}\theta_2-\ip{\rho(a_2)}\dr\theta_1)+\dr\langle\alpha,[a_1,a_2]\rangle)\\
  &-[\Omega_ua_1+(0,\ldr{X}\theta_1+\dr\langle\alpha,a_1\rangle),\tau_2]_d
  -[\tau_1, \Omega_ua_2+(0,\ldr{X}\theta_2+\dr\langle\alpha,a_2\rangle)]_d\\
  &+\Omega_{\nabla_{a_1}^{\rm bas}u}a_2+(0,\ldr{\pr_{TM}(\nabla_{a_1}^{\rm bas}u)}\theta_2+\dr\langle \nabla_{a_1}^{\rm bas}u, (a_2,0)\rangle)\\
  &-\Omega_{\nabla_{a_2}^{\rm bas}u}a_1-(0,\ldr{\pr_{TM}(\nabla_{a_2}^{\rm bas}u)}\theta_1+\dr\langle \nabla_{a_2}^{\rm bas}u, (a_1,0)\rangle)+(0,\dr\langle\tau_1,\nabla^{\rm bas}_{a_2}u\rangle)\\
  =\,&-R_\Delta^{\rm bas}(a_1,a_2)u
  +(0,\ldr{X}(\ldr{\rho(a_1)}\theta_2-\ip{\rho(a_2)}\dr\theta_1)+\dr\langle\alpha,[a_1,a_2]\rangle)\\
  &+(0, \ip{\rho(a_2)}\dr\ldr{X}\theta_1-\ldr{\rho\circ\pr_A\Omega_ua_1}\theta_2-\dr\langle (\rho,\rho^t)\Omega_ua_1,(a_2,0)\rangle)\\
  &-(0,\ldr{\rho(a_1)}\ldr{X}\theta_2+\ldr{\rho(a_1)}\dr\langle\alpha,a_2\rangle+\ip{\rho\circ\pr_A\Omega_ua_2}\dr\theta_1)\\
  &+(0,\ldr{\pr_{TM}(\nabla_{a_1}^{\rm bas}u)}\theta_2+\dr\langle \nabla_{a_1}^{\rm bas}u, (a_2,0)\rangle)\\
  &-(0,\ldr{\pr_{TM}(\nabla_{a_2}^{\rm bas}u)}\theta_1+\dr\langle
  \nabla_{a_2}^{\rm bas}u,
  (a_1,0)\rangle)+(0,\dr\langle\tau_1,\nabla^{\rm bas}_{a_2}u\rangle).
\end{align*}
But we have 
\begin{align*}
  &\ldr{X}\ldr{\rho(a_1)}\theta_2-\ldr{\rho\circ\pr_A\Omega_ua_1}\theta_2-\ldr{\rho(a_1)}\ldr{X}\theta_2+\ldr{\pr_{TM}(\nabla_{a_1}^{\rm bas}u)}\theta_2\\
  =\,&\ldr{[X,\rho(a_1)]-\rho\circ\pr_A\Omega_ua_1+\pr_{TM}(\nabla_{a_1}^{\rm
      bas}u)}\theta_2=0
\end{align*}
since $\pr_{TM}(\nabla_{a_1}^{\rm bas}u)=[\rho(a_1),X]+\rho\circ\pr_A\Omega_ua_1$,
and in the same manner
\begin{align*}
&-\ldr{X}\ip{\rho(a_2)}\dr\theta_1+\ip{\rho(a_2)}\dr\ldr{X}\theta_1+\ip{\rho\circ\pr_A\Omega_ua_2}\dr\theta_1-\ldr{\pr_{TM}(\nabla_{a_2}^{\rm bas}u)}\theta_1\\
&\qquad +\dr\langle \nabla_{a_2}^{\rm bas}u, (a_1,0)\rangle-\dr\langle\tau_1,\nabla^{\rm bas}_{a_2}u\rangle\\
=\,&\ip{[\rho(a_2),X]+\rho\circ\pr_A\Omega_ua_2}\dr\theta_1-\ldr{\pr_{TM}(\nabla_{a_2}^{\rm bas}u)}\theta_1-\dr\langle(0,\theta_1),\nabla^{\rm bas}_{a_2}u\rangle=0.
\end{align*}
Since \[-\langle (\rho,\rho^t)\Omega_ua_1,(a_2,0)\rangle+\langle \nabla_{a_1}^{\rm bas}u, (a_2,0)\rangle=
\langle \ldr{a_1}u,(a_2,0)\rangle=\langle \ldr{a_1}\alpha,a_2\rangle\]
and 
\[\ldr{\rho(a_1)}\dr\langle\alpha,a_2\rangle=\dr(\rho(a_1)\langle\alpha,a_2\rangle)=\dr(\langle\ldr{a_1}\alpha,a_2\rangle+\langle\alpha, [a_1,a_2]\rangle),
\]
the remaining sum 
\begin{align*}
\dr\langle\alpha,[a_1,a_2]\rangle-\dr\langle (\rho,\rho^t)\Omega_ua_1,(a_2,0)\rangle-\ldr{\rho(a_1)}\dr\langle\alpha,a_2\rangle+\dr\langle \nabla_{a_1}^{\rm bas}u, (a_2,0)\rangle
\end{align*}
vanishes as well.
\end{proof}

\def\cprime{$'$} \def\polhk#1{\setbox0=\hbox{#1}{\ooalign{\hidewidth
  \lower1.5ex\hbox{`}\hidewidth\crcr\unhbox0}}} \def\cprime{$'$}
  \def\cprime{$'$}
\providecommand{\bysame}{\leavevmode\hbox to3em{\hrulefill}\thinspace}
\providecommand{\MR}{\relax\ifhmode\unskip\space\fi MR }
\providecommand{\MRhref}[2]{%
  \href{http://www.ams.org/mathscinet-getitem?mr=#1}{#2}
}
\providecommand{\href}[2]{#2}

\end{document}